\documentclass[11pt]{article}

\RequirePackage[OT1]{fontenc}
\RequirePackage{amsthm,amsmath}
\RequirePackage[numbers]{natbib}
\usepackage{authblk}
\RequirePackage[colorlinks,citecolor=blue,urlcolor=blue]{hyperref}
\usepackage{graphicx}
\usepackage{amsfonts}
\usepackage{subfig}
\usepackage[T1]{fontenc}
\usepackage[utf8]{inputenc}
\usepackage[utf8]{inputenc}
\usepackage{xcolor}
\usepackage{natbib}
\usepackage{hyperref}

\newcommand \bl{\bar{L}}
\newcommand \E{\mathbb{E}}
\newcommand \PP{\mathbb{P}}

\newcommand \convp{\overset{p}{\to}}
\newcommand \convas{\overset{a.s.}{\to}}
\newtheorem{theorem}{Theorem}
\newtheorem{lemma}{Lemma}
\newtheorem{prop}{Proposition}
\newtheorem{cor}{Corollary}

\numberwithin{equation}{section}
\theoremstyle{plain}

\title{Projective, Sparse, and Learnable Latent Position Network Models}
\author{Neil A. Spencer and Cosma Rohilla Shalizi}
\affil{University of Connecticut and Carnegie Mellon University}
\date{}

\begin{document}
\maketitle

\vspace{-1cm}
\begin{abstract}

When modeling network data using a latent position model, it is typical to assume that the nodes' positions are independently and identically distributed. However, this assumption implies the average node degree grows linearly with the number of nodes, which is inappropriate when the graph is thought to be sparse. We propose an alternative assumption---that the latent positions are generated according to a Poisson point process---and show that it is compatible with various levels of sparsity. Unlike other notions of sparse latent position models in the literature, our framework also defines a projective sequence of probability models, thus ensuring consistency of statistical inference across networks of different sizes. We establish conditions for consistent estimation of the latent positions, and compare our results to existing frameworks for modeling sparse networks. 
\end{abstract}

\section{Introduction}

Network data consist of relational information between entities, such as friendships between people or interactions between cell proteins.
Often, these data take the form of binary measurements on dyads, indicating the presence or absence of a relationship between entities. 
Such network data can be modeled as a stochastic graph, with each individual dyad being a random edge. Stochastic graph models have been an active area of research for over fifty years across physics, sociology, mathematics, statistics, computer science, and other disciplines \citep{MEJN-on-network-structure-and-function}.

Many leading stochastic graph models assume that the inhomogeneity in
connection patterns across nodes is explained by node-level latent variables. The most
tractable version of this assumption is that the dyads are conditionally independent given the latent variables. In this article, we focus on a subclass of these conditionally independent dyad models---the distance-based latent position network model (LPM) of \citet{Hoff-Raftery-Handcock}. 

In LPMs, each node is assumed to have a latent position in a continuous space. The edges follow independent Bernoulli distributions with probabilities given by a decreasing function of the distance between the nodes' latent positions. By the triangle inequality, LPMs exhibit edge transitivity; friends of friends are more likely to be friends. When the latent space is assumed to be $\mathbb{R}^2$ or $\mathbb{R}^3$, the inferred latent positions can provide an embedding with which to visualize and interpret the network.

Recently, there has been an effort to classify stochastic graph models into general
unified frameworks. One notable success story has been that of the graphon for
exchangeable networks \citep{Diaconis-Janson-graph-limits}. The graphon characterizes
 all stochastic graphs invariant under isomorphism as latent variable models. LPMs can be placed within the graphon
 framework by assuming the latent positions are random effects drawn independently from the same
(possibly unknown) probability distribution. However, graphons can be inappropriate for some modeling tasks, due to their asymptotic properties. 

The typical asymptotic regime for statistical theory of network models considers the number of nodes growing to infinity in a single graph. Implicitly, this approach requires the network model to define a distribution over a sequence of increasingly sized graphs. There are several natural questions to ask about this sequence. Prominent questions include:

\begin{enumerate}
 \item At what rate does the number of edges in these graphs grow? 
\item Is the model's behavior consistent across networks of different sizes? 
\item Can one eventually learn the model's parameters as the graph grows? 
\end{enumerate}

For all non-trivial\footnote{The only exception is an empty graph, for which all edges are absent with probability one. }
 models falling within the graphon framework, the answer to question 1 is the same; the expected number of edges grows quadratically with the number of nodes \citep{Orbanz-Roy-Bayesian}. Such sequences of graphs---in which the average degree grows linearly---are called \emph{dense}. In contrast, many real-world networks are thought to have sub-linear average degree growth. This property is known as \emph{sparsity} \citep[Chapter 6.9]{MEJN-on-networks}). 
 
 For sparse graphs, graphon models are unsuitable. Accordingly, recent years have seen an effort to develop sparse graph models that preserve the advantages of graphons. In particular, the sparse graphon framework \citep{Bollobas-Janson-Riordan-Phase, Borgs-Chayes-Zhao-sparse-graph-convergence} and the graphex framework \citep{Caron-Fox-sparse-graphs, Veitch-Roy-graphex-class, Borgs-Chayes-Holden-sparse-exchangeable} both provide straightforward ways to modify network models from the dense regime to accommodate sparsity. 

In this article, we add to the sparse graph literature by formulating a new sparse LPM. We target three criteria: \emph{sparsity} (\S \ref{sec:sparsity}), \emph{projectivity} (\S \ref{sec:projectivity}) and \emph{learnablity} (\S \ref{learnabilitypre}). Projectivity of a model ensures consistency of the distributions it assigns to graphs of different sizes, and learnability ensures consistent estimation of the latent positions as the number of nodes grows. 

As we outline in Section~\ref{sec:remarks}, the existing methods for sparsifying graphons of \citet{Borgs-Chayes-Zhao-sparse-graph-convergence} and \citet{Veitch-Roy-graphex-class} do not satisfy these criteria; they either violate projectivity or make it difficult to establish learnability. We thus take a more specialized approach to develop our sparse LPMs, turning to non-exchangeable network models for inspiration. Specifically, our new LPM framework extends the Poisson random connection model \citep{Meester-Roy-continuum-percolation}---a specialized LPM framework in which the nodes' latent positions are generated according to a Poisson process. We modify the observation window approach proposed by \citet{Krioukov-Ostilli-Duality-Networks} to allow our LPMs to exhibit arbitrary levels of sparsity without sacrificing projectivity. 

To obtain learnability results for our LPM framework, we develop and modify a combination of results related to low rank matrix estimation \citep{Davenport-et-al-1-bit}, the Davis-Kahan Theorem \citep{Yu-Wang-Samworth-davis-kahan}, and eigenvalues of random Euclidean distance matrices. Our proof strategy culminates in a concentration inequality for a restricted maximum likelihood estimator of the latent positions that applies to wide a variety of LPMs, providing a straightforward sufficient conditions for LPM learnability.

The remainder of this article is organized as follows.
Section \ref{sec:background} defines sparsity (\S \ref{sec:sparsity}) and projectivity (\S \ref{sec:projectivity}) for graph
sequences. It also defines the LPM, establishing sparsity and projectivity results for its exchangeable (\S \ref{elsm}) and random connection model (\S \ref{prc}) formulations. Section~\ref{sec:sparse-latent-space} describes our new framework for modeling
projective sparse LPMs, and includes results that demonstrate that the resultant graph sequences are projective and sparse. Section \ref{learnability} defines learnability of latent position models, and provides conditions under which sparse latent position models are learnable. Finally, Section~\ref{sec:remarks} elaborates on connections between our approach, sparse graphon-based LPMs, and the graphex framework. It also includes a discussion of the limitations of our work. All proofs are deferred to Appendix~\ref{A}. 

\section{BACKGROUND}\label{sec:background}

\subsection{Sparsity}
\label{sec:sparsity}

Let $(Y^n)_{n=1,\ldots, \infty}$ be a sequence of increasingly sized ($n \times n$) random adjacency matrices associated with a sequence of increasingly sized simple undirected random graphs (on $n$ nodes). Here, each entry $Y^n_{ij}$ indicates the presence of an edge between nodes $i$ and $j$ for a graph on $n$ nodes. 

We say the sequence of stochastic graph models defined by $(Y^n)_{n=1,\ldots, \infty}$ is \emph{sparse in expectation} if
\begin{align}
\lim_{n \rightarrow \infty} \mathbb{E}\left( \frac{\sum_{i=1}^n \sum_{j=1}^n Y^n_{ij}}{n^2}\right) = 0.
\end{align}
In other words, a sequence of graphs is sparse in expectation if the expected number of edges scales sub-quadratically in the number of nodes. 

Recall that a node's degree is defined as the number of nodes to which it is adjacent. Sparsity in expectation is equivalent to the expected average node degree growing sub-linearly. If instead the average degree grows linearly, we say the graph is \emph{dense} in expectation. 

In this article, we are also interested in distinguishing between degrees of sparsity. We say that a graph is \emph{$e(n)$-sparse in expectation} if
\begin{align}
\lim_{n \rightarrow \infty} \mathbb{E} \left(\frac{\sum_{i=1}^n \sum_{j=1}^n Y_{ij}}{e(n)}\right) = C
\end{align}
for some constant $C \in \mathbb{R}_+$. That is, the number of edges scales $\Theta(e(n))$. A dense graph could also be called $n^2$-sparse in expectation.

Note that sparsity and $e(n)$-sparsity are asymptotic properties of graphs, defined for increasing sequences of graphs but not for finite realizations. These definitions differ from the informal use of ``sparse graph'' to refer to a single graph with few edges. It also differs from the definition of sparsity for weighted graphs used in \citet{Rastelli-sparse-weighted-networks}. In practice, we typically observe a single finite realization of a graph, but the notion of sparsity remains useful because many network models naturally define a sequence of networks.

\subsection{Projectivity}
\label{sec:projectivity}

Let $(\mathbb{P}^n)_{n=1\ldots \infty}$ denote the probability distributions corresponding to a growing sequence of random adjacency matrices $(Y^n)_{n=1,\ldots, \infty}$ for a sequence of graphs. We say that the sequence $(\mathbb{P}^n)_{n=1\ldots \infty}$ is \emph{projective} if, for any $n_1 < n_2$, the distribution over adjacency matrices induced by $\mathbb{P}^{n_1}$ is equivalent to the distribution over $n_1 \times n_1$ sub-matrices induced by the leading $n_1$ rows and columns of an adjacency matrix following $\mathbb{P}^{n_2}$. That is, $(\mathbb{P}^n)_{n=1,\ldots, \infty}$ is projective if for any $y \in \left\{0,1\right\}^{n_1 \times n_1}$,
\begin{align}
\mathbb{P}^{n_1}(Y^{n_1} = y) &= \mathbb{P}^{n_2}(Y^{n_2} \in X),
\end{align}
where $X = \left\{x \in \left\{0,1\right\}^{n_2 \times n_2}: x_{ij} = y_{ij} \text{ if } 1 \leq i,j \leq n_1 \right\}$.

Projectivity ensures a notion of consistency between networks of different sizes, provided that they are generated from the same model class. This property is particularly useful for problems of superpopulation inference \citep{DAmour-Airoldi-misspecification}, such as testing whether separate networks were drawn from the same population, predicting the values of dyads associated with a new node, or pooling together estimates from separate networks in a hierarchical model. Such problems require that parameter inferences be comparable across differently sized graphs. Without projectivity, it is unclear how to make comparisons without additional assumptions.

Projectivity has thus received considerable attention recently in the networks literature \citep{Snijders-marginalization-for-ERGMs, your-favorite-ergm-sucks, Crane-Dempsey-Framework-Network-Modelling, Schweinberger-et-al-Finite-Population-Graphs, Kartun-Giles-et-al-Sparse-Power-Law}. Our definition of projectivity departs from others in the literature in that it depends on a specific ordering of the nodes. Other definitions require consistency under subsampling of any $n_1$ nodes, not just the first $n_1$ nodes. The two definitions coincide when exchangeability is assumed, but differ otherwise.

\subsection{Latent Position Network Models}

The notion that entities in networks possess latent positions has a long
history in the social science literature. The idea of a ``social space'' that
influences the social interactions of individuals traces back to at least the
seventeenth century \citep[p.\ 3]{Sorokin-social-mobility}. A thorough history
of the notions of social space and social distance as they pertain to social
networks is provided in \citet{McFarland-Brown-social-distance}.

In the statistical network modeling literature, assigning continuous latent
positions to nodes dates back to the 1970s, in which
multi-dimensional scaling was used to summarize similarities between nodes
in the data \citep[p.\ 385]{Wasserman-Faust}. However, it was not until
\citet{Hoff-Raftery-Handcock} that the modern notion of latent continuous positions were used to define a probabilistic model for stochastic
graphs in the statistics literature. In this article, we focus on this probabilistic formulation, with our definition of latent position
models (LPMs) following that of the distance model of \citet{Hoff-Raftery-Handcock}. 

Consider a binary graph on $n$ nodes. The LPM is characterized by each node $i$ of the network possessing a latent
position $Z_i$ in a metric space $(S,\rho)$. Conditional on these latent positions,
the edges are drawn as independent Bernoulli random variables following
\begin{align}
\mathbb{P}(Y_{ij} = 1|Z_i, Z_j) &= K(\rho(Z_i, Z_j)).
\end{align}
Here, $K:\mathbb{R}_+ \rightarrow [0,1]$ is known as the link probability function; it captures the dependency 
of edge probabilities on the latent inter-node distances. For the majority of this article, we assume $K$ is independent of $n$ (\S \ref{sparsegraphon} is an exception). Furthermore, we focus on link probability functions that smoothly decrease with distance and are integrable on the real line, such as expit($-\rho^2$), $\exp{(-\rho^2)}$ and $(1 + \rho^2)^{-1}$. Though the general formulation of the LPM in \citet{Hoff-Raftery-Handcock} allows for dyad-specific covariates to influence connectivity, our exposition assumes that no such covariates are available. We have done this for purposes of clarity; our framework does not specifically exclude them.

\subsection{Exchangeable Latent Position Network Models}\label{elsm}

Originally, \citet{Hoff-Raftery-Handcock} proposed modeling the nodes' latent positions as independent and identically distributed random effects
 drawn from a distribution $f$ of known parametric form. This approach remains popular in practice today, with $S$ assumed to be a low-dimensional Euclidean space $\mathbb{R}^d$ and $f$ typically assumed to be multivariate Gaussian or a mixture of multivariate Gaussians \citep{Hancock-Raftery-Tantrum-latent-cluster-network}.  We refer to this class of models as \emph{exchangeable LPMs} because they assume the nodes are infinitely exchangeable. Exchangeable latent position network models are projective, but must be dense in expectation.

\begin{prop}\label{exprojective}
Exchangeable latent position network models define a projective sequence of models.
\end{prop}
\begin{proof}
Provided in \S \ref{exprojectiveproof}.
\end{proof}

\begin{prop}\label{exsparse}
Exchangeable latent position network models define dense in expectation graph sequences.
\end{prop}
\begin{proof}
Provided in \S \ref{exsparseproof}.
\end{proof}

Consequently, LPMs with exchangeable latent positions cannot be sparse. To develop sparse LPMs, we must consider alternative assumptions. 

\subsection{Poisson Random Connection Model}\label{prc}

Instead of the latent positions being generated independently from a distribution over $S$, we can treat them as drawn according to a point process over $S$. This approach---known as the random connection model---has been well-studied in the context of percolation theory \citep{Meester-Roy-continuum-percolation}. Most of this focus has been on random geometric graphs \citep{Penrose-random-geometric-graphs}, a version of a LPMs for which K is an indicator function of the distance (i.e. $K(\rho(Z_i, Z_j)) \propto I(\rho(Z_i, Z_j) < \epsilon)$). Here, we instead study the random connection model as a statistical model, focusing the case where $K$ is a smoothly decaying and integrable function. 

In particular, we consider the \emph{Poisson} random connection model \citep{Gilbert-Random-Plane, Penrose-on-continuum-percolation}, for which the point process is assumed to be a homogeneous Poisson process \citep{Kingman-Poisson} over $S \subseteq \mathbb{R}^d$. Because Poisson random connection models on finite-measure $S$ are equivalent to exchangeable LPMs, the interesting cases occur when $S$ has infinite measure, such as $\mathbb{R}^d$. In these cases, the expected number of points is almost-surely infinite, resulting in an infinite number of nodes. 

These infinite graphs can be converted into a growing sequence of finite graphs via the following procedure. Let $G$ denote an infinite graph generated according to a Poisson random connection model on $S$. Let
\begin{align}
S_1 \subset S_2 \subset \cdots \subset S_n \subset \cdots \subset S
\end{align}
 denote a nested sequence of finitely-sized \emph{observation windows} in $S$. For each $S_i$, define $G_i$ to be the subgraph of $G$ induced by keeping only those nodes with latent positions in $S_i$. Because these positions form a Poisson process, each $G_i$ consists of a Poisson distributed number of nodes with mean given by the size of $S_i$. Each $G_i$ is thus almost-surely finite, and the sequence of graphs $(G_i)_{i =1, \ldots \infty}$ contains a stochastically increasing number of nodes. 

For many choices of $S$, such as $\mathbb{R}^d$, this approach straightforwardly extends to a continuum of graphs by considering a continuum of nested observation windows of $(S_t)_{t \in \mathbb{R}_+}$. In such cases, the number of nodes follows a continuous-time stochastic process, stochastically increasing in $t$.  

As far as we are aware, the above approach was first proposed by \citet{Krioukov-Ostilli-Duality-Networks}
in the context of defining a growing sequence of geometric random graphs. Their exposition concentrated on a one-dimensional example with $S = \mathbb{R}_+$ and observation windows given by $S_t = [0,t]$. For this example, one would expect to observe $n$ nodes if $t = n$, with the total number of nodes for a given $t$ being random. As noted by \citet{Krioukov-Ostilli-Duality-Networks}, the formulation can be altered to ensure that $n$ nodes are observed by treating $n$ as fixed and treating the window size $t_n$ as the random quantity. Here, $t_n$ it equal to the smallest window width such that $[0, t_n]$ contains exactly $n$ points.  These two viewpoints (random window size and random number of nodes) are complementary for analyzing the same underlying process. 

Under the appropriate conditions, the one-dimensional Poisson random connection model results in networks which are $n$-sparse in expectation. We formalize this notion as Proposition~\ref{sparseRCM}. The finite window approach approach also defines a projective sequence of models, as stated in Proposition~\ref{projectiveRCM}.

\begin{prop}\label{sparseRCM}
  For a Poisson random connection model on $\mathbb{R}_+$ with an integrable link
  probability function, the graph sequence resulting from the finite window approach is $n$-sparse in expectation. 
  \end{prop}
\begin{proof}
Provided in \S \ref{sparseRCMproof}
\end{proof}

\begin{prop}\label{projectiveRCM}
 Consider a Poisson random connection model on $\mathbb{R}_+$ with link
  probability function $K$.
  Then, the graph sequence resulting from the finite window approach is projective.
\end{prop}
\begin{proof}
Provided in \S \ref{projectiveRCMproof}.
\end{proof}

These results indicate that the Poisson random connection model restricted to observation windows
is capable of defining a sparse graph sequences, but only for a specific sparsity level if the link probability function is integrable. 
For our new framework, we extend this observation window approach to higher dimensional $S$. By including an auxiliary dimension, we
achieve all rates between $n$-sparsity and $n^2$-sparsity (density) in expectation.

\section{NEW FRAMEWORK}
\label{sec:sparse-latent-space}

When working in a one-dimensional Euclidean latent space $S = \mathbb{R}_+$, the observation window approach for the Poisson random connection model is straightforward---the width of the window grows linearly with $t$, with nodes arriving as the window grows. As shown in Proposition~\ref{sparseRCM}, this process results in graph sequences which are $n$-sparse in expectation whenever $K$ is integrable. However, extending to $d$ dimensions ($\mathbb{R}^d$) provides freedom in defining how the window grows; different dimensions of the window can be grown at different rates.

We exploit this extra flexibility to develop our new sparse LPM model. Specifically, through the inclusion of an auxiliary dimension---an additional latent space coordinate which influences when a node becomes visible without influencing its connection probabilities---we can control the level of sparsity of the graph by trading off how quickly we grow the window in the auxiliary dimension versus the others.

In this section, we formalize this auxiliary dimension approach, showing that it allows us to develop a new LPM framework for which the level of sparsity can be controlled while maintaining projectivity. Our exposition consists of two parts: first, we present the framework in the context of a general $S$. Then, we concentrate on a special subclass with $S = \mathbb{R}^d$ for which it is possible to prove projectivity, sparsity, and establish learnability results. We refer to this special class as rectangular LPMs.

\subsection{Sparse Latent Position Model}

Our new LPM's definition follows closely with that of the Poisson random connection model restricted to finite windows: the positions in the latent space are given by a homogeneous Poisson point process, and the link probability function $K$ is independent of the number of nodes. The main departure from the random connection model is formulating $K$ such that it depends on the inter-node distance in just a subset of the dimensions---specifically all but the auxiliary dimension. The following is a set of ingredients to formulate a sparse LPM. 

\begin{itemize}
\item \emph{Position Space:} A measurable metric space $(S, \mathcal{S}, \rho)$ equipped with a Lebesgue measure $\ell_1$. 
\item \emph{Auxiliary Dimension:} The measure space ($\mathbb{R}_+$, $\mathcal{B}$, $\ell_2$) where $\mathcal{B}$ is Borel and $\ell_2$ is Lebesgue.
\item \emph{Product Space:} The product measure space $(S^*, \mathcal{S}^*, \lambda)$ on $(S \times \mathbb{R}_+, \mathcal{S} \times \mathcal{B})$, equipped with $\lambda = \ell_1 \times \ell_2$, the coupling of $\ell_1$ and $\ell_2$.
\item \emph{Continuum of observation windows:} A function $H: \mathbb{R}_+ \rightarrow \mathcal{S}^*$ such that $t_1 < t_2 \Rightarrow H(t_1) \subset H(t_2)$ and $|H(t)| = t$.
\item \emph{Link probability function}: A function $K : \mathbb{R}_+ \rightarrow [0,1]$.
\end{itemize}

Jointly, we say the triple $((S,\mathcal{S}, \rho), H, K)$ defines a stochastic graph sequence called a sparse LPM. The position space plays the role of the latent space as in traditional LPMs, with the link probability function $K$ controlling the probability of an edge given the corresponding latent distance. The auxiliary dimension plays no role in connection probabilities. Instead, a node's auxiliary coordinate---in conjunction with its
latent position and the continuum of observation windows---determines when it appears. 

Specifically, a node with position $(Z,r)$ is observable at \emph{time}
$t \in \mathbb{R}_+$ if and only if $(Z,r) \in H(t)$. Here, time need not correspond to
physical time; it is merely an index for a continuum of graphs as in the case for the Poisson random connection model. We refer to
$t_i$---defined as the smallest $t \in\mathbb{R}_+$ for which $(Z_i,r_i) \in H(t)$---as the arrival time of
the $i$th node where $(Z_i,r_i)$ are the corresponding latent position and auxiliary value for
node $i$.

\begin{figure}[!h]
\centering
\subfloat[A realization of a point process on the product space. Square observation windows $H(t)$ for $t = 4,8,16$ are depicted in green, red, and purple, respectively. The points are coloured according to the first observation window for which they are observable.]{\includegraphics[width= 2.0 in]{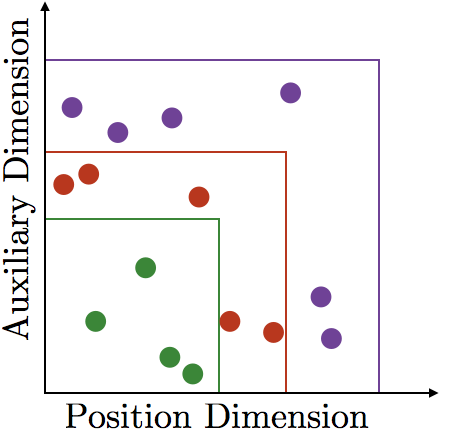}} \hspace{0.25 in}
\subfloat[Latent position graphs corresponding to the three observation windows depicted in Figure~1(a). The link probability function used is a decreasing function of distance in the position dimension.]{\includegraphics[width= 2.5 in]{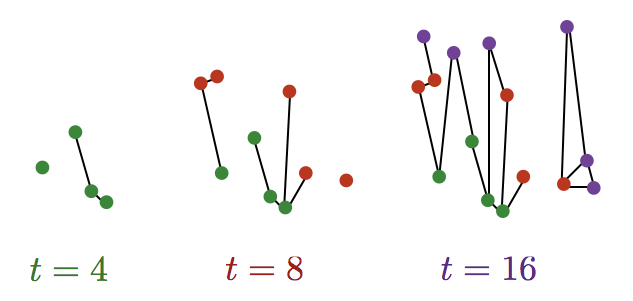}}
\caption{An example of a point process and observation windows which generate a sequence of sparse latent position graphs}
\end{figure}

Considered jointly, the coordinates defined by the latent positions and auxiliary positions
assigned to nodes can be viewed as a point process over
$S \times \mathbb{R}_+$. As in the Poisson random connection model, we assume this
point process is a unit-rate Poisson. The continuum of observation windows
$H(t)$ controls the portion of the point process which is observed at time
$t$. Since the size of $H$ is increasing in $t$, this model defines a growing sequence
of graphs with the number of nodes growing stochastically in $t$ as follows.

\begin{itemize}
\item Generate a unit-rate Poisson process $\Psi$ on $(S^*, \mathcal{S}^*)$.
\item Each point $(Z, r) \in S \times \mathbb{R}_+$ in the process corresponds to a node with latent position $Z$
and auxiliary coordinate $r$.
\item For a dyad on nodes with latent positions $Z_i$ and $Z_j$, include an edge with probability
  $K(\rho(Z_i, Z_j))$.
\item At time $t$ the subgraph induced by by restricting $\Psi$ to $H(t)$ is
  visible.
\end{itemize}
A graph of size $n$ can be obtained from the above framework by choosing any $t_n$ such that
$|\Psi \cap H(t_n)| = n$. Each $t_n < t_{n + 1}$ with probability one (by Lemma~\ref{waitingtimes}).
Thus, the above generative process is well-defined for any $n$, and the nodes are well-ordered by their arrival times.

Due to its flexibility, the above framework defines a broad class of LPMs. For instance, the exchangeable LPM can be viewed as a special case of the above framework in which the observation window grows only in the auxiliary dimension. However, the full generality of this framework makes it difficult to establish general sparsity and learnability results. For this reason, we have chosen to focus on a subclass of sparse LPMs to derive our sparsity, projectivity, and learnability proofs. We refer to this class as rectangular LPMs. We have chosen this class because it allows us to emphasize the key insights in the proofs without having to do too much extra bookkeeping.

\subsection{Rectangular Latent Position Model}

For rectangular LPMs, we impose further criteria on the basic sparse LPM. The latent space is assumed to be Euclidean ($S = \mathbb{R}^d$). The continuum of observation windows $H(t)$ are defined by the nested regions
\begin{align}
H(t) &= \left[- g(t), g(t)\right]^d \times \left[0, \frac{t}{(2g(t))^d} \right] 
\end{align}
where $g(t) = t^{p/d}$ for $0 \leq p\leq 1$ controls the rate at which the observation window grows for the latent position coordinates. The growth rate in the auxiliary dimension is chosen to be $2^{-d} t^{1- p}$ to ensure that the volume of $H(t)$ is $t$. We further assume that
\begin{align}
0< \int_0^{\infty} u^{d-1} K(u) \text{d}u < \infty
\end{align}
to ensure that the average distance between a node and its neighbors remains bounded as $n$ grows. We now demonstrate the projectivity and sparsity of rectangular LPMs as Theorems~\ref{projective} and \ref{sparsity}. 


\begin{theorem}\label{projective}
Rectangular sparse latent position network models define a projective sequence of models.
\end{theorem}
\begin{proof}
Provided in \S \ref{projectiveproof}
\end{proof}

\begin{theorem}\label{sparsity}
A $d$-dimensional rectangular latent position network model is $n^{2-p}$-sparse in expectation, where $g(n) = n^{p/d}$.
\end{theorem}
\begin{proof}
Provided in \S \ref{sparsityproof}
\end{proof}

By specifying the appropriate value of $p$ for a rectangular LPM, it is thus possible to obtain any polynomial level of sparsity within $n$-sparse and $n^2$-sparse (dense) in expectation. Other intermediate rates of sparsity such as $n\log(n)$ can also be obtained considering non-polynomial $g(n)$. We now investigate for which levels of sparsity it is possible to do reliable statistical inference of the latent positions.

\section{LEARNABILITY}\label{learnability}

\subsection{Preliminaries}\label{learnabilitypre}

Recall that the edge probabilities in a LPM are controlled by two things: the link probability function $K$ and the latent positions $Z \in S^n$. In this section, we consider the problem of consistently estimating the latent positions for a LPM using the observed adjacency matrix. We focus on the case where both $K$ and $S = \mathbb{R}^d$ are known, relying on assumptions that are compatible with rectangular LPMs.

In the process of establishing our consistent estimation results for $Z$, we also establish consistency results for two other quantities: the squared latent distance matrix $D^{Z} \in \mathbb{R}^{n \times n}$ defined by $D^Z_{ij} = \|Z_i - Z_j\|^2$ and the link probability matrix $P^Z \in [0,1]^{n \times n}$ defined by $P^Z_{ij} = K((D^Z_{ij})^{1/2})$. These results are also of independent interest because---like $Z$---the distance matrix and link probability matrix also characterize a LPM when $K$ is known.

We use the following notation and terminology to communicate our results. Let $\| \cdot\|_F$ denote the Frobenius norm of a matrix, $ \convp$ denote convergence in probability, $\mathcal{O}_d$ denote the space of orthogonal matrices on $\mathbb{R}^{d \times d}$, and $\mathcal{Q}_{nd} \subset \mathbb{R}^{n \times d}$ denote the set of all $n \times d$ matrices with identical rows. 

We say that a LPM has \emph{learnable latent positions} if there exists an estimator $\hat{Z}(Y^n)$ such that
\begin{align}
\lim_{n \rightarrow \infty} \inf_{O \in \mathcal{O}_d, Q \in \mathcal{Q}_{nd}}\frac{\|\hat{Z}(Y^n)O- Q -Z\|_F^2}{n}  \convp 0.
\end{align}
That is, a LPM has learnable positions if there exists an estimator $\hat{Z}(Y^n)$ of the latent positions such that the average distance between $\hat{Z}(Y^n)$ and the true latent positions converges to 0. The infimum over the transformations induced by $O \in \mathcal{O}_d$ and $Q \in \mathcal{Q}_{nd}$ is included to account for the fact that the likelihood of a LPM is invariant to isometric translations (captured by $Q$) and rotations/reflections (captured by $O$) of the latent positions \citep{CRS-Asta-consistency-of-embedding}. 

We say that a LPM has \emph{learnable squared distances} if there exists an estimator $\hat{Z}(Y^n)$ such that
\begin{align}
\lim_{n \rightarrow \infty} \frac{\|D^{\hat{Z}(Y^n)} - D^{Z}\|_F^2}{n^2}  \convp 0.
\end{align}
That is, a LPM has learnable squared distances if the average squared difference between the estimator for the matrix of squared distances induced by $\hat{Z}(Y^n)$ and the true matrix of squared distances $D^Z$ converges to 0. Unlike the latent positions, $D^Z$ is uniquely identified by the likelihood; there is no need to account for rotations, reflections, or translations. 

Finally, we say a LPM that is $e(n)$-sparse in expectation has \emph{learnable link probabilities} if there exists an estimator $\hat{Z}(Y^n)$ such that
\begin{align}
\lim_{n \rightarrow \infty} \frac{\|P^{\hat{Z}(Y^n)} - P^{Z}\|_F^2}{e(n)}  \convp 0.
\end{align}
Note that a scaling factor of $e(n)$ is used instead of $n^2$ to account for the sparsity. Otherwise the link probability matrix for a sparse graph could be trivially estimated because $n^{-2} \|P^{Z}\|_F^2  \convp 0$.

\subsection{Related Work on Learnability}

Before presenting our results, we summarize some of the existing work on learnability of LPMs in the literature.  \citet{Choi-Wolfe-LPM-learnability} considered the problem of estimating LPMs from a classical statistical learning theory perspective. They established bounds on the growth function and shattering number for LPMs with link function given by $K(\delta) = (1 + \exp{\delta})^{-1}$. However, we have found that their inequalities were not sharp enough to be helpful for proving learnability for sparse LPMs. 

\citet{CRS-Asta-consistency-of-embedding} provide regularity conditions under which LPMs have learnable positions on general spaces $S$, assuming that the link probability function $K$ is known and possesses certain regularity properties. Specifically, they require that the absolute value of the logit of the link probability function is slowly growing, which does not necessarily hold in our setting. 

Our learnability results more closely resemble those of \citet{Ma-Ma-Exploration}, who consider a latent variable network model of the form $\text{logit}(\mathbb{P}(A_{ij} = 1)) = \alpha_i + \alpha_j + \beta X_{ij} + Z_i^T Z_j$, originally due to \citet{Hoff-billinear-dyadic}. Here, $\alpha_i$ denote node-specific effects, $X_{ij}$ denote observed dyadic covariates and $\beta$ denotes a corresponding linear coefficient. If there are no covariates and $\alpha_i = \|Z_i\|^2/2$, their approach defines a LPM with $K(\delta) = \text{expit}(-\delta^2)$. \citet{Ma-Ma-Exploration} provide algorithms and regularity conditions for consistent estimation of both the logit-transformed probability matrix and $Z^TZ$ under this model, using results from \citet{Davenport-et-al-1-bit}. Here, we will use similar concentration arguments to establish Lemmas~\ref{consistentprob} and \ref{consistentdist}, but our results differ in that we consider a more general class of link functions, and also establish learnability of latent positions via an application of the Davis-Kahan theorem.

Our learnability of latent positions result (Lemma~\ref{consistentpositions}) resembles that of \citet{Sussman-Tang-Priebe-consistent}, who establish that the latent positions for dot-product network models can be consistently estimated. The dot product model---a latent variable model which is closely related to the LPM---has a link probability function defined by $K(Z_i, Z_j) = Z_i \cdot Z_j$ with $Z_j, Z_j \in S$. The latent space $S \subset \mathbb{R}^d$ is defined such that all link probabilities must fall with $[0,1]$. Our proof technique follows a similar argument as the one used to prove their Proposition 4.3.

It should be noted that learnability of the link probability matrix for the sparse LPM could be established by applying results from \emph{Universal Singular Value Thresholding} \citep{Chatterjee-matrix-estimation, Xu-rates-convergence-graphon}. However, it is unclear how to extend such estimators to establish learnability of the latent positions; estimated probability matrices from universal singular value thresholding do not necessarily translate to a valid set of latent positions for a given link function.

Other related work includes \citet{Arias-Castro-et-al-estimating-latent-distances}, which considers the problem of estimating latent distances between nodes when the functional form of the link probability function is unknown. They show that, if the link probability function is non-increasing and zero outside of a bounded interval, the lengths of the shortest paths between nodes can be used to consistently rank the distances between the nodes. \citet{Diaz-McDiarmid-Mitsche-geometric-graphs} and \citet{Rocha-Janssen-Kalyaniwalla-linear-graphs} also propose estimators in similar settings with more specialized link functions. None of these approaches are appropriate for our case---we are interested in recovering the latent positions under the assumption $K$ is known with positive support on the entire real line.

\subsection{Learnability Results}\label{criteria}

Our learnability results assume the following criteria for a LPM:
\begin{enumerate}
\item The link probability function $K$ is known, monotonically decreasing, differentiable, and upper bounded by $1-\epsilon$ for some $\epsilon > 0$.
\item The latent space $S\subseteq \mathbb{R}^d$.
\item There exists a known differentiable function $G(n)$ such that 
\begin{align}
I(\|Z_n\| \leq G(n)) & \convp 1.
\end{align}
\end{enumerate}
We refer to the above conditions as \emph{regularity criteria} and refer to any LPM that meets them as \emph{regular}. Criterion 3 implies that the sequence of latent positions is tight \citep[p. 66]{Kallenberg-mod-prob}. The class of regular LPMs contains several popular LPMs. Notably, both rectangular and exchangeable LPMs due to \citet{Hoff-Raftery-Handcock} are regular, as shown in Lemmas~\ref{rectangularareregular} and Lemma~\ref{exchangeableareregular}. For a rectangular LPM, the $G(n)$ in criterion 3 is closely related to $g(t)$---the width of the observation window. Specifically, it is established in Lemmas 10 and 11 in \S A.1 that a rectangular LPM with $g(t) = t^{p/d}$ satisfies criterion 3 with $G(n) = 2\sqrt{d} n^{p/d}$. Here, $t$ refers to the size of observation window (i.e. the expected number of observed nodes), and $n$ refers directly to the number of observed nodes.

Our approach for establishing learnability of $Z$ involves proposing a particular estimator for $Z$ which meets the learnability requirement as $n$ grows. Our proposed estimator is a restricted maximum likelihood estimator for $Z$, provided by the following equation:
\begin{align}\label{mle}
\hat{Z}(Y^n) = \text{argmax}_{z: \|z_i\| \leq G(n) \forall i \in 1:n} L(z: Y^n)
\end{align}
where $L(z: Y^n)$ denotes the log likelihood of latent positions $z = (z_1, \ldots z_n) \in \mathbb{R}^{n \times d}$ for a $n \times n$ adjacency matrix $Y^{n}$. We use $D^{\hat{Z}(Y^n)}$ and $P^{\hat{Z}(Y^n)}$ to denote the corresponding estimates of the squared distance matrix and link probability matrix. Note that the log likelihood $L(z:Y^n)$ is given by
\begin{align}
L(z: Y^n) &= \sum_{i=1}^n \sum_{j =1}^n Y^n_{ij}  \log \left(K(\|z_i - z_j\|)\right) + (1-Y^n_{ij})  \log \left(1- K(\|z_i - z_j\|)\right). \label{thelikelihood}
\end{align}

To establish consistency, we first provide a concentration inequality for the maximum likelihood estimate of $Z$ in Lemma~\ref{consistentpositions}. En route to deriving Lemma~\ref{consistentpositions}, we also derive inequalities for the associated squared distance matrix $D^{Z} \in \mathbb{R}^{n \times n}$ defined by $D^Z_{ij} = \|Z_i - Z_j\|_F^2$ (Lemma~\ref{consistentdist}) and the link probability matrix $P^Z \in [0,1]^{n \times n}$ defined by $P^Z_{ij} = K((D^Z_{ij})^{1/2})$ (Lemma~\ref{consistentprob}). We combine these results in Theorem~\ref{consistencies} to provide conditions under which it is possible to consistently estimate $Z$, $D^Z$, and $P^Z$. 

Our results are sensitive to the particular choices of link probability function $K$ and upper bounding function $G$. For this reason, we introduce the following notation to communicate our results.
\begin{align}
\alpha^K_{n} = \sup_{0 \leq x \leq 2G(n)} \frac{|K'(x)|}{|x| K(x)\epsilon},\label{alpha}\\
\beta^K_{n} = \sup_{0 \leq x \leq 2G(n)} \frac{x^2 K(x)}{K'(x)^2} \label{beta},
\end{align}
where $K'(x)$ denotes the derivative of $K(x)$ and $\epsilon$ is given by the criteria on $K$ imposed by regularity criterion 1.

\begin{lemma}\label{consistentprob}
Consider a sequence adjacency matrices $Y^n$ generated by a regular LPM with $\|Z_n\| \leq G(n)$ for all $n$. Let $P^{\hat{Z}(Y^n)}$ denote the estimated link probability matrix obtained via $\hat{Z}(Y^n)$ from (\ref{mle}). Then,
\begin{align}
\mathbb{P}\left(\|P^{\hat{Z}(Y^n)} - P^{Z}\|_F^2 \geq 16 e \alpha^K_n G(n)^2 n^{1.5} (d + 2) \right) \leq \frac{C}{n^2}
\end{align}
for some constant $C > 0$.
\end{lemma}
\begin{proof}
Provided in \S \ref{proofconsistentprob}.
\end{proof}

\begin{lemma}\label{consistentdist}
Consider a sequence adjacency matrices $Y^n$ generated by a regular LPM with $\|Z_n\| \leq G(n)$ for all $n$. Let $D^{\hat{Z}(Y^n)}$ denote the matrix of estimated squared distances obtained via $\hat{Z}(Y^n)$ from (\ref{mle}). Then,
\begin{align}
\mathbb{P}\left(\|D^{\hat{Z}(Y^n)} - D^{Z}\|_F^2 \geq 2^9 e \alpha^K_n \beta^K_n G(n)^2 n^{1.5} (d + 2) \right) \leq \frac{C}{n^2}
\end{align}
for some constant $C > 0$.
\end{lemma}
\begin{proof}
Provided in \S \ref{proofconsistentdist}.
\end{proof}

Establishing concentration of the estimated latent positions is complicated by the need to account for the minimization over all possible rotations, translations, and reflections. The following matrix, known as the double-centering matrix, is a useful tool to account for translations: 
\begin{align}\label{doublecentering}
\mathcal{C}_n = I_n - \frac{1}{n} 1_n1_n^T
\end{align}
Here, $I_n$ denotes the $n$-dimensional identity matrix and $1_n$ denotes $n \times 1$ matrix consisting of ones.

To establish our concentration of the estimated latent positions, we place conditions on the eigenvalues of the matrix $\mathcal{C}_n Z Z^T \mathcal{C}_n$.  For a regular LPM, let $\lambda_1 \geq \cdots \geq \lambda_d$ denote the $d$ nonzero eigenvalues of $\mathcal{C}_n Z Z^T \mathcal{C}_n$ and define $\lambda_{d+ 1} := 0$. For functions $a:\mathbb{N} \rightarrow \mathbb{R}$ and  $b:\mathbb{N} \rightarrow \mathbb{R}$, we say that a LPM possesses \emph{$a(n)-b(n)$ distinctly bunched eigenvalues} if there exists a $k \in \left\{1, \ldots, d\right\}$ and integers $i_{1}, \ldots, i_{k+1}$ satisfying $1= i_1 < i_2 < \cdots < i_k < i_{k+1} = d+1$ such that
\begin{align*}
\frac{\left(\lambda_{i_{j+1} - 1} - \lambda_{i_j} \right)^2}{ \lambda_{i_{j+1} - 1} } \leq a(n) \text{ and } \frac{\left(\lambda_{i_{j+1} - 1} - \lambda_{i_{j+1}} \right)^2}{ \lambda_{i_j}} \geq b(n) \text{ for all $j \in \left\{1, \ldots, k\right\}$.}
\end{align*}

In this definition, $i_1,\ldots, i_{k+1}$ are boundary indices partitioning the eigenvalues $\lambda_1, \ldots, \lambda_d$. Eigenvalues within the same subset of the partition can be thought of as remaining close to each other as $n$ increases, whereas those from different subsets are distinguishable from each other as $n$ grows. The levels of $a(n)$ and $b(n)$ dictate the level of proximity and distinguishability. Corollary~\ref{rectangular_eigenstable} in Appendix~\ref{A} establishes that rectangular LPMs possess $a(n)-b(n)$ distinctly bunched eigenvalues with $a(n)$ and $b(n)$ depending on the level of sparsity---sparser graphs require larger $a(n)$ and smaller $b(n)$'s. Similarly, Corollary~\ref{hoffeigenstable} establishes that exchangeable LPMs due to \citet{Hoff-Raftery-Handcock} possess $a(n)-b(n)$ distinctly bunched eigenvalues with $a(n) = O\left(1\right)$ and $b(n)^{-1} = O(n^{-1})$.

\begin{lemma}\label{consistentpositions}
Consider a sequence adjacency matrices $Y^n$ generated by a regular LPM possessing $a(n)-b(n)$ distinctly bunched eigenvalues with $\|Z_n\| \leq G(n)$ for all $n$. Then,
\begin{align}
\mathbb{P}\left(\inf_{\substack{O \in \mathcal{O}_{d}\\ Q \in \mathcal{Q}_{nd}}} \|\hat{Z}(Y^n)O - Z - Q\|^2_F \geq \frac{3 a(n)}{(d - 1)^{-1}} +  \frac{ 2^9 e \alpha^K_n \beta^K_n G(n)^2 n^{1.5}}{(50 (d + 2))^{-1} b(n) }  \right) \leq \frac{C}{n^2}
\end{align}
for $C > 0$, where $\mathcal{O}_d$ denotes the space of orthogonal matrices on $\mathbb{R}^{d \times d}$, $\mathcal{Q}_{nd} \subseteq \mathbb{R}^{n \times d}$ is the set of matrices with $n$ identical $d$-dimensional rows,  and $\hat{Z}(Y^n)$ is obtained via (\ref{mle}).
\end{lemma}
\begin{proof}
Provided in \S \ref{proofconsistentpositions}.
\end{proof}


These three concentration results can be translated into sufficiency conditions for learnability. We summarize these in Theorem~\ref{consistencies}. 
\begin{theorem} \label{consistencies}
A regular LPM that is $e(n)$-sparse in expectation has:
\begin{enumerate}
\item learnable link probabilities if $\alpha^K_n e(n)^{-1} n^{1.5} G(n)^2 \rightarrow 0$ as $n$ grows.
\item learnable squared distances if $\beta^K_n \alpha^K_n n^{-0.5} G(n)^2 \rightarrow 0$ as $n$ grows.
\item learnable latent positions if it possesses $a(n)-b(n)$ distinctly bunched eigenvalues with
\begin{align*}
\frac{a(n)}{n} \rightarrow 0 \text{  and  }
\frac{\beta^K_n \alpha^K_n n^{0.5} G(n)^2}{b(n)} \rightarrow 0
\end{align*}
as $n$ grows.
\end{enumerate} 
\end{theorem}
\begin{proof}
Provided in \S \ref{consistenciesproof}.  
\end{proof}

It may seem counter-intuitive that the conditions for learnability of $Z$, $P^Z$ and $D^Z$ differ, even though their estimators are all derived from the same quantity. For example, if $\beta^K_n$ grows quickly enough, the LPM may have learnable link probabilities but not squared distances. This disparity can be understood by considering the metrics implied by each form learnability. 

Suppose that $\delta_{ij} = \|Z_i - Z_j\|$ is very large. Then mis-estimating $\delta_{ij}$ by a constant $c>0$  (i.e. $\hat{\delta}_{ij} = \delta_{ij} + c$) contributes $(2\delta_{ij}c + c^2)^2$ to the error in $\|D^{\hat{Z}} - D^Z\|_F^2$. This contribution to the error is sizable, and can hinder convergence if made too often. However, the influence of the same mistake on $\|P^{\hat{Z}} - P^Z\|_F^2$ is minor; because the probability $K(\delta)$ is already small for large $\delta$, $(K(\delta +c) - K(\delta))^2$ does not contribute much to the error. For small distances, the opposite may be true; a small mistake in estimated distance may lead to a large mistake in estimated probability. Thus, learnability of squared distances penalizes mistakes differently than learnability of link probabilities. However, there are typically far more large distances than small distances, meaning that the distance metric imposed by learnability of link probabilities is typically less stringent than for learnability of squared distances.

Theorem~\ref{consistencies} can be used to establish Corollary~\ref{rectangularlearnability}, a learnability result for rectangular LPMs. 
\begin{cor} \label{rectangularlearnability}
Consider a $d$-dimensional rectangular LPM with $g(n) = n^{p/d}$ and link probability function $K(\delta) = (C + \delta^2)^{-q}$ for some $C > 0$, where $q > \text{max}(\left\{d/2, 1\right\})$ and $0 \leq p \leq 1$. Such a network has learnable 
\begin{enumerate}
\item link probabilities if $2p < \left( 1 + 2/d\right)^{-1}$,
\item distances if $2p < d\left( 2q+6\right)^{-1}$,
 \item latent positions if $2p < d\left( 2q+4\right)^{-1}$. 
\end{enumerate}
Thus, for any $s \in (1.5, 2]$, it is possible to construct a LPM that is projective, $n^{s}$-sparse in expectation, and has learnable latent positions, distances, and link probabilities. 
\end{cor}
\begin{proof}
Provided in \S \ref{rectangularlearnabilityproof}.  
\end{proof}
Corollary~\ref{rectangularlearnability}, combined with the projectivity of rectangular LPMs, guarantees the existence of a LPM that is projective, learnable, and sparse for any sparsity level that is denser than $n^{3/2}$-sparse in expectation. Thus, we have shown that we have met our desiderata for LPMs laid out in the introduction.

Perhaps surprisingly, our result in Corollary~\ref{rectangularlearnability} depends upon the dimension of the latent space. The higher the dimension, the richer the levels of learnable sparsity. Moreover, the learnability results in Theorem~\ref{consistencies} only apply to rectangular LPMs with link functions that decay polynomially. The $\beta^K_n$ term is too large for the exponential-style decays that are commonly considered in practice \citep{Hoff-Raftery-Handcock, Rastelli-latent-variable-network}. We elaborate on these points in \S\ref{conclusion}.

In contrast, it is possible to prove learnability of exchangeable LPMs with exponentially decaying $K$. Corollary~\ref{hofflearnability} guarantees learnability of the exchangeable LPM for two exponential-style link functions. As far as we are aware, these are the first result learnability results for the latent positions for the original exchangeable LPM. 


\begin{cor}\label{hofflearnability}
Consider a LPM on $S = \mathbb{R}^d$ with each latent position independently and identically distributed according to the multivariate Gaussian distribution with mean zero and diagonal variance matrix $\Sigma$. Let $\sigma^2_1, \ldots, \sigma^2_d$ denote the entries along the diagonal of $\Sigma$, with $\sigma_1 \geq \sigma_2 \geq \cdots \geq \sigma_d > 0$. Suppose that the link probability function is given by either
\begin{align}
K(\delta) = (1 + \exp{(\delta^2)})^{-1} \text{  or  }
K(\delta) = \tau e^{-\delta^2}.
\end{align}
for $\tau \in (0,1)$. Such a network has learnable link probabilities, distances, and latent positions provided that $\sigma_1^2 < 1/4$.
\end{cor}
\begin{proof}
Provided in \S \ref{hofflearnabilityproof}.  
\end{proof}

Notably, the set of link functions in Corollary~\ref{hofflearnability} does not include the traditional expit link function that was suggested in the original paper LPM by \citet{Hoff-Raftery-Handcock}. The expit class of link functions implies a value $\alpha_n^k$---defined as in (\ref{alpha})---that is unbounded (see Table~\ref{differentlinks} in Appendix~\ref{A} for a summary of the $\alpha^k_n$ and $\beta^K_n$ values for various link functions), meaning that Lemma~\ref{consistencies} cannot be applied to prove learnability for this class of LPMs. This does not necessarily mean that expit LPMs are not learnable, just that determining their learnability remains an open problem. Note however, that some classes of sparse LPMs (such as the example considered in Theorem~\ref{negative} (\S \ref{negatives})) \emph{are} provably unlearnable. We elaborate on this point in \S\ref{conclusion}.

The results in Theorem~\ref{consistencies} can also be used to obtain learnability results for more specialized LPMs such as sparse graphon-based LPMs. We provide such a result in \S\ref{sparsegraphon} when comparing sparse graphons with our approach.

\section{COMPARISONS AND REMARKS}\label{sec:remarks}

Existing tools for constructing sparse graph models, such as the sparse graphon framework \citep{Bollobas-Janson-Riordan-Phase,
  Borgs-Chayes-Zhao-sparse-graph-convergence} or the graphex framework \citep{Caron-Fox-sparse-graphs, Veitch-Roy-graphex-class,
  Borgs-Chayes-Holden-sparse-exchangeable} can be used to develop suitably sparse latent position models. However, both approaches introduce sparsity in ways that have undesirable side effects for LPMs. We now describe both the sparse graphon framework (\S \ref{sparsegraphon}) and the graphex framework (\S \ref{graphex}), with discussion of how these frameworks fail to meet our desiderata of projectivity, learnability, and other useful properties for LPMs such as edge transitivity. Finally, we conclude by making some remarks on the results we have derived this article (\S \ref{conclusion}).

\subsection{Sparse Graphon-based Latent Position Models}\label{sparsegraphon}

\citet{Borgs-Chayes-Zhao-sparse-graph-convergence} proposed a modification of
graphon models to allow sparse graph sequences. Seeing as exchangeable LPMs are within the graphon family, it is straightforward to specialize
this approach to define sparse graphon-based LPMs.

As in exchangeable latent position models, the latent positions for a sparse
graphon-based LPM are each drawn from a common
distribution $f$, independently of each other the number of nodes $n$. However, the
link probability function
$\mathbb{P}(Y_{ij} = 1|Z_i, Z_j) = K_n(\rho(Z_i, Z_j))$ is allowed to depend on
$n$. Specifically, $K_n(x) = \text{min}(\left\{s_nK(x), 1\right\})$ where $(s_n)_{1\ldots \infty}$ is a
non-increasing sequence and $K: \mathbb{R}_+ \rightarrow \mathbb{R}_+$ satisfies $\mathbb{E}(K(\rho(Z_i, Z_j))) < \infty$ for $Z_i, Z_j \sim f$. These models express
sparse graph sequences, with the sequence $(s_n)_{1\ldots \infty}$ controlling
the sparsity of the resultant graph sequence.

\begin{prop}\label{sparsegraphonissparse}
Sparse graphon-based latent position models define a $n^2 s_n$-sparse in expectation graph sequence. 
\end{prop}
\begin{proof}
Proof provided in \S \ref{sparsegraphonissparseproof}.
\end{proof}

Moreover, the learnability results in Theorem~\ref{consistencies} can be used to establish learnability results for sparse graphon-based versions of popular LPMs. 
\begin{cor}\label{learnablesparsegraphon}
Consider the following sparse graphon-based version of the exchangeable LPM. Let $S = \mathbb{R}^d$ with the latent positions distributed according an isotropic Gaussian random vector with any variance $\sigma^2 < 1/4$. Suppose that the link probability function is given by either
\begin{align}
K_n(\delta) = n^{-p}(1 + \exp{(\delta^2)})^{-1} \text{  or  }
K_n(\delta) = \tau n^{-p} e^{-\delta^2}
\end{align}
for $\tau \in (0,1)$, $0 \leq p \leq 1$. Such a network has learnable link probabilities, squared distances, and latent positions if $ p < 1/2 - 2\sigma^2(1+c)$ for $c>0$. Thus, given an appropriate $\sigma^2$, this LPM can be both $n^b$-sparse and learnable for $b\in(1.5,2]$.
\end{cor}
\begin{proof}
Proof provided in \S \ref{learnablesparsegraphonproof}
\end{proof}
As such, many sparse graphon-based LPMs achieve learnability under the same sparsity rate derived for rectangular LPMs in Corollary~\ref{rectangularlearnability}. Additionally, learnability can be established for link probability functions with lighter tails, as well as for latent spaces of arbitrary dimension $d$. These findings suggest a potential trade-off between projectivity and learnability under lighter-tailed link probability functions.

Despite these advantages, there are practical ramifications of sparse graphon-based LPMs that limit their applicability as statistical models for a network. To start, the resultant sparse network sequences are not projective.
\begin{prop}\label{notprojective}
Sparse-graphon latent position models do not define a projective sequence of models if $(s_n)_{n=1\ldots \infty}$ is not constant.
\end{prop}
\begin{proof}
Proof provided in \S \ref{notprojectiveproof}.
\end{proof}
As noted in \S~\ref{sec:projectivity}, inferences drawn using non-projective network models can be difficult to interpret, especially when the statistical application requires super-population inference \citep{DAmour-Airoldi-misspecification}. As such, extra care must be taken to ensure sparse graphon-based inferences are reliable for a given application.

In the specific context of LPMs, another ramification stems from the particular way the non-projective link function $K_n(x) = \text{min}(\left\{s_nK(x), 1\right\})$ is defined. In particular, consider the probability of edge transitivity in sparse graphon-based LPMs as $n$ increases. Edge transitivity---that is, the extra tendency for two nodes in a network to be connected given a shared neighbor---is one of the main selling points identified by \cite{Hoff-Raftery-Handcock} in their initial proposal of the LPM was as a useful statistical model. Notably, the triangle inequality for distances combines with the strictly decreasing LPM link probability function to promote transitivity in virtually all commonly-used exchangeable LPMs. Somewhat surprisingly, however, is the fact that this fact does not hold for sparse graphon-based versions of popular LPMs. Under fairly general conditions, the conditional probability of two nodes being connected given a shared neighbor declines to zero as $n$ grows. We formally state this result as Theorem~\ref{triangles}.
\begin{theorem}\label{triangles}
Consider a sparse graphon-based latent position model on the latent space $\mathbb{R}^d$ equipped with the Euclidean distance. Suppose that the sequence of link functions is given by $K_n(\delta) = \text{min}\left(\left\{s_nK(\delta), 1\right\} \right)$ where $(s_n)_{n \in \mathbb{N}}$ is a non-increasing sequence with a limit of 0 (i.e. the resultant LPM is sparse), and $K(\delta)$ is a non-negative, continuous, strictly decreasing function satisfying $\mathbb{E}(K(\rho(Z_i, Z_j)) K(\rho(Z_i, Z_k)) K(\rho(Z_j, Z_k)) ) < \infty$ for $Z_i, Z_j, Z_k \sim f$. Under these conditions, the resultant LPM will satisfy
\begin{align}
\PP(Y^n_{ij} =  1\mid Y^n_{ik}=1, Y^n_{jk} = 1) \rightarrow_p 0 
\end{align}
as $n \rightarrow \infty$ for any arbitrarily chosen node indices $(i,j,k)$. That is, the probability of edge transitivity will go to 0 as the number of nodes goes to infinity.
\end{theorem}
\begin{proof}
Proof provided in \S \ref{prooftriangles}.
\end{proof}

The conditions required for Theorem~\ref{triangles} are quite general. Notably,
\begin{align}
\mathbb{E}(K(\rho(Z_i, Z_j)) K(\rho(Z_i, Z_k)) K(\rho(Z_j, Z_k)) ) < \infty
\end{align}
is guaranteed to hold for any bounded function $K$, such as the standard choices expit($-\rho^2$), $\exp{(-\rho^2)}$ and $(1 + \rho^2)^{-1}$, as well as many choices of unbounded $K$. For this reason, it may be undesirable to consider sparse-graphon based LPMs to model real-world networks in which edge transitivity is expected to be present, at least for standard link functions.

In contrast, the sparse LPMs presented in Section~\ref{sec:sparse-latent-space} exhibit nonzero probabilities of edge transitivity as $n$ increases under standard assumptions on $K$. We establish this fact as Theorem~\ref{triangles2}.

\begin{theorem}\label{triangles2}
Let $\pi$ be a permutation on $(1, \ldots, n)$ chosen uniformly at random from the set of permutations on $(1, \ldots, n)$ for $n \in \mathbb{N}$. A $d$-dimensional regular rectangular latent position network model will satisfy
\begin{align}
 \lim_{n \rightarrow \infty} \PP(Y^n_{\pi(i)\pi(j)} =  1\mid Y^n_{\pi(i) \pi(k)}=1, Y^n_{\pi(j)\pi(k)} = 1) \neq 0.
\end{align}
That is, the probability of edge transitivity does not goes to 0 as the number of nodes goes to infinity.
\end{theorem}
\begin{proof}
Proof provided in \S \ref{prooftriangles2}.
\end{proof}
As such, our projective sparse LPMs are more suitable for modeling networks where edge transitivity is expected to be present. Seeing as edge transitivity is a primary selling point of LPMs, we argue that our LPMs are thus more suitable in most practical applications.

Finally, It is also worth acknowledging that the sparse graph representation of \citet{Bollobas-Janson-Riordan-Phase} is more general than the sparse graphon
representation described above. It allows for latent variables assigned defined through a point process rather than generated independently
from the same distribution. For LPMs, this set-up equates to the traditional random connection model (\S \ref{prc}).

\subsection{Comparison with the Graphex Framework}\label{graphex}

Beyond the random connection model \citep{Meester-Roy-continuum-percolation},
there has been a recent renewed interest in using point processes to define
networks. This was primarily spurned by the developments in
\citet{Caron-Bayesian-bipartite} and \citet{Caron-Fox-sparse-graphs} in which
they propose a new graph framework---based on point processes---for infinitely
exchangeable and sparse networks. This approach was generalized as the graphex framework in
\citet{Veitch-Roy-graphex-class}.  Other
variants and extensions of this work include
\citet{Borgs-Chayes-Holden-sparse-exchangeable, Herlau-Schmidt-Morup,
  Palla-Caron-Teh, Todeschini-Miscouridou-Caron}.

In the graphex framework, a graph is defined by a homogeneous Poisson process
on an augmented space $\mathbb{R}_+ \times \mathbb{R}_+$, with the points
representing nodes. The two instances of $\mathbb{R}_+$ play the roles of the
parameter space and the auxiliary space. The parameter space determines the
connectivity of nodes through a function $W:\mathbb{R}_+^2 \rightarrow
[0,1]$. Connectivity is independent of the auxiliary dimension $\mathbb{R}_+$ that
determines the order in which the nodes are observed. Clearly, our sparse LPM set-up shares
many similarities with the graphex framework. Both assign latent variables to nodes according to a
homogeneous Poisson process defined on a space composed of a parameter space to influence connectivity and an auxiliary space to influence order of node arrival. The graphex is defined in terms of a one-dimensional parameter space, but it
can be equivalently expressed as a multi-dimensional parameter space as we do for the sparse LPM.
The link probability function $K$ for the sparse LPM depends solely on the distance between points, but it would be straightforward to extend to the more general set-up for $W$ as in the graphex. However, it would take additional work to determine the sparsity levels and learnability properties
   of such graphs.
 
The major difference between our framework and the graphex framework is how a finite subgraph
is observed. To observe a finite graphex-based graph, one restricts the point
process to a window $\mathbb{R}_+ \times [0, \nu]$. Here, the restriction is
limited to the auxiliary space, with the parameter space remaining
unrestricted. This alone is not enough to lead to a finite graph, as a unit
rate Poisson process on $\mathbb{R}_+ \times [0,\nu]$ still has an infinite
number of points almost-surely. To compensate, an additional criterion for node
visibility is included. A node is visible only if it has at least one neighbor. For some choices of $W$, this results in a finite number of
visible nodes for a finite $\nu$.  \citet{Veitch-Roy-graphex-class} show that
the expected number of nodes $n_{\nu}$ and edges $e_{\nu}$ are given by
\begin{align}
\mathbb{E}\left(n_{\nu} \right) &= \nu \int_{0}^{\infty}1 - \exp{\left( - \nu \int_{\mathbb{R}_+} W(x,y) \text{d}y \right)} \text{d}x,\label{graphexes1} \\ 
\mathbb{E}\left(e_{\nu} \right) &= \frac{1}{2} \nu^2 \int_{0}^{\infty}\int_0^{\infty} W(x,y) \text{d}x \text{d}y \label{graphexes2}
\end{align}
respectively. Thus, the degree of sparsity in the graph is controlled through
the definition of $W$. Clearly, for a finite-node restriction to be defined,
the two dimensional integral over $W$ in (\ref{graphexes2}) must be finite. Otherwise, the number of
nodes is infinite for any $\nu$.

A sparse graphex-based LPM cannot be implemented in the naive
manner because, if $W$ is solely a function of distance between nodes, the two
dimensional integral (\ref{graphexes2}) is infinite. One modification to prevent this to
modify $W$ to have bounded support, e.g.
$W(x,y) = K(|x-y|) I(0 \leq x,y \leq C)$. However, this
framework is equivalent to the graphon framework and results in dense graphs
\citep{Veitch-Roy-graphex-class}. It does not define a sparse LPM.

Alternatively, we could relax the graphex such that latent positions are generated according to an inhomogeneous point process over the parameter sparse. This can be done though the definition of $W$. For instance, consider
\begin{align}
W(x,y) &= K\left(|\exp{\left(x \right)} - \exp{\left(y \right)}|\right).
\end{align}
with $K$ being the link probability function as defined in the traditional LPM. In this set-up, $W$ can be viewed as the composition of two operations. First, an exponential transformation is applied to the latent positions resulting in an inhomogeneous rate function given by $f(x)= 1/(1+x)$. Then, we proceed as if it were a traditional LPM in this new space, connecting the nodes according to $K$ on their transformed latent positions. Finally, the isolated nodes are discarded. This approach defines a sparse and projective latent position network model, with the level of sparsity controlled by $K$. Though (\ref{graphexes1}) and (\ref{graphexes2}) provide a means with which to calculate the sparsity level, these expressions do not yield analytic solutions for most $K$. As a result, the graphex framework is far more difficult to work with when defining sparse LPMs; they lack the straightforward control over the
level of sparsity provided by the growth function $g(t)$ in rectangular LPMs.

Furthermore, it is difficult to apply the tools derived in Theorem~\ref{consistencies} to establish learnability for graphex-based LPMs. The difficulty stems from the fact that regularity requires a probability bound on the maximum of the distances between the origin and the first $n$ observed nodes. That is, we need a bound on $\max_{i \leq n} \|X_i\|$  where $X_1 \ldots, X_n$ denote the latent positions of the first $n$ observed nodes. Because of the irregular sampling scheme in which isolated nodes are discarded, it is difficult to establish such a bound for the graphex. Furthermore, any such bound is usually large due to the fact the latent positions at any $n$ are generated according to an improper distribution. For this reason, whether or not such graphex-based LPMs are learnable is an open problem.



\subsection{Remarks}\label{conclusion}

We have established a new framework for sparse and projective latent position models that enables straightforward control the level of sparsity. The sparsity is a result of assuming the latent positions of nodes are a realization of a Poisson point process on an augmented space, and that the growing sequence of graphs is obtained by restricting observable nodes to those with positions in a growing sequence of nested observation windows.

The notion of projectivity we consider here is slightly weaker than the one usually considered in the literature (e.g. \citet{your-favorite-ergm-sucks}). Our definition
requires consistency under marginalization of the most recently arrived node, rather than consistency under marginalization of any node. We do not consider this to be a major limitation---if the entire sequence of graphs were observed, the order of the nodes would be apparent.

In practice, only a single network of finite size is available when conducting inference. However, in these cases the order of nodes is not required---we make no use of it when defining the maximum likelihood estimator. A finite observation from our new sparse LPM is equivalent to finite observation from an equivalent exchangeable LPM with $f$ given by the shape of $H(n_t)$. This follows from Lemma~\ref{restriction} which indicates that the distribution of latent positions can be viewed as iid after conditioning on the number of nodes and randomly permuting the ordering. This means that the analysis and inference tools developed for exchangeable LPMs extend immediately to our approach when analyzing a single, finite network. From this viewpoint, we have merely proposed a different asymptotic regime for studying the same classes of models available under the exchangeability assumption. 

Theorem~\ref{consistencies} provides some consistency results under this asymptotic regime. However, the rates of learnability we achieved are upper bounds---the inequalities in Lemmas \ref{consistentprob}-\ref{consistentpositions} are not necessarily tight. They are derived to hold even for the worse-case regular LPMs regardless of how the latent positions are generated. We demonstrate in Theorem~\ref{negative} (\S \ref{negatives}) that there are some classes regular LPMs for which it is impossible to learn the latent positions. This class of models includes any regular LPMs with $G(n) = n^{p/d}$ and $K$ exponentially decreasing. In these cases, it is possible for the LPM to result in graphs which are disconnected with probability trending to one by clustering the latent positions at two extreme points of the space.

Though the regularity criteria technically allow for such instances by placing no assumptions on the distribution of $Z$ besides bounded norms, these clusters arise with vanishing probability when the latent positions are assumed to follow a homogeneous Poisson process such as in rectangular LPMs. For this reason, a future research direction to explore is to establish better learnability rates for rectangular LPMs by tightening the bounds Lemmas \ref{consistentprob}-\ref{consistentpositions} through assumptions on the distribution of the latent positions.

\appendix

\section{Proofs of Results and Supporting Lemmas}\label{A}

\subsection{Intermediary Results}\label{intermed}

The following are useful lemmas toward establishing the main results in this article.

\begin{lemma}\label{restriction}
  Restriction Theorem in \citet[p.\
17]{Kingman-Poisson} Let $\Lambda$ be a Poisson process with mean measure $\mu$ on $S$, and let
  $S_1$ be a measurable subset of $S$. Then the random countable set
  \begin{equation}
    \Lambda_1 = \Lambda \cap S_1
  \end{equation}
  can be regarded as a Poisson process on $S$ with mean measure
  \begin{equation}
    \mu_1 (A) = \mu(A \cap S_1)
  \end{equation}
  or as a Poisson process on $S_1$ possessing a mean measure that is the restriction of
  $\mu$ to $S_1$.
\end{lemma}

\begin{lemma}\label{prop:unif}
  For a rectangular LPM, the number of nodes which are
  visible at time $t$ is Poisson distributed with mean $t$.
\end{lemma}
\begin{proof} According to Lemma~\ref{restriction}, the latent positions of nodes visible at
  time $t$ follow a unit-rate Poisson process over $H(t)$. Therefore, the
  number of nodes is Poisson distributed with expectation equal to the volume
  of $H(t)$, which is $t$.
\end{proof}

\begin{lemma}\label{waitingtimes}
Let $t_n$ denote the arrival time of the $n$th node in a sparse LPM. Then, $t_n \sim $ Gamma($n$, 1) if $H(t)$ has volume $t$. Moreover, $t_n/n \convas 1$.
\end{lemma}
\begin{proof}
Let $n_t = |\Psi \cap H(t)|$ where $\Psi$ denotes the unit rate Poisson process of latent positions. Then, it is straightforward to verify that $n_t$ follows a one-dimensional homogeneous Poisson process on the positive real line. Note that $t_n$ can be equivalently expressed as 
\begin{align}
t_n = \inf\left\{t \geq 0: |\Psi \cap H(t) | = n\right\}.
\end{align}
That is, $t_n$ is the index of the smallest observation window containing $n$ nodes for all positive integers $n$. Under this perspective, $t_n$ can be viewed as a stopping time of $n_t$. It is well-known that $t_1$, the first arrival time of a unit-rate Poisson process, follows an exponential distribution with rate 1. Then, by the strong Markov property of Poisson processes $t_n - t_{n-1}$ is identical in distribution to $t_1$. Thus, $t_n$ is equivalent to the sum of $n$ independent exponential distributions, meaning it follows Gamma($n$, 1). The fact that $t_n/n \overset{a.s.}{\to} 1$ follows from the strong law of large numbers because $t_n$ is the sum of $n$ independent exponential random variables with mean one.
\end{proof}

\begin{lemma}\label{prop:indunif}
Consider a sparse rectangular LPM. Let $z$ denote the latent position of a node chosen uniformly at random of the nodes visible at time $t$. Then $z$ follows a uniform distribution over $[-g(t) , g(t)]^d$.
\end{lemma}
\begin{proof}
If a node is visible at time $t$, its latent position and auxiliary coordinate pair $(z,r)$ are a point in a unit-rate Poisson process restricted H(t). By Lemma~\ref{restriction}, this point process is a Poisson process with unit rate over the restricted space. Thus, if a node is visible at time $(z, r)$, it is uniformly distributed over $H(t) = [-g(t),g(t)]^d \times [0,t/(2g(t))^d]$. Marginalizing $r$ provides the result.
\end{proof}

\begin{lemma}\label{distanceswitch}
Let $K$ be a decreasing non-negative function such that
\begin{align}
0<\int_0^{\infty} r^{d-1}K(r) \text{dr} < \infty, \label{thisispositive}
\end{align}
for some $d \in \mathbb{Z}_+$. Then,
\begin{align}
0< \int_{y \in [-B,B]^d} K(\|x - y\|) \text{d}y < \infty
\end{align}
for any $B\in \mathbb{R}_+$.
\end{lemma}
\begin{proof}
Note that for all decreasing positive functions $K$, the function
\begin{align}
R(x) = \int_{y \in [-B,B]^d} K(\|y-x\|) \text{d}y 
\end{align}
is maximized when $x$ is at the origin. Thus, for all $x \in \mathbb{R}^d$,
\begin{align}
\int_{y \in [-B,B]^d} K(\|y-x\|) \text{d}y &\leq  \int_{y \in [-B,B]^d} K(\|y\|) \text{d}y\\
                                                           &\leq \int_{y \in \mathbb{R}^d:\|y\| < \sqrt{d} B} K(\|y\|) \text{d}y\\
                                                           &\propto \int_{0}^{\sqrt{d} B} r^{d-1} K(r) \text{d}r\\
                                                           &< \infty.
\end{align}
The positivity follows from $K$ being non-negative and the positivity of the expression in (\ref{thisispositive}).
\end{proof}

\begin{lemma}\label{uniformZ}
Consider a rectangular LPM, with $t_i$ denoting the arrival time of the $i$th node. Let $\pi$ denote permutation chosen uniformly at random from all permutations on $\left\{1,\ldots, n-1\right\}$. Then, conditional on $t_n = T$, each $t_{\pi(i)}$'s marginal distribution is uniform on $[0, T]$ for $i=1,\ldots, n-1$. Consequently, the latent position $Z^{\pi(i)}$ of node $\pi(i)$ is uniformly distributed on $[-g(T), g(T)]^d$. 
\end{lemma}
\begin{proof}
Let $(w_i)_{i=1,\ldots,n}$ denote the inter-arrival of times of the nodes. That is, $w_1 = t_1$ and $w_i = t_i - t_{i-1}$. As argued in the proof of Lemma~\ref{waitingtimes}, each $w_i$ is exponentially distributed. Thus, the density of $t_1, \ldots, t_{n-1}$ given $t_n = T$ satisfies:
\begin{align}
f(t_1, \ldots, t_{n-1}|t_n = T) \propto I(0 \leq t_1 \leq t_2 \leq \cdots \leq t_{n-1} \leq t_{n})
\end{align}
which is the same density as the order statistics of a uniform distribution on $[0,T]$. Thus, a randomly chosen waiting time $t_{\pi(i)}$ is uniformly distributed on $[0,T]$. Let $r^{\pi(i)}$ denote the auxiliary coordinate of node $\pi(i)$. It follows that $\mathbb{P}((Z^{\pi(i)}, r^{\pi(i)}) \in [-g(a), g(a)]^d \times [0,a/g(a)^d]) = a/T$ for all $0 \leq a \leq T$. It follows that $Z^{\pi(i)}$ is uniformly distributed on $[-g(T), g(T)]^d$.
\end{proof}

\begin{lemma} \label{windowdist}
Consider a rectangular sparse LPM model restricted to $H(t_n)$ such that $n$ nodes are visible. Let $\left\{Z_1, \ldots, Z_n \right\}$ denote the latent positions of these nodes. Let $\delta^{(n)} = \max_{i=1,\ldots n} \|Z_i\|$ denote the largest Euclidean distance between a visible node's latent position and the origin. Then,
\begin{align}
\mathbb{P}(\delta^{(n)} >  \sqrt{d}g(n +\sqrt{n\log(n)})) &\leq \log(n)^{-1}
\end{align}
indicating that
\begin{align}
\lim_{n \rightarrow \infty} \mathbb{P}(\delta^{(n)} > \sqrt{d} g(n + \sqrt{n\log(n)})) \rightarrow 0. \label{a2}
\end{align} 
Consequently,
\begin{align}
\lim_{n \rightarrow \infty} \mathbb{P}(\delta^{(n)} > 2\sqrt{d} g(n)) \rightarrow 0. \label{a3}
\end{align} 
\end{lemma}
\begin{proof}
Let $Z_{ij}$ denote the $j$th latent coordinate of node $i$. By construction, $\|Z_{ij}\| \leq g(t_n)$ for any $i\leq n, j \leq d$. Thus, $\delta^{(n)} \leq \sqrt{d} g(t_n)$. By Lemma~\ref{waitingtimes}, know that $t_n \sim$ Gamma($n$, 1). By Chebyshev inequality,
\begin{align}
\mathbb{P}(|t_n - n| > \sqrt{n\log(n)}) &\leq \log(n)^{-1}\\ 
\Rightarrow \mathbb{P}(t_n > n +\sqrt{n\log(n)}) &\leq \log(n)^{-1}\\
\Rightarrow \mathbb{P}(g(t_n) > g(n +\sqrt{n\log(n)})) &\leq \log(n)^{-1}\\
\Rightarrow \mathbb{P}( d^{-1/2} \delta^{(n)} > g(n +\sqrt{n\log(n)})) &\leq \log(n)^{-1}\\
\Rightarrow \mathbb{P}( \delta^{(n)} >  \sqrt{d}g(n +\sqrt{n\log(n)})) &\leq \log(n)^{-1}
\end{align}
The result in (\ref{a2}) follows from taking the limit, and the result in (\ref{a3}) follows from $g(n +\sqrt{n\log(n)}) \leq 2 g(n)$ for all non-decreasing $g(n) = n^{p/d}$ and $n \geq 1$.
\end{proof}

\begin{lemma}\label{rectangularareregular}
Rectangular LPMs are regular with $G(n) = 2\sqrt{d} n^{p/d}$.
\end{lemma}
\begin{proof}
Criteria 1 and 2 of a regular LPM hold by definition of a rectangular LPM. Lemma~\ref{windowdist} guarantees that satisfaction of criterion 3.
\end{proof}

\begin{lemma}\label{exchangeableareregular} 
Consider a LPM on $S = \mathbb{R}^d$ with each latent position independently and identically distributed according to a multivariate Gaussian distribution with mean zero and diagonal variance matrix $\Sigma$. Let $\sigma^2_1, \ldots, \sigma^2_d$ denote the entries along the diagonal of $\Sigma$, with $\sigma_1 \geq \sigma_2 \geq \cdots \geq \sigma_d > 0$.  If the link probability function is upper bounded by $1 -\epsilon$, then the LPM is regular with $G(n) =  \sqrt{2 \sigma_1^2 (1+c) \log(n)}$ for any $c > 0$.
\end{lemma}
\begin{proof}
Criteria 1 and 2 for regularity hold trivially. Thus, it is sufficient to prove criteria 3 for the prescribed $G(n)$. Let $Z_1, \ldots, Z_n$ denote the latent positions. Define $X_1, \ldots, X_n$ such that for $i \in \left\{1, \ldots, n\right\}$ and $j \in \left\{1, \ldots, d\right\}$, 
\begin{align}
X_{ij} := \frac{\sigma_{1}}{\sigma_{j}} Z_{ij}.
\end{align}
It follows from $\sigma_1 \geq \sigma_j$ that $\|X_{i}\| \geq \| Z_{i}\|$. Moreover, each $X_{i}$ is a mean zero isotropic Gaussian random vector with variance $\sigma_1^2$ in each dimension. Therefore, $\|X_i\|^2/\sigma_1^2$ follows a $\chi^2$ distribution with parameter $d$. We can apply the concentration inequality on $\chi^2$ random variables implied by \citet[Lemma 1]{Laurent-Massart-adaptive}, to conclude, for any $t > 0$
\begin{align}
\mathbb{P}\left(\|X_i\| > \sigma_1 \sqrt{d + 2t + 2 \sqrt{d t}} \right) &\leq \exp{(-t)}\\
\Rightarrow \mathbb{P}\left(\|X_i\| > \sqrt{2} \sigma_1 (u + \sqrt{d}) \right) &\leq \exp{(-u^2)}
\end{align}
for any $u > 0$. Applying the union bound results in
\begin{align}
\mathbb{P}\left(\max_{1 \leq i \leq n} \|X_i\| >  \sqrt{2} \sigma_1 (u + \sqrt{d})  \right) &\leq n\exp{(-u^2)}.
\end{align}
So long as $u^2 \geq (1+c)\log(n)$, for $c > 0$, the above probability goes to 0. Note that $ \sqrt{2 \sigma_1^2 (1+c) \log(n)}$ dominates $ \sqrt{2d \sigma_1^2}$ as $n$ grows.  Because $\|X_{i}\| \geq \| Z_{i}\|$, a choice of $G(n) =  \sqrt{2 \sigma_1^2 (1+c) \log(n)}$ yields the desired result for $c > 0$. 
\end{proof}

\begin{lemma}\label{symmetrization}
Symmetrization Lemma\\
Let
\begin{align}
\Omega &= \left\{X \in \mathbb{R}^{n \times d}: \|X_i\| \leq G(n) \forall i \in [n] \right\}
\end{align}
for $G(n) \in \mathbb{R}_+$. Let $L(x: Y^n)$ denote the log likelihood of the latent positions $x \in \Omega$ as defined in (\ref{thelikelihood}) for a link function $K$. Let $\bl(x) = L(x: Y^n) - L(\textbf{0}:Y^n)$ and $\E(\bl(x))$ denote its expectation. Then, for $h \geq 1$,
\begin{equation}
\begin{aligned}
&\E \left(\sup_{x \in  \Omega} |\bl(x) - \E(\bl(x))|^h \right) \\
\leq ~ & 2^h \E \left( \sup_{x \in  \Omega} \left|\sum_{j=1}^n \sum_{i=1}^n R_{ij}\left(Y^n_{ij}  \log \left(\frac{K(\delta^x_{ij})}{K(0)}\right) + (1 - Y^n_{ij})  \log \left(\frac{1- K(\delta^x_{ij})}{1- K(0)}\right) \right) \right|^h\right) 
\end{aligned}
\end{equation}
where $R$ denotes an array of independent Rademacher random variables and $\delta^x_{ij} = \|x_i - x_j\|$.
\end{lemma}
\begin{proof}
This proof follows the same argument of that of \citet[Lemma 6.3]{Ledoux-Talagrand}.
Let $\bl_{ij}(x)$ denote the contribution of $Y^n_{ij}$ to the standardized log likelihood. Thus, 
\begin{align}
\bl(x) &= \sum_{i=1}^n \sum_{j=1}^n \bl_{ij}(x) \text{ and}\\
\bl(x) - \E(\bl(x)) &= \sum_{i=1}^n \sum_{j=1}^n \ell_{ij}(x)
\end{align}
where each $\ell_{ij}(x) = \bl_{ij}(x) - \E(\bl_{ij}(x))$ is a zero mean random variable. For each $i,j$, let $\ell'_{ij}(x)$ denote a random variable that is independently drawn from the distribution of $\ell_{ij}(x)$. Then, $\ell_{ij}(x) - \ell'_{ij}(x)$ is a symmetric zero mean random variable with the same distribution as $R_{ij}(\ell_{ij}(x) - \ell'_{ij}(x))$. Moreover, we can view $\sup_{x \in  \Omega} |f(x)|$ as defining a norm on the Banach space of functions $f: \Omega \rightarrow \mathbb{R}$. These facts, along with the convexity of exponentiating by $h$, imply the following.

{\small
\begin{align}
&\E \left(\sup_{x \in  \Omega} \left|\bl(x) - \E(\bl(x))\right|^h \right)\\
 &= \E \left(\sup_{x \in  \Omega} \left|\sum_{i=1}^n \sum_{j=1}^n \ell_{ij}(x)\right|^h \right)\\
   &\leq \E \left(\sup_{x \in  \Omega} \left|\sum_{i=1}^n \sum_{j=1}^n \ell_{ij}(x) - \ell'_{ij}(x) \right|^h \right) \text{ (cf. \citep[Equation 2.5]{Ledoux-Talagrand})}\\
   &=   \E \left(\sup_{x \in  \Omega} \left|\sum_{i=1}^n \sum_{j=1}^nR_{ij}\left(\ell_{ij}(x) - \ell'_{ij}(x)\right) \right|^h \right)\\
   &=   \E \left( \left|\sup_{x \in  \Omega} \sum_{i=1}^n \sum_{j=1}^nR_{ij}\left(\bl_{ij}(x) - \bl'_{ij}(x)\right) \right|^h \right)\\
   &\leq   \E \left( \left(\sup_{x \in  \Omega}\left|\sum_{i=1}^n \sum_{j=1}^n R_{ij}\bl_{ij}(x)\right| + \sup_{x \in  \Omega}\left|\sum_{i=1}^n \sum_{j=1}^n R_{ij}\bl'_{ij}(x) \right|\right)^h \right)\\
   & \leq   \E \left( \frac{1}{2}\sup_{x \in  \Omega} \left|2\sum_{i=1}^n \sum_{j=1}^n R_{ij}\bl_{ij}(x)\right|^h + \frac{1}{2}\sup_{x \in  \Omega}\left|2\sum_{i=1}^n \sum_{j=1}^n R_{ij}\bl'_{ij}(x) \right|^h \right) \\
   & \text{ by convexity of exponentiating by $h$}\\
   &=  \frac{1}{2} \E \left( \sup_{x \in  \Omega} \left|2\sum_{i=1}^n \sum_{j=1}^n R_{ij}\bl_{ij}(x)\right|^h\right) + \frac{1}{2} \mathbb{E}\left(\sup_{x \in  \Omega}\left|2\sum_{i=1}^n \sum_{j=1}^n R_{ij}\bl'_{ij}(x) \right|^h \right) \\
   &= 2^h \E \left(\sup_{x \in  \Omega} \left|\sum_{i=1}^n \sum_{j=1}^nR_{ij}  \bl_{ij}(x) \right|^h \right).
\end{align}
The result follows from the definitions of the $\bl_{ij}$.
}
\end{proof}

\begin{lemma}\label{contraction}
Contraction Theorem \citep[Theorem 4.12]{Ledoux-Talagrand}.\\
 Let $F: \mathbb{R}_+ \rightarrow \mathbb{R}_+$ be convex and increasing. Let $\phi_i:\mathbb{R} \rightarrow \mathbb{R}$ for $i \leq N$ satisfy $\phi_i(0) = 0$ and $|\phi_i(s) - \phi_i(t)| \leq |s-t|$ for all $s,t \in \mathbb{R}$. Then, for any bounded subset $\Omega \subset \mathbb{R}$,
\begin{align}
\mathbb{E}\left(F \left(\frac{1}{2} \sup_{t \in \Omega^N} \left| \sum_{i=1}^N R_i \phi_{i}(t_i) \right| \right) \right) & \leq \mathbb{E}\left(F \left(\sup_{t \in \Omega^N} \left| \sum_{i=1}^N R_i t_i \right| \right) \right)
\end{align}
where $R_1, \ldots, R_N$ denote independent Rademacher random variables.
\end{lemma}

\begin{cor}\label{gettingalpha}
Let $R$ denote an $n\times n$ array of independent Rademacher random variables, $K: \mathbb{R}_+ \rightarrow [0, 1-\epsilon]$ denote a link function that satisfies the regularity criteria in \S\ref{criteria} (i.e. monotonically decreasing, differentiable function that is upper bounded by $1- \epsilon$ for some $\epsilon$), and 
\begin{align}
\Omega &= \left\{X \in \mathbb{R}^{n \times d}: \|X_i\| \leq G(n) \text{ for all } i \in [n] \right\}
\end{align}
with $G(n) \in \mathbb{R}_+$, and $Y^n \in \left\{0,1\right\}^{n \times n}$. Define $\alpha_n^K$ as in (\ref{alpha}). That is,
\begin{align}
\alpha^K_{n} = \sup_{0 \leq x \leq 2G(n)} \frac{|K'(x)|}{|x| K(x)\epsilon}.
\end{align}
Then,
\begin{align}
\E \left( \sup_{x \in  \Omega} \left|\sum_{j=1}^n \sum_{i=1}^n R_{ij}\left(Y^n_{ij}  \log \left(\frac{K(\delta^x_{ij})}{K(0)}\right) + (1 - Y^n_{ij})  \log \left(\frac{1- K(\delta^x_{ij})}{1- K(0)}\right) \right) \right|^h\right)\\
\leq (2\alpha^K_n)^h \E \left(\sup_{x \in  \Omega} \left(\left|\sum_{j=1}^n \sum_{i=1}^n R_{ij}\|x_i - x_j\|^2 \right|^h\right) \right)
\end{align}
for $h \geq 1$, where $\delta^x_{ij} = \|x_i - x_j\|$.
\end{cor}
\begin{proof}
We can apply Lemma~\ref{contraction} to obtain this result as follows.

For all $x \in \Omega$, $i,j \in [n]$, we know by the triangle inequality that $\|x_i - x_j\|^2 \leq 4G(n)^2$. Moreover, $K(2G(n)) \leq K(\|x_i - x_j\|) \leq 1 - \epsilon$ because $K$ is regular. A Taylor expansion of $\log(K(\sqrt{\cdot}))$ around $0$ reveals that 
\begin{align}
& \log(K(\sqrt{u})) - \log(K(\sqrt{0}))\\ &= \frac{u K'(\sqrt{w})}{2\sqrt{w} K(\sqrt{w})} \text{ for some }w\in[0,4G(n)^2]\\
&= \frac{u K'(v)}{2 v K(v)} \text{ for some }v\in[0,2G(n)].
\end{align}
Taking the supremum over possible values of $v$, it follows that
\begin{align}
\left| \frac{\log(K(\sqrt{u})) - \log(K(\sqrt{0}))}{\alpha^K_n} \right| &\leq u.
\end{align}
Similarly, a Taylor expansion of $\log(1 -K(\sqrt{\cdot}))$ around $0$ yields
\begin{align}
& \log(1-K(\sqrt{u})) - \log(1-K(\sqrt{0}))\\ &= \frac{-u K'(\sqrt{w})}{2v(1-K(\sqrt{w}))} \text{ for some }w\in[0,4G(n)^2] \\
 &= \frac{-u K'(v)}{2v(1-K(v))} \text{ for some }v\in[0,2G(n)].
\end{align}
Similarly, taking the supremum over possible values of $v$ yields
\begin{align}
\left| \frac{\log(1-K(\sqrt{u})) - \log(1-K(\sqrt{0}))}{\alpha^K_n}\right| &\leq u.
\end{align}
Together, we have
\begin{align}
\left| \frac{Y^n_{ij} \log \left(\frac{K(u)}{K(0)}\right) + (1 - Y^n_{ij})  \log \left(\frac{1- K(u)}{1- K(0)}\right) }{\alpha^K_n}\right| & \leq u.
\end{align}
Moreover, for any $i,j$, the function on the lefthand side is 0 at $u=0$. Thus, the function meets the criteria required of the $\phi$ functions in Lemma~\ref{contraction} and the result follows from convexity of exponentiating by $h$.
\end{proof}

\begin{lemma}\label{daviskahan}
Let $\Sigma, \hat{\Sigma} \in \mathbb{R}^{n \times n}$ be symmetric, with eigenvalues $\lambda_1  \geq \lambda_2 \geq \cdots \geq \lambda_n$ and $\hat{\lambda}_1 \geq \cdots \geq \hat{\lambda}_n$, respectively. Fix $1 \leq r \leq s \leq n$ and assume that $\text{min}\left(\lambda_{r-1} - \lambda_{r}, \lambda_{s} - \lambda_{s + 1}\right) > 0$ where $\lambda_0 := \infty$ and $\lambda_{n + 1} := - \infty$. Let $q := s - r + 1$. Let $V, \hat{V} \in \mathbb{R}^{n \times q}$ have orthonormal columns satisfying $\Sigma V_j = \lambda_j V_j$ and $\hat{\Sigma} \hat{V}_j = \hat{\lambda}_j  \hat{V}_j $ for $j \in \left\{1, \ldots, q\right\}$. Then, there exists an orthogonal matrix $O \in \mathbb{R}^{q \times q}$ such that
\begin{align}
\| \hat{V} O - V\|^2_F & \leq \frac{2^{3} \| \hat{\Sigma} - \Sigma \|^2_F}{\text{min}\left(\lambda_{r-1} - \lambda_{r}, \lambda_{s} - \lambda_{s + 1}\right)^2}.
\end{align}
\end{lemma}
\begin{proof}
This follows from the Davis-Kahan Theorem \citep[Theorem 2]{Yu-Wang-Samworth-davis-kahan}.
\end{proof}

\subsection{Projectivity Proofs}\label{projectivityproofs}

\subsubsection{Proof of Proposition~\ref{exprojective}}\label{exprojectiveproof}
\begin{proof}
Let $Y^{n_1}$ and $Y^{n_2}$ denote random graphs with $n_1$ and $n_2$ nodes ($n_1 < n_2$) generated according to an exchangeable LPM, and let $\mathbb{P}^{n_1}$ and $\mathbb{P}^{n_2}$ be their corresponding distributions. Let $Z^j_{i}$ denote the random latent position of node $i$ in $Y^{j}$ for $j=n_1,n_2$. By definition, $Z^{n_1}_{i}$ and $Z^{n_2}_{i}$ are iid draws from the same distribution $f$ on $S$. Thus, the $(Z^{j}_{i})_{i=1\ldots n_1}$ have identical distributions for each $j$. As a result, $K(\rho(Z^{n_1}_{i_1}, Z^{n_1}_{i_1}))$ has the same distribution as $K(\rho(Z^{n_1}_{i_2}, Z^{n_1}_{i_2}))$ for any $1 \leq i_1, i_2 \leq n_1$. Because the distributions for each dyad coincide, the distributions over adjacency matrices coincide.
\end{proof}

\subsubsection{Proof of Proposition~\ref{projectiveRCM}} \label{projectiveRCMproof}
\begin{proof}
 Let $Y^{n_1}$ and $Y^{n_2}$ denote random graphs distributed with $n_1$ and $n_2$ nodes ($n_1 < n_2$) obtained by the finite window approach on the Poisson random connection model on $\mathbb{R}_+$, and let $\mathbb{P}^{n_1}$ and $\mathbb{P}^{n_2}$ be their corresponding distributions. Let $Z^{j}_i$ denote the random latent position of node $i$ in $Y^{j}$ for $j=n_1,n_2$. For both cases, the random variables $Z^{j}_1 - 0$, $Z^{j}_2 - Z^{j}_1$, \ldots, $Z^{j}_{n_1} - Z^j_{n_1 - 1}$ are iid exponential random variables, by the interval theorem for point processes \citep[p. 39]{Kingman-Poisson}. Thus, the $(Z^{j}_i)_{i=1\ldots n_1}$ have identical distributions for each $j$. The rest follows identically as for Proposition~\ref{exprojective}.
\end{proof}

\subsubsection{Proof of Theorem~\ref{projective}}\label{projectiveproof}
\begin{proof}
Let $Y^{n_1}$ and $Y^{n_2}$ denote random graphs distributed with $n_1$ and $n_2$ nodes ($n_1 < n_2$) obtained from a rectangular LPM on $\mathbb{R}_+^d$. Let $\mathbb{P}^{n_1}$ and $\mathbb{P}^{n_2}$ be their corresponding distributions. Let $t_{i}^j,$ denote the arrival time for the $i$th node in $Y^j$ for $j=n_1,n_2$. Following, Lemma~\ref{waitingtimes}, both $t_{i}^{n_1}$ and $t_i^{n_2}$ are equally distributed. Therefore, $Z_{i}^{n_1}$ and $Z_{i}^{n_2}$ must also be equally distributed. The rest follows as in the proofs for Proposition~\ref{exprojective}.
\end{proof}

\subsubsection{Proof of Proposition~\ref{notprojective}}\label{notprojectiveproof}
\begin{proof}
Suppose $(s_n)_{n=1\ldots \infty}$ is not constant. Then there is an $n_2 > n_1 \geq 2$ such that $s_n \neq s_{n_2}$. Let $Y^{n_1}$ and $Y^{n_2}$ denote random graphs with $n_1$ and $n_2$ nodes. Notice that the marginal distribution of $Y^{n}_{12}$ in a graph with $n$ nodes is given by
\begin{align}
\mathbb{P}^{n}(Y_{12}=1) &= \mathbb{E}\left(\mathbb{P}(Y^n_{12}=1|Z_1, Z_2)\right)\\
                                      & =  \mathbb{E}\left(s_{n}K\left(\rho(Z_1, Z_2) \right) \right)\\
                                      &= s_{n} \mathbb{E}\left(K\left(\rho(Z_1, Z_2) \right) \right).
\end{align}
Clearly,  $\mathbb{P}^{n_1}(Y_{12}=1)  \neq \mathbb{P}^{n_2}(Y_{12}=1)$ because $s_{n_1} \neq s_{n_2}$ and $Z_1,Z_2 \sim f$ independently of $k$. Thus the model cannot be projective.
\end{proof}

\subsection{Sparsity Proofs} \label{sparsityproofs}

\subsubsection{Proof of Proposition~\ref{exsparse}} \label{exsparseproof}
\begin{proof}
Let $n$ be the number of nodes in the latent position network model. Then the expected number of edges $\sum_{i=1}^n \sum_{j=1}^n Y_{ij}$ is given by
\begin{align}
\mathbb{E}\left(\frac{\sum_{i=1}^n \sum_{j=1}^n Y_{ij}}{n^2}\right) &= \frac{1}{n^2}\sum_{i=1}^n \sum_{j=1}^n \mathbb{E}( \mathbb{E}(Y_{ij}| Z_i, Z_j ))\\
                                    &= \frac{1}{n^2} \sum_{i=1}^n \sum_{j=1}^n  \mathbb{E} K(\rho(Z_i, Z_j))\\
                                    &= \mathbb{E} K(\rho(Z_i, Z_j))
\end{align}
where $\mathbb{E} K(\rho(Z_i, Z_j))$ is constant due to $Z_i$ being independent and identically distributed. Thus, as long as the network is not empty, it is dense.
\end{proof}

\subsubsection{Proof of Proposition~\ref{sparseRCM}} \label{sparseRCMproof}
\begin{proof}
A special case of Theorem~\ref{sparsity} with $d=1$ and $g(t) = t$.
\end{proof} 

\subsubsection{Proof of Theorem~\ref{sparsity}} \label{sparsityproof}
\begin{proof}
 Let $\pi$ be a permutation on $(1, \ldots, n)$ chosen uniformly at random from the set of permutations on $(1, \ldots, n)$. Then, 
 \begin{align}  
\sum_{i=1}^{n} \sum_{j=1}^{n} Y_{ij} &= \sum_{i=1}^{n} \sum_{j=1}^{n} Y_{\pi(i)\pi(j)}.
\end{align}
Let $\Omega = [-g(t_{n+1}), g(t_{n+1})]^d$. By Lemma~\ref{uniformZ},
\begin{align} 
\mathbb{E}(Y_{\pi(i)\pi(j)}|t_{n+1}) &= \int_{z,z' \in \Omega} K(\|z - z'\|) \frac{1}{2^dg(t_{n+1})^d} \text{d}z'  \frac{1}{2^dg(t_{n+1})^d} \text{d}z \\
                                                  &\leq \frac{1}{4^d g(t_{n+1})^{2d}} \int_{z \in \Omega} C \text{d}z\\
                                                  &\propto \frac{1}{2^dg(t_{n+1})^{d}} 
\end{align}
for some $C \in \mathbb{R}_+$ by Lemma~\ref{distanceswitch}. Similarly,
\begin{align} 
\mathbb{E}(Y_{\pi(i)\pi(j)}|t_{n+1}) &= \int_{z,z' \in \Omega} K(\|z - z'\|) \frac{1}{2^dg(t_{n+1})^d} \text{d}z'  \frac{1}{2^dg(t_{n+1})^d} \text{d}z \\
                                                  &\geq \frac{C'}{2^dg(t_{n+1})^{d}} 
\end{align}
for some $C' \in \mathbb{R}_+$ by Lemma~\ref{distanceswitch}. Thus,
\begin{align}
\mathbb{E}\left(\frac{g(n)^d}{n^2} \sum_{i=1}^n \sum_{j=1}^n Y_{ij} \mid t_{n+1} \right) &= \mathbb{E}\left(\frac{g(n)^d}{n^2} \sum_{i=1}^n \sum_{j=1}^n Y_{\pi(i)\pi(j)} \mid t_{n+1} \right)\\
&= g(n)^d \mathbb{E}(Y_{\pi(i)\pi(j)}| t_{n+1})\\
&\propto \frac{g(n)^d}{g(t_{n+1})^d}.
\end{align}
We can analytically integrate over possible $t_{n+1}$ because $t_{n+1}$ follows Gamma($n+1$, 1), as given by Lemma~\ref{waitingtimes}.
\begin{align}
\mathbb{E}\left(\frac{g(n)^d}{n^2} \sum_{i=1}^n \sum_{j=1}^n Y_{ij} \right) &= \mathbb{E}\left(\mathbb{E}\left(\frac{g(n)^d}{n^2} \sum_{i=1}^n \sum_{j=1}^n Y_{ij}\mid t_{n+1} \right) \right)\\
&\propto \mathbb{E}\left(\frac{g(n)^d}{g(t_{n+1})^d}\right)\\
&= n^{p}\int_{0}^{\infty} t^{-p} \frac{1}{\Gamma(n+1)} t^{n} \exp{(-t)}  \text{d}t\\
&= n^{p} \frac{\Gamma(n-p+1)}{\Gamma(n+1)}
\end{align}
which converges to one as $n$ goes to infinity.
\end{proof}

\subsubsection{Proof of Proposition~\ref{sparsegraphonissparse}}\label{sparsegraphonissparseproof}

\begin{proof}
Let $n$ be the number of nodes in the LPM. Then the expected number of edges $\sum_{i=1}^n \sum_{j\neq i} Y_{ij}$ is given by
\begin{align}
\mathbb{E}\left(\sum_{i=1}^n \sum_{j\neq i}  Y_{ij}\right) &= \sum_{i=1}^n \sum_{j\neq i} \mathbb{E}( \mathbb{E}(Y_{ij}| Z_i, Z_j ))\\
                                    &= \sum_{i=1}^n \sum_{j\neq i}  \mathbb{E} \left(K_n(\rho(Z_i, Z_j))\right)\\
                                    &= n(n-1) \mathbb{E} \left(K_n(\rho(Z_i, Z_j))\right)
\end{align}
where $\mathbb{E} \left(K_n(\rho(Z_i, Z_j))\right)$ takes the same value for all $i\neq j$ because the $Z_i$ are independent and identically distributed. By definition of $K_n$, we have that
\begin{align}
s_n K(\rho(Z_i, Z_j)) \geq K_n(\rho(Z_i, Z_j)) \geq  s_n s_1^{-1} K_1(\rho(Z_i, Z_j)).
\end{align}
Therefore,
\begin{align}                                                            
\mathbb{E} \left(K_n(\rho(Z_i, Z_j))   \right)                               &\leq s_n \mathbb{E} (K(\rho(Z_i, Z_j)))\\
\mathbb{E} \left(K_n(\rho(Z_i, Z_j))  \right)                                &\geq s_n s_1^{-1} \mathbb{E} (K_1(\rho(Z_i, Z_j))).
\end{align}
Since both $\mathbb{E} (K(\rho(Z_i, Z_j))$ and $s_1^{-1}\mathbb{E} (K_1(\rho(Z_i, Z_j)))$ are constants that are independent of $n$, the expected number of edges must be of order $s_n n^2$.
\end{proof}

\subsection{Learnability Proofs}\label{consistencyproofs}

\subsubsection{Proof of Lemma~\ref{consistentprob}}\label{proofconsistentprob}

Much of the argument provided here can be viewed specialization of the results established in \citet[Theorem 6]{Davenport-et-al-1-bit}. For clarity, we include the entirety of the argument, illustrating our non-standard choices for many of the components, as well as some small differences such as using a restricted maximum likelihood estimator. 

The notation for our proofs is simplified by working with the following standardized version of the likelihood
\begin{align}
\bl(z: Y^n) &= L(z: Y^n) - L(z=\textbf{0}: Y^n)\\
    &= \sum_{i=1}^n \sum_{j =1}^n Y^n_{ij}  \log \left(\frac{K(\delta^z_{ij})}{K(0)}\right) + (1- Y^n_{ij}) \log \left(\frac{1- K(\delta^z_{ij})}{1 - K(0)}\right)
\end{align}
where $\delta^z_{ij} = \|z_i - z_j\|$.
Note that the standardized likelihood and non-standardized version of the likelihood are maximized by the same value of $z$ for a given $Y^n$. Going forward, we use the shorthand $\bl(z)$; $Y^n$ is implied.

In order to establish concentration of $\|P^{\hat{z}(Y^n)} - P^{z}\|_F^2$, we first establish concentration of related quantities. Specifically, Lemma~\ref{boundingkl} establishes concentration of $\text{KL}(P^{\hat{z}(Y^n)}, P^{z})$. Here, KL($P, Q$) denotes the Kullback-Leibler divergence \citep{Cover-and-Thomas-2nd} between two link probability matrices $P$ and $Q$. It is a non-negative and given by
\begin{align}
\text{KL}(P, Q) &= \sum_{i=1}^n \sum_{j=1}^n \log{\left(\frac{P_{ij}}{Q_{ij}} \right)} + (1-P_{ij}) \log{\left(\frac{1 - P_{ij}}{ 1- Q_{ij}} \right)}.
\end{align}

\begin{lemma}\label{boundingkl}
Consider a sequence of adjacency matrices $Y^n$ generated by a LPM meeting the regularity criteria provided in \S\ref{criteria}. Further assume that the true latent positions are within 
\begin{align}
\Omega &= \left\{X \in \mathbb{R}^{n \times d}: \|X_i\| \leq G(n) \right\}
\end{align}
where $X^i$ denotes the $i$th row of $X$. Let $P^{\hat{Z}(Y^n)}$ denote the estimated link probability matrix obtained via $\hat{Z}(Y^n)$ from (\ref{mle}). Then,
\begin{align}
\mathbb{P}\left(\text{KL}(P^{\hat{z}(Y^n)}, P^{z}) \geq 16 e \alpha^K_n G(n)^2 n^{1.5} (d + 2) \right) \leq \frac{C}{n^2}
\end{align}
for some $C > 0$.
\end{lemma}

\begin{proof}
Note that for any $z_0 \in  \Omega$, we have
\begin{align}
\bl(z_0) - \bl(z) &= \E(\bl(z_0) - \bl(z)) + \bl(z_0) - \E(\bl(z_0)) - (\bl(z) - \E(\bl(z)))\\
                        &\leq \E(\bl(z_0) - \bl(z))  + | \bl(z_0) - \E(\bl(z_0))| + |(\bl(z) - \E(\bl(z)))|\\
                        & \leq \E(\bl(z_0) - \bl(z)) + 2 \sup_{x \in \Omega} |(\bl(x) - \E(\bl(x)))|\\
                        & \leq -\text{KL}(P^{z_0}, P^{z}) + 2 \sup_{x \in   \Omega} |(\bl(x) - \E(\bl(x)))|.
\end{align}
Let $z_0 = \hat{Z}(Y^n)$ denote the maximum likelihood estimator given in (\ref{mle}). Then, because $\bl(z_0) - \bl(z) \geq 0$,
\begin{align}
\text{KL}(P^{\hat{z}(Y^n)}, P^{z}) & \leq 2 \sup_{x \in  \Omega} |(\bl(x) - \E(\bl(x)))|.
\end{align}
So we can upper bound KL($P^{\hat{z}(Y_n)}, P^{z}$) by bounding
\[\sup_{x \in  \Omega} |(\bl(x) - \E(\bl(x)))|.\] 
Let $h$ be an arbitrary positive integer (we will later let it be $2\log(n)$). Applying the Markov inequality for $\sup_{x \in  \Omega} |(\bl(x) - \E(\bl(x)))|^h$ yields:
\begin{align}
\PP \left(\sup_{x \in  \Omega} |(\bl(x) - \E(\bl(x)))|^h > c(n)^h\right) &\leq \frac{\E (\sup_{x \in  \Omega} |(\bl(x) - \E(\bl(x)))|^h) }{c(n)^h}
\end{align}
for a positive function $c: \mathbb{N} \rightarrow \mathbb{R}_+$. To bound the expectation, we use a symmetrization argument (provided as Lemma~\ref{symmetrization} in \S\ref{intermed}) followed by a contraction argument (stated as Corollary~\ref{gettingalpha} in \S \ref{intermed}).

\begin{align}
& \hspace{0.5cm }\E \left(\sup_{x \in  \Omega} |(\bl(x) - \E(\bl(x)))|^h \right) \\
&\leq 2^h \E \left( \sup_{x \in  \Omega} \left|\sum_{j=1}^n \sum_{i=1}^n R_{ij}\left(Y^n_{ij}  \log \left(\frac{K(\delta^z_{ij})}{K(0)}\right) + (1 - Y^n_{ij})  \log \left(\frac{1- K(\delta^z_{ij})}{1- K(0)}\right) \right) \right|^h\right)\\
& \text{ by Lemma~\ref{symmetrization}}\\
&\leq (4\alpha^K_n)^h \E \left(\sup_{x \in  \Omega} \left(\left|\sum_{j=1}^n \sum_{i=1}^n R_{ij}\|x^i - x^j\|^2 \right|^h\right) \right) \text{ by Corollary~\ref{gettingalpha}.}\\
&= (4\alpha^K_n)^h \E\left( \sup_{x \in  \Omega} \left| \langle R, D^x \rangle\right|^h \right)
\end{align}
where $R = (R_{ij})$ is a matrix of independent Rademacher random variables, $D^x$ denotes the matrix of squared distances implied by $x$, and $\alpha^K_n$ is defined as in (\ref{alpha}). 

Let $\|\cdot\|_o$ denote the operator norm and $\| \cdot\|_{*}$ denote the nuclear norm. To bound $\E\left( \sup_{x \in  \Omega} \langle R,D^x \rangle^h \right)$, we make use of the fact that $\left|\langle A, B \rangle\right| \leq \|A\|_o \|B\|_{*}$. Then,
\begin{align}
\E\left( \sup_{x \in  \Omega} \left| \langle R, D^x \rangle\right|^h \right) &\leq \E\left(\sup_{\Omega} \|R\|_o^h \|D^x\|^h_{*} \right)\\
&= \E(\|R\|_o^h)\sup_{x \in  \Omega}  \|D^x\|_{*}^h\\
 &\leq  C n^{h/2} \sup_{x \in  \Omega}  \|D^x\|_{*}^h
 \end{align}
 where $\E(\|R\|_o^h)$ was bounded using \citet[Theorem 1.1]{Seginer-expected-norm-matrices} and $C > 0$ is a constant provided that $h \leq 2 \log(n)$. 
 
Recall that the rank of a squared distance matrix $D^x$ is at most $d+2$ where $d$ is the dimension of the positions $x$. Moreover, each of the eigenvalues of $D^x$ must be upper bounded by the product of maximum distance in $D^x$ and $n$ (where $n$ is the number of points). For $x \in  \Omega$, the maximum entry in $D^x$ is at most $4G(n)^2$. Thus, $\sup_{x \in  \Omega}  \|D^x\|_{*} \leq (d+2) 4nG(n)^2$. Therefore, 
\begin{align}
 \E\left( \sup_{x \in  \Omega} \left| \langle E, D^x \rangle\right|^h \right)  &\leq C n^{3h/2}  (2G(n))^{2h} (d + 2)^{h}.
\end{align}

Combining the above results yields
\begin{align}
\PP\left(\text{KL}(P^{\hat{z}(Y_n)}, P^{z}) \geq c(n)\right) &\leq \frac{2^{4h} C n^{3h/2}(\alpha_n^K)^h  G(n)^{2h}  (d + 2)^{h} }{c(n)^h}. 
\end{align}
Let $c(n) = 2^4 C_0 \alpha^K_n G(n)^2 (d+2) n^{3/2}$ for some constant $C_0$. Then, by letting $h = 2\log(n)$, we get
\begin{align}
\PP\left(\text{KL}(P^{\hat{z}(Y_n)}, P^{z}) \geq c(n)\right) &\leq C C_0^{-h}\\
&= C C_0^{-2\log(n)}\\
&= \frac{C }{n^{2 \log(C_0)}}
\end{align}
The result follows from letting $C_0 = e$.
\end{proof}

We can now leverage Lemma~\ref{boundingkl} into Corollary~\ref{boundHellinger}, a concentration bound on the squared Hellinger distance $d^2_H(P^{\hat{z}(Y^n)}, P^{z})$. Here, $d^2_H(P, Q)$ denotes the squared Hellinger distance between two link probability matrices $P$ and $Q$ given by
\begin{align}
d^2_H(P, Q) &= \sum_{i=1}^n \sum_{j=1}^n (\sqrt{P_{ij}} - \sqrt{Q_{ij}})^2 + (\sqrt{1 - P_{ij}} - \sqrt{1 - Q_{ij}})^2.
\end{align}

\begin{cor}\label{boundHellinger}
Consider a sequence adjacency matrices $Y^n$ generated by a LPM meeting the criteria provided in Section~\ref{criteria}. Further assume that the true latent positions are within 
\begin{align}
\Omega &= \left\{X \in \mathbb{R}^{n \times d}: \|X_i\| \leq G(n) \right\}.
\end{align}
Let $P^{\hat{Z}(Y^n)}$ denote the estimated link probability matrix obtained via $\hat{Z}(Y^n)$ from (\ref{mle}). Then,
\begin{align}\label{hel}
\mathbb{P}\left(d^2_H(P^{\hat{z}(Y^n)}, P^{z}) \geq 16 e \alpha^K_n G(n)^2 n^{1.5} (d + 2) \right) \leq \frac{C}{n^2}
\end{align}
\end{cor}
\begin{proof}
(\ref{hel}) Follows from Lemma~\ref{boundingkl} and the fact that Kullback-Leibler divergence upper bounds the squared Hellinger distance \citep{Gibbs-Su-probability-metrics}. 
\end{proof}

Finally, the Frobenius norm $\|P - Q\|_F^2$ between $P$ and $Q$ is upper bounded by the squared Hellinger distance. 

We can thus proceed with our proof of Lemma~\ref{consistentprob}.
\begin{proof}
The result follows from Corollary~\ref{boundHellinger} because the squared Hellinger distance between two link probability matrices upper bounds the squared Frobenius norm between them. This follows from the fact that $u \leq \sqrt{u}$ for all $u \in [0,1]$.
\end{proof}

Note that, rather than immediately upper bounding $\|P - Q\|_F^2$ by $\text{KL}(P, Q)$, we introduce $d^2_H(P, Q)$ in Corollary~\ref{boundHellinger} due to its utility in proving Lemma~\ref{consistentdist}.

\subsubsection{Proof of Lemma~\ref{consistentdist}}\label{proofconsistentdist}

\begin{proof}
Let $\hat{d}_{ij}$ and $d_{ij}$ denote the $(i,j)$th entry of $D^{\hat{Z}(Y^n)}$ and $D^Z$ respectively. Then, $d^2_H(P^{\hat{Z}(Y^n)}, P^Z) = $ 
\begin{align}
 \sum_{i=1}^n \sum_{j=1}^n \left(K(\sqrt{d_{ij}})^{1/2} - K(\sqrt{\hat{d}_{ij}})^{1/2}\right)^2 + \left((1-K(\sqrt{d_{ij}}))^{1/2} - (1-K(\sqrt{\hat{d}_{ij}}))^{1/2}\right)^2 
\end{align}
can be lower-bounded as follows. 

Notice that for any $a,b \in \mathbb{R}$, $a^2 + b^2 \geq \frac{1}{2}(a-b)^2$. Therefore, $d^2_H(P^{\hat{Z}(Y^n)}, P^Z) \geq $ 
\begin{align}
&\sum_{i=1}^n \sum_{j=1}^n \frac{1}{2}\left(\left(K(\sqrt{d_{ij}})^{1/2} - K(\sqrt{\hat{d}_{ij}})^{1/2}\right) - \left(\left(1-K(\sqrt{d_{ij}})\right)^{1/2} - \left(1-K(\sqrt{\hat{d}_{ij}})\right)^{1/2}\right)\right)^2\\
&= \sum_{i=1}^n \sum_{j=1}^n \frac{1}{2}\left(\left(K(\sqrt{d_{ij}})^{1/2} -  \left(1-K(\sqrt{d_{ij}})\right)^{1/2} \right) - \left( K(\sqrt{\hat{d}_{ij}})^{1/2} - \left(1-K(\sqrt{\hat{d}_{ij}})\right)^{1/2}\right)\right)^2.\label{fortaylor}
\end{align}

Let $\gamma(t) = \sqrt{K(\sqrt{t})} - \sqrt{1- K(\sqrt{t})}$. Taylor expanding $\gamma(t)$ around $t = d_{ij}$ reveals that 
\begin{align}
\gamma\left(\hat{d}_{ij}\right) &=  \gamma(d_{ij}) + \gamma'(v) (\hat{d}_{ij} - d_{ij}) \text{ for some $v \in [0,4G(n)^2]$}
\end{align}
with
\begin{align}
\gamma'(v) &= \frac{K'(\sqrt{v})}{4\sqrt{v}} \left(\frac{1}{\sqrt{K(\sqrt{v})}} + \frac{1}{\sqrt{1 - K(\sqrt{v})}} \right)
\end{align}
and $K'(v)$ denoting the derivatives of $\gamma$ and $K$ with respect to $v$, respectively. Noting that 
\begin{align}
|\gamma'(v)| &\leq  \frac{|K'(\sqrt{v})|}{4\sqrt{v}} \left(\frac{1}{\sqrt{K(\sqrt{v}) \left(1 - K(\sqrt{v})\right)}} \right),
\end{align}
we can combine these results with (\ref{fortaylor}) to obtain a bound:
\begin{align}
d^2_H(P^{\hat{Z}(Y^n)}, P^Z)  &\geq \sum_{i=1}^n \sum_{j=1}^n \frac{1}{2} \inf_{v \in [0,4G(n)^2]} \gamma'(v)^2 (\hat{d}_{ij} - d_{ij})^2 \\
&= \frac{1}{2} \left(\inf_{v \in [0,4G(n)^2]} \gamma'(v)^2 \right)  \|D^{\hat{Z}(Y^n)} - D^Z \|_F^2\\
&\geq \frac{1}{32} \left(\inf_{\theta \in [0,2G(n)]} \frac{K'(\theta)^2}{\theta^2 K(\theta)(1-K(\theta))}\right) \|D^{\hat{Z}(Y^n)} - D^Z \|_F^2\\
&= \frac{1}{32} \left(\sup_{\theta \in [0,2G(n)]} \left( \frac{\theta^2 K(\theta)(1-K(\theta))}{K'(\theta)^2}\right)\right)^{-1} \|D^{\hat{Z}(Y^n)} - D^Z \|_F^2\\
&\geq \frac{1}{32 \beta^K_n} \|D^{\hat{Z}(Y^n)} - D^Z \|_F^2  \end{align}
where $\beta^K_n$ is defined as in (\ref{beta}). Combining this inequality with Corollary~\ref{boundHellinger} yields
\begin{align}
\mathbb{P}\left(d^2_H(P_{\hat{z}}, P_{z}) \geq 16 e \alpha^K_n G(n)^2 n^{1.5} (d + 2) \right) &\leq \frac{C}{n^2} \\
 \Rightarrow \mathbb{P}\left(\frac{1}{32\beta^K_{n}} \|D^{\hat{Z}(Y^n)} - D^Z \|_F^2 \geq 16 e \alpha^K_n G(n)^2 n^{1.5} (d + 2) \right) &\leq \frac{C}{n^2}\\
  \Rightarrow \mathbb{P}\left(\|D^{\hat{Z}(Y^n)} - D^Z \|_F^2 \geq 512e\beta^K_{n} \alpha^K_n G(n)^2 n^{1.5} (d + 2) \right) &\leq \frac{C}{n^2}.
\end{align}
\end{proof}                

\subsubsection{Proof of Lemma~\ref{consistentpositions}}\label{proofconsistentpositions}

Before proving Lemma~\ref{consistentpositions}, it is useful to first summarize our general strategy and introduce some notation. Our proof involves translating our concentration inequality for the latent squared distances (provided as Lemma~\ref{consistentdist}) to an analogous one for the latent positions. To do this, we combine results from classical multidimensional scaling \citep[Chapter 12]{Borg-Groenen-MDS}, Weyl's inequality \citep{Horn-Johnson-matrix-analysis}, and the Davis-Kahan theorem (Lemma~\ref{daviskahan}). 

Classical multi-dimensional scaling recovers a set of positions $Z \in \mathbb{R}^{n \times d}$ corresponding to a squared distance matrix $D \in \mathbb{R}^{n \times n}$. It does so through an eigendecomposition of the double centered distance matrix $-0.5 \mathcal{C}_n D \mathcal{C}_n$ with $\mathcal{C}_n$ defined in (\ref{doublecentering}). Here,
\begin{align}
Z &= V \Lambda^{1/2} 
\end{align}
is used to denote the recovered positions with $\Lambda \in \mathbb{R}^{d\times d}$ denoting a diagonal matrix consisting of the $d$ nonzero eigenvalues of $-0.5 \mathcal{C}_n D \mathcal{C}_n$ and $V \in \mathbb{R}^{n \times d}$ denoting a matrix with columns comprised of the corresponding eigenvectors. This technique is guaranteed to recover $Z$ exactly (up to translations, rotations and reflections).

Multidimensional scaling of both $D^Z$ and $D^{\hat{Z}(Y^n)}$ recovers the versions of the true latent positions and maximum likelihood estimates
\begin{align}
Z &= V \Lambda^{1/2} \\
\hat{Z}(Y^n) &= \hat{V} \hat{\Lambda}^{1/2}.
\end{align}
We resolve the identifiability issues (stemming from translations, rotations, and reflections) of these objects by minimizing over $O \in \mathcal{O}_p$ (rotations and reflections) and $Q \in \mathcal{Q}_{nd}$ (translations). We avoid explicitly minimizing over translations in our proof by considering the centered estimates obtain from multi-dimensional scaling. That is, the versions of $Z$ and $\hat{Z}(Y^n)$ obtained from multi-dimensional scaling end up being sufficient and
\begin{align}
\inf_{\substack{O \in \mathcal{O}_{d}\\ Q \in \mathcal{Q}_{nd}}} \|\hat{Z}(Y^n)O - Z - Q\|^2_F & \leq \inf_{O \in \mathcal{O}_{d}} \|\hat{V} \hat{\Lambda}^{1/2} O - V \Lambda^{1/2}\|^2_F.
\end{align}

A few properties of orthonormal matrices are useful in our proof. By definition, the columns of $\hat{V}$ and $V$ have norm 1. Multiplying a matrix by $O$, $V$, or $\hat{V}$ does not modify its Frobenius norm because their columns are orthonormal. Similarly, centering a matrix cannot increase its Frobenius norm, so pre-multiplying or post-multiplying by $\mathcal{C}_n$ does not increase the Frobenius norm. 

Finally, the following bits of notation are useful for keeping the proof succinct. For any $m \in \mathbb{N}$, we let $[m] = \left\{ 1, 2, \ldots, m\right\}$. Also, the \emph{direct sum} of matrices $M_1$ and $M_2$ is defined as
\begin{align}
M_1 \oplus M_2 &:= \left[\begin{array}{cc}
M_1 & 0 \\
0 & M_2
\end{array}\right].
\end{align}

We can now proceed with the details of the proof. Let $Z = V \Lambda^{1/2}$ and $\hat{Z} = \hat{V} \hat{\Lambda}^{1/2}$ be obtained from multidimensional scaling on $D^Z$ and $D^{\hat{Z}(Y^n)}$. Recall that the diagonal entries in $\Lambda$ and $\hat{\Lambda}$ are arranged in decreasing order such that $\lambda_1 \geq \lambda_2 \geq \cdots \geq \lambda_d > 0$ with the same decreasing structure for $\hat{\lambda}$ (but allowing for $\hat{\lambda}_d = 0$). Let $O \in \mathcal{O}_p$ denote a generic orthogonal matrix. First, we note that for any $O \in \mathcal{O}_p$,
\begin{align}
 \|\hat{Z}O - Z\|^2_F &= \| \hat{V} \hat{\Lambda}^{1/2}O - V \Lambda^{1/2} \|^2_F \\
                                     &= \| \hat{V}\hat{\Lambda}^{1/2}O - \hat{V} \Lambda^{1/2}O + \hat{V} \Lambda^{1/2}O  - V\Lambda^{1/2} \|_F^2\\
                                     &\leq 2\| \hat{V}\hat{\Lambda}^{1/2}O - \hat{V} \Lambda^{1/2}O\|_F^2 + 2\| \hat{V} \Lambda^{1/2}O  - V\Lambda^{1/2} \|_F^2\\
                                     &= 2 \|\hat{V} (\hat{\Lambda}^{1/2} - \Lambda^{1/2}) O\|^2_F + 2\| \hat{V} \Lambda^{1/2}O  - V\Lambda^{1/2} \|_F^2\\
                                     &= 2 \|\hat{\Lambda}^{1/2} - \Lambda^{1/2}\|^2_F + 2\| \hat{V} \Lambda^{1/2}O  - V\Lambda^{1/2} \|_F^2  \label{subback}
                                     \end{align}

Furthermore, for an arbitrary matrix $\Phi \in \mathbb{R}^{d \times d}$, we have
\begin{align}
\|\hat{V} \Lambda^{1/2}O  - V\Lambda^{1/2} \|_F^2 &= \| \hat{V} \left(\Lambda^{1/2} - \Phi \right)O  - V\left(\Lambda^{1/2} - \Phi \right) + \hat{V} \Phi O  - V \Phi \|_F^2 \\
& \leq 3 \|\Lambda^{1/2} - \Phi \|^2_F + 3 \|\Lambda^{1/2} - \Phi \|^2_F + 3 \|\hat{V} \Phi O  - V \Phi \|_F^2 \\
& = 6 \|\Lambda^{1/2} - \Phi \|^2_F + 3 \|\hat{V} \Phi O  - V \Phi \|_F^2.
\end{align}
Substituting this back into (\ref{subback}), we have
\begin{align}
 \|\hat{Z}O - Z\|^2_F & \leq  2 \|\hat{\Lambda}^{1/2} - \Lambda^{1/2}\|^2_F  \label{secsum1} \\
 & \;\;\; + 12 \|\Lambda^{1/2} - \Phi \|^2_F\label{secsum2} \\
  & \;\;\; + 6 \|\hat{V} \Phi O  - V \Phi \|_F^2. \label{secsum3}
\end{align}

We begin by establishing a bound on the first summand (\ref{secsum1}). Weyl's inequality \citep{Horn-Johnson-matrix-analysis} tells us that
\begin{align}\label{weyl}
|\hat{\lambda}_i - \lambda_i | &\leq \|\hat{V} \hat{\Lambda} \hat{V}^T - V \Lambda V^T\|_F
\end{align} 
for all $i \in [d]$.  Because centering a matrix cannot increase its Frobenius norm, this implies that
\begin{align}
\|\hat{V} \hat{\Lambda} \hat{V}^T - V \Lambda V^T\|_F &= \| \mathcal{C}_n(D^{\hat{Z}(Y^n)} - D^Z)\mathcal{C}_n\|_F\\
 &\leq \|D^{\hat{Z}(Y^n)} - D^Z \|_F.
\end{align}
Noting that 
\begin{align}
|\hat{\lambda}^{1/2}_i - \lambda_i^{1/2}|  &= \left| \frac{\hat{\lambda}_i - \lambda_i}{\lambda_i^{1/2} + \hat{\lambda}_i^{1/2} } \right|  \leq \left| \frac{\hat{\lambda}_i - \lambda_i}{\lambda_d^{1/2} } \right|,
\end{align}
it thus follows from Weyl's inequality that
\begin{align}
\|\hat{\Lambda}^{1/2} - \Lambda^{1/2}\|^2_F &= \sum_{i=1}^d (\hat{\lambda}^{1/2}_i - \lambda_i^{1/2})^2\\
                                                                       &\leq \frac{\sum_{i=1}^d (\hat{\lambda_i} - \lambda_i)^2}{\lambda_d}\\
                                                                       &\leq \frac{d}{\lambda_d}  \|D^{\hat{Z}(Y^n)} - D^Z \|^2_F. \label{thisthis}
\end{align}
The result in (\ref{thisthis}) is enough to for us to establish concentration of the first summand in (\ref{secsum1}). Next, we consider the second and third summands from (\ref{secsum2}) and (\ref{secsum3}).

It will be convenient to build a diagonal structure on the matrix $\Phi$ as follows. Recall our assumption that the LPM possesses $a(n)-b(n)$ distinctly bunched eigenvalues for functions $a(n)$ and $b(n)$. Define $k \in \mathbb{N}$ and $i_{1}, \ldots, i_{k}, i_{k+1}$ so as to satisfy the requirement of $a(n)-b(n)$ distinctly bunched eigenvalues. For notational convenience, we also define $i_0 := 0$, $\lambda_{0} := \infty$.  Given $k$ and the $i_j$'s, we define the diagonal matrix $\Phi \in \mathbb{R}^{d \times d}$ as
\begin{align}
\Phi &:= \lambda^{1/2}_{i_1} I_{i_2 - i_1} \oplus \lambda^{1/2}_{i_2} I_{i_3 - i_2} \oplus \cdots \oplus \lambda^{1/2}_{i_{k-1}} I_{i_k - i_{k-1}} \oplus \lambda^{1/2}_{i_k} I_{i_{k +1} - i_k}
\end{align}
where $I_{d^*}$ denotes the $d^*\times d^*$ identity matrix for any $d^* \in \mathbb{N}$. 

It follows that 
\begin{align}
 \|\Lambda^{1/2} - \Phi \|^2_F &= \sum_{j=1}^k \sum_{i = i_j}^{i_{j+1} - 1} \left(\lambda_{i}^{1/2} - \lambda_{i_j}^{1/2} \right)^2 \\
 &\leq \sum_{j=1}^k (i_{j+1} - i_j - 1) \left(\lambda_{i_{j+1} - 1}^{1/2} - \lambda_{i_j}^{1/2} \right)^2 \\
  &= \sum_{j=1}^k (i_{j+1} - i_j - 1) \frac{\left(\lambda_{i_{j+1} - 1} - \lambda_{i_j} \right)^2}{\left(\lambda_{i_{j+1} - 1}^{1/2} + \lambda_{i_j}^{1/2} \right)^2} \\
    &\leq \sum_{j=1}^k (i_{j+1} - i_j - 1) \frac{\left(\lambda_{i_{j+1} - 1} - \lambda_{i_j} \right)^2}{4 \lambda_{i_{j+1} - 1} } \\
    &\leq (d - k) \sup_{j \in [k]} \frac{\left(\lambda_{i_{j+1} - 1} - \lambda_{i_j} \right)^2}{4 \lambda_{i_{j+1} - 1} } \\
    &\leq \frac{(d - 1)}{4} a(n),
 \end{align}
 where the last line follows from the fact that the LPM possesses $a(n)-b(n)$ distinctly bunched eigenvalues, as well as the fact that $k \geq 1$.  This provides a bound for (\ref{secsum2}). 
 
 Finally, let's consider  (\ref{secsum3}). Let's impose some additional structure on $O$ by assuming $O = O^{\Phi}$, where $O^{\Phi}$ is a block diagonal orthogonal matrix such that
 \begin{align}
O^{\Phi} &= O^{1} \oplus O^{2} \oplus \cdots \oplus O^{k},
\end{align}
where for each $j \in [k]$, $O^{j}  \in \mathcal{O}_{i_{j+1} - i_{j}}$. Let $\mathcal{O}^{\Phi}$ denote the class of matrices possessing this block diagonal orthogonal structure, and note that $\mathcal{O}^{\Phi} \subseteq \mathcal{O}_p$. Because $\Phi$ and $O^{\Phi}$ possess the same block diagonal structure with each block of $\Phi$ being a scalar matrix, we have commutativity of the multiplication $O^{\Phi} \Phi = \Phi O^{\Phi}$. Therefore,
\begin{align}
\| \hat{V} \Phi O^{\Phi} - V \Phi\|_F^2 =&  \| \hat{V} O^{\Phi} \Phi - V \Phi\|_F^2\\
= &  \sum_{j=1}^k \lambda_{i_j}  \| \hat{V}_{i_{j}:(i_{j+1} - 1)} O^{j} - V_{i_{j}:(i_{j+1} - 1)} \|_F^2 \label{splitup}
\end{align}
by Lemma~\ref{lemma:unfurl}, where $V_{i_{j}:(i_{j+1} - 1)}$ and $\hat{V}_{i_{j}:(i_{j+1} - 1)}$ are the submatrices of $V$ and $\hat{V}$ consisting of columns $i_{j}$ through $i_{j+1} - 1$.

This format of (\ref{splitup}) is amenable to the application of the Davis-Kahan Theorem as stated in \citep{Yu-Wang-Samworth-davis-kahan} (Lemma~\ref{daviskahan}). Recall that this theorem tells us that, for any $1 \leq r \leq s \leq d$ and $q = s - r + 1$,
\begin{align}
\inf_{O \in \mathcal{O}_{q}}  \| \hat{V}_{r:s} O - V_{r:s} \|_F^2 & \leq \frac{2^{3} \|\hat{Z} \hat{Z}^T  - Z Z^T \|^2_F}{\text{min}\left(\lambda_{r-1} - \lambda_{r}, \lambda_{s} - \lambda_{s+1}\right)^2}
\end{align}

Recall that $\|\hat{Z} \hat{Z}^T  - Z Z^T \|^2_F \leq \|D^{\hat{Z}(Y^n)} - D^Z \|^2_F$. For a fixed choice of $\Phi$, we thus have
\begin{align}
\inf_{O \in \mathcal{O}_p}\|\hat{V} \Phi O  - V \Phi \|_F^2 &\leq \inf_{O \in \mathcal{O}^{\Phi}} \|\hat{V} \Phi O  - V \Phi \|_F^2\\
= & \sum_{j=1}^k \lambda_{i_j}  \inf_{O \in \mathcal{O}_{i_{j+1} - i_j}} \| \hat{V}_{i_{j}:(i_{j+1} - 1)} O - V_{i_{j}:(i_{j+1} - 1)} \|_F^2 \\
\leq & \sum_{j=1}^k \lambda_{i_j}  \frac{2^{3} \|D^{\hat{Z}(Y^n)} - D^Z \|^2_F}{\text{min}\left(\lambda_{i_j - 1} - \lambda_{i_j}, \lambda_{i_{j+1} - 1} - \lambda_{i_{j+1}}\right)^2}\\
\leq &  2^{3} k \|D^{\hat{Z}(Y^n)} - D^Z \|^2_F \sup_{j \in [k]} \frac{\lambda_{i_{j}}}{\left(\lambda_{i_{j+1} - 1} - \lambda_{i_{j+1}}\right)^2}\\
\leq &  2^{3} d \|D^{\hat{Z}(Y^n)} - D^Z \|^2_F b(n)^{-1},
\end{align}
due to the assumption that the LPM possesses $a(n)-b(n)$ distinctly bunched eigenvalues and that $k \leq d$.

Bringing together our bounds on the equations (\ref{secsum1}), (\ref{secsum2}), (\ref{secsum3}), we obtain the bound
\begin{align}
\inf_{O \in \mathcal{O}_{d}} \|\hat{Z} O - Z\|^2_F & \leq  3 (d - 1) a(n) + d \left(\frac{2}{\lambda_d} +  \frac{48}{b(n)} \right)  \|D^{\hat{Z}(Y^n)} - D^Z \|^2_F\\
& \leq  3 (d - 1) a(n) + 50 d b(n)^{-1}  \|D^{\hat{Z}(Y^n)} - D^Z \|^2_F
  \end{align}
  because $\lambda_d \geq b(n)$ by the definition $a(n)-b(n)$ distinctly bunched eigenvalues.


The final result then follows from applying Lemma~\ref{consistentdist} to bound $\|D^{\hat{Z}(Y^n)} - D^Z \|^2_F$.  This concludes the proof of Lemma~\ref{consistentpositions}. 
\begin{lemma} \label{lemma:unfurl}
Consider $d, k \in \mathbb{N}$ such that $k \leq d$ and $\ell_1, \ldots, \ell_k \in \mathbb{N}$ such that $\sum_{j=1}^k \ell_j = d$. Define 
\begin{align}
O &= O^1  \oplus O^2  \oplus \cdots \oplus O^k,\\
\Phi &= \phi_1 I_{\ell_1}  \oplus \phi_2 I_{\ell_2}  \oplus \cdots \oplus \phi_k I_{\ell_k},
\end{align}
where $O^1, \ldots, O^k$ are orthogonal matrices satisfying $O^1 \in \mathcal{O}_{\ell_1}, \ldots. O^k \in \mathcal{O}_{\ell_k}$ and $\phi_1, \ldots, \phi_k \in \mathbb{R}$. Define $i_{1} := 1$ and $i_j =: i_{j-1} + \ell_{j-1}$  for $j \in \left\{2, \ldots, k\right\}$. Then,
\begin{align}
\| U O \Phi - V \Phi\|_F^2
& = \sum_{j=1}^k \phi_{j}^2  \| U_{i_{j}:(i_{j} + \ell_j - 1)} O^{j} - V_{i_{j}:(i_{j} + \ell_j - 1)} \|_F^2
\end{align}
for any two matrices $U, V \in \mathbb{R}^{n \times d}$ satisfying $U^T U = I_d$ and $V^T V = I_d$. Here, $U_{a:b}$ and $V_{a:b}$ denote the submatrices of $U$ and $V$ consisting of columns $a$ through $b$.
\end{lemma}

\begin{proof}
Define $A^{1}, \ldots, A^{k} \in \left\{0,1\right\}^{d \times d}$ such that for each $j \in [k]$, 
\begin{align}
A^{j} &:= c^j_{1} I_{\ell_1} \oplus c^j_{2} I_{\ell_2} \oplus \cdots \oplus c^j_{k} I_{\ell_k}
\end{align}
and for $i,j \in [k]$, $c^j_i = 1$ if $i = j$, and is 0 otherwise. Note that each $A^{j}$ is symmetric, diagonal, and idempotent with
\begin{align}
\sum_{j=1}^k A^{j} = I_d.
\end{align}
Furthermore, each $A^{j}$ can be decomposed according to
\begin{align}
A^{j} = a^{j} \left.a^{j} \right.^T
\end{align}
where, for $j \in [k]$, $a^{j} \in \left\{0,1\right\}^{d \times \ell_j}$ such that
\begin{align}
a^{j}_{\ell m} &= \begin{cases}
1 & \text{ if } \ell \in \left\{i_j, i_j + 1 \ldots, i_{j} + \ell_j - 1\right\} \text{ and } m = \ell - i_j + 1\\
0 & \text{ otherwise }.
\end{cases}
\end{align}
For any matrix $M$ with $d$ columns, it follows that
\begin{align}
M_{i_{j}:(i_{j} + \ell_j - 1)} &= M a^{j}.
\end{align}
Each of $A^{1}, \ldots, A^{k}$, $\Phi$, and $O$ share the same block diagonal structure. Moreover, each block of $A^{1}, \ldots, A^{k}$, and $\Phi$ is a scalar matrix. Therefore, we have that $A^{j} \Phi = \Phi A^{j}$, $A^{j} O =O A^{j}$, and $\Phi O = O \Phi$ for any $j \in [k]$.

Applying these properties, as well as the trace definition of the Frobenius norm, we have that
\begin{align*}
 \|U \Phi O - V \Phi \|^2_F &= \text{Tr}\left(O^T \Phi^T U^T U \Phi O + \Phi^T V^T V \Phi - 2 \Phi^T V^T U \Phi O\right) \\
 &= \text{Tr}\left(\Phi^2\right) + \text{Tr}\left(\Phi^2\right) - \text{Tr}\left(2 \Phi^T V^T U \Phi O \right)\\
& = 2\text{Tr}\left( \Phi^2 \right) - 2\text{Tr}\left(\Phi^2   V^T U O \right).
\end{align*}

Focusing on the $\text{Tr}\left(\Phi^2   V^T U O \right)$ term, we have that
\begin{align}
\text{Tr}\left(\Phi^2   V^T U O \right) &= \text{Tr}\left(I_d \Phi^2   V^T U O \right) \\
&= \text{Tr}\left(\sum_{j=1}^k A^{j}  \Phi^2  V^T U O \right) \\
&= \sum_{j=1}^k \text{Tr}\left(A^{j} \Phi^2   V^TU O \right) \\
&= \sum_{j=1}^k \text{Tr}\left(A^{j} A^{j} A^{j} \Phi^2   V^T U O \right) \\
&= \sum_{j=1}^k  \text{Tr}\left(A^{j}  \Phi^2 A^{j}    V^T U A^{j} O \right).
\end{align}

Next, we apply the decomposition $A^{j} = a^{j} \left.a^{j} \right.^T$ described above.

\begin{align}
\sum_{j=1}^k  \text{Tr}\left(A^{j}  \Phi^2 A^{j}    V^T U A^{j} O \right) &= \sum_{j=1}^k  \text{Tr}\left(a^{j} \left.a^{j}\right.^T  \Phi^2 a^{j} \left.a^{j}\right.^T     V^T U a^{j} \left.a^{j}\right.^T   O \right) \\
&= \sum_{j=1}^k \text{Tr}\left( \left.a^{j}\right.^T  \Phi^2 a^{j} \left.a^{j}\right.^T     V^TU a^{j} \left.a^{j}\right.^T  O a^{j} \right) \\
&= \sum_{j=1}^k   \text{Tr}\left(  \Phi^2_{i_{j}:(i_{j} + \ell_j - 1)}   \left(V_{i_{j}:(i_{j} + \ell_j - 1)}\right)^T U_{i_{j}:(i_{j} + \ell_j - 1)} O^{j} \right) \\
&= \sum_{j=1}^k   \phi_j^2 \text{Tr}\left( \left(V_{i_{j}:(i_{j} + \ell_j - 1)}\right)^T U_{i_{j}:(i_{j} + \ell_j - 1)} O^{j} \right) 
\end{align}

Recognizing that $\text{Tr}\left( \Phi^2 \right) = \sum_{j=1}^k \phi_j^2 \text{Tr}\left(I_{\ell_j} \right)$, we thus have:

\begin{align}
 \|U \Phi O - V \Phi \|^2_F &=  \sum_{j=1}^k \phi_j^2 2\text{Tr}\left(I_{\ell_j} - \left(V_{i_{j}:(i_{j} + \ell_j - 1)}\right)^T U_{i_{j}:(i_{j} + \ell_j - 1)} O^{j} \right)\\
 &= \sum_{j=1}^k \phi_{j}^2  \| U_{i_{j}:(i_{j} + \ell_j - 1)} O^{j} - V_{i_{j}:(i_{j} + \ell_j - 1)} \|_F^2
\end{align}
by the definition of the Frobenius norm.

\end{proof}

\subsubsection{Proof of Theorem~\ref{consistencies}}\label{consistenciesproof}
\begin{proof}
We first consider the learnability of the link probabilities. Let $\delta_n = \alpha^K_n G(n) n^{1.5} e(n)^{-1}$ and suppose that $\delta_n \rightarrow 0$. Then,
\begin{align}
\mathbb{P}\left(\frac{\|P^{\hat{z}(Y^n)} - P^{z}\|^2_F}{e(n)} > \delta_n\right) &\leq \mathbb{P}\left(\frac{\|P^{\hat{z}(Y^n)} - P^{z}\|^2_F}{e(n)} > \delta_n e(n) \mid  \sup_{1 \leq i \leq n} \|z_i\| \leq G(n)\right) \\
&+ \mathbb{P}(\sup_{1 \leq i \leq n} \|z_i\| \geq G(n))\\
&\leq \frac{C}{n^2} + \mathbb{P}\left(\sup_{1 \leq i \leq n} \|z_i\| \geq G(n)\right)
\end{align}
by Lemma~\ref{consistentprob}. By the third regularity assumption, this expression converges to 0 in probability. Thus,
\begin{align}
\frac{\|P^{\hat{z}(Y^n)} - P^{z}\|^2_F}{e(n)}  \convp 0
\end{align}
because $\delta_n = o(1)$. The proof follows the same reasoning for squared distances and latent positions, simply swapping out Lemma~\ref{consistentprob} for Lemma~\ref{consistentdist} and Lemma~\ref{consistentpositions}, respectively.
\end{proof}

\subsubsection{Proof of Corollary~\ref{rectangularlearnability} }\label{rectangularlearnabilityproof}

\begin{proof}
The results follow from Theorem~\ref{consistencies}, and the following observations.
Because $g(n) = n^{p/d}$ and $K$ is integrable by Lemma~\ref{distanceswitch}, Theorem~\ref{sparsity} imples that $e(n) = n^{2-p}$. Corollary~\ref{rectangular_eigenstable} implies that the LPM almost surely possesses $a(n)-b(n)$ distinctly bunched eigenvalues with $a(n) = O(n^{2\frac{p}{d} - 1 + \epsilon})$ for any $\epsilon > 0$ (i.e. $a(n) = o(n)$ for $p < d$) and $b(n)^{-1} = O(n^{-1 - 2\frac{p}{d}})$. Table~\ref{differentlinks} shows that for $K(x) = (c + x^2)^{-q}$, $\alpha_n^K \sim \Theta(1)$ and $\beta_n^K \sim \Theta(G(n)^{2q + 4})$.  Recall that $G(n) = \Theta(n^{p/d})$ by Lemma~\ref{rectangularareregular}. Thus,  $\beta^K_n = \Theta( n^{\frac{p}{d}(2q+4)})$. Inserting these values into Theorem~\ref{consistencies} provides the results in points (1) through (3).

Letting $d$ grow large while simultaneously setting $q$ to be the smallest integer larger than $d/2$ allows for learnability of all three targets for values of $p$ that are arbitrarily close to 0.5. Thus, we can have learnability arbitrarily close to $e(n) = n^{1.5}$. 
\end{proof}

The proof above relied on Corollary~\ref{rectangular_eigenstable}. Before presenting this result, we will first establish lemma~\ref{eigenlemma} to control the behavior of the eigenvalues $\lambda_1, \ldots, \lambda_d$. This lemma and its proof were inspired by the work of \citet[Proposition 4.3]{Sussman-Tang-Priebe-consistent}.

\begin{lemma}\label{eigenlemma} 
Let $\lambda_i$ denote the $i$th largest eigenvalue of $\mathcal{C}_nZZ^T\mathcal{C}_n$, where $\mathcal{C}_n Z \in \mathbb{R}^{n \times d}$ is the centered matrix of $n$ latent positions associated with a rectangular LPM generated with $g(n) = n^{p/d}$. Then, if $p < d$ and $i \leq d$,
\begin{align}
\lambda_1 - \lambda_d &= O\left(g(n)^{2}n^{\epsilon} \right) \text{ almost surely} \label{thetanot2}
\end{align}
for any $\epsilon > 0$ and
\begin{align}
\frac{\lambda_i}{ng(n)^2}  \convas \frac{1}{12} \label{thetanot}.
\end{align}
\end{lemma}
\begin{proof}
Recall that both $Z^T\mathcal{C}_n\mathcal{C}_n Z$ and $\mathcal{C}_nZZ^T\mathcal{C}_n$ have the same $d$ non-zero eigenvalues $\lambda_1, \ldots, \lambda_d$.
Furthermore,
\begin{align}
\|Z^T\mathcal{C}_n \mathcal{C}_n Z - \mathbb{E}(Z^T\mathcal{C}_n\mathcal{C}_nZ)\|_F \leq  \|Z^TZ - \mathbb{E}(Z^TZ)\|_F + \frac{\|Z^T1_{n}Z - \mathbb{E}(Z^T 1_{n}Z)\|_F}{n}  \label{aaa}
\end{align}
where $1_{n } \in \mathbb{R}^{n \times n}$ denotes a matrix filled with ones. Suppose $t_{n+1}$ is known, and $\pi$ is a random permutation on $1, \ldots, n$. Then, each $Z^{\pi(i)}$ is uniformly distributed on $[-g(t_{n+1}), g(t_{n+1})]^d$ by Lemma~\ref{prop:indunif}. Thus, after randomly permuting the row indices in $Z$, they can be treated as independent samples from this distribution. Going forward, in a slight abuse of notation, we assume that the rows of $Z$ have been randomly permuted, meaning each row can be treated as an iid uniform sample on $[-g(t_{n+1}), g(t_{n+1})]^d$.

Therefore, $(Z^TZ)_{ij} = \sum_{k=1}^n Z_{ki} Z_{kj}$ is the sum of $n$ iid random variables, and each summand's absolute value is upper bounded by $ng(t_{n+1})^2$. The Hoeffding bound \citep{Hoeffding-on-Hoeffding} provides us with the following concentration result.
\begin{align}
\mathbb{P}(|(Z^TZ)_{ij} - \mathbb{E}((Z^TZ)_{ij})| > \delta) \leq 2 \exp{\left(  \frac{- 2\delta^2}{ g(t_{n+1})^4} \right)}.
\end{align}
Combining this with a union bound provides
\begin{align}
\mathbb{P}(\|(Z^TZ) - \mathbb{E}((Z^TZ))\|_F > d^2\delta) \leq 2d^2 \exp{\left(  \frac{- 2\delta^2}{ g(t_{n+1})^4} \right)}. \label{bbb}
\end{align}

Furthermore, $(Z^T1_n Z)_{ij} = (\sum_{k=1}^n Z_{ki}) (\sum_{k=1}^n Z_{kj})$ with both factors in this product being identically distributed. Moreover, the entries in either summand are bounded in absolute value according to  $|Z_{kj}| \leq g(t_{n+1})$. Therefore, applying the Hoeffding bound to each entry in $Z^TZ$ yields
\begin{align}
\mathbb{P}(|(Z^T1_n Z)_{ij} - \mathbb{E}((Z^T1_n Z)_{ij})| > \delta) &\leq 2\mathbb{P}(|(\sum_{k=1}^n Z_{ki}) - \mathbb{E}(\sum_{k=1}^n Z_{ki})) | > \sqrt{\delta})\\
                    &\leq 4 \exp{\left(  \frac{- 2 \delta}{ g(t_{n+1})} \right)},
\end{align}
achieved through a union bound on the two summands differing from their means by $\sqrt{\delta}$. Another union bound over all matrix entries results in \begin{align}
\mathbb{P}\left(\|Z^T1_{n}Z - \mathbb{E}(Z^T 1_{n}Z)\|_F > d^2 \delta \right) \leq 4d^2 \exp{\left(  \frac{- 2 \delta}{ g(t_{n+1})} \right)}. \label{ccc}
\end{align}

Combining Equations (\ref{aaa}), (\ref{bbb}), and (\ref{ccc}) yields
\begin{align}
\mathbb{P}\left(\|Z^T\mathcal{C}_n \mathcal{C}_n Z - \mathbb{E}(Z^T\mathcal{C}_n\mathcal{C}_nZ)\|_F >  2d^2 \delta \right) \leq 2d^2 \exp{\left(  \frac{- 2\delta^2}{ g(t_{n+1})^4} \right)} + 4d^2 \exp{\left(  \frac{- 2 \delta}{ g(t_{n+1})} \right)}.
\end{align}
To translate this into a result for the eigenvalues, we can apply Weyl's inequality \citep{Horn-Johnson-matrix-analysis} to obtain
\begin{align}
\mathbb{P}\left(|\lambda_i(Z^T\mathcal{C}_n \mathcal{C}_n Z) - \lambda_i(\mathbb{E}(Z^T\mathcal{C}_n \mathcal{C}_n Z))| > 2d^2 \delta\right) \leq 2d^2 \exp{\left(  \frac{- 2\delta^2}{ g(t_{n+1})^4} \right)} + 4d^2 \exp{\left(  \frac{- 2 \delta}{ g(t_{n+1})} \right)} \label{choosedelta}
\end{align}
for $1\leq i \leq d$. 
We can analytically determine the values of $\lambda_i$ by noting that
\begin{align}
\mathbb{E}(Z^T\mathcal{C}_n \mathcal{C}_n Z)_{ii} &= \frac{(n-1)g(t_{n+1})^2}{12}\\
\mathbb{E}(Z^T\mathcal{C}_n \mathcal{C}_n Z)_{i \neq j} &= 0,
\end{align}
which indicates that 
\begin{align}
\lambda_i(\mathbb{E}(Z^T\mathcal{C}_n \mathcal{C}_n Z)) &= \frac{(n-1)g(t_{n+1})^2}{12}
\end{align}
for $i \leq d$, 0 otherwise. Substituting these values into (\ref{choosedelta}) we obtain
\begin{align}
\mathbb{P}\left(\left|\lambda_i(Z^T\mathcal{C}_n \mathcal{C}_n Z) - \frac{(n-1)g(t_{n+1})^2}{12}\right| > 2d^2 \delta\right) \leq 2d^2 \exp{\left(  \frac{- 2\delta^2}{ g(t_{n+1})^4} \right)} + 4d^2 \exp{\left(  \frac{- 2 \delta}{ g(t_{n+1})} \right)} \label{choosedelta}
\end{align}
for $1\leq i \leq d$. Choosing $\delta = \epsilon (n-1)g(t_{n + 1})^{2}$ for $\epsilon > 0$  yields:
\begin{align}
\mathbb{P}\left(\left|\frac{\lambda_i(Z^T\mathcal{C}_n \mathcal{C}_n Z)}{(n-1)g(t_{n+1})^2} - \frac{1}{12}\right| > 2d^2 \epsilon \right) \leq 2d^2 \exp{\left(- 2\epsilon^2 (n-1)^2 \right)} + 4d^2 \exp{\left(  - 2(n-1) \epsilon g(t_{n+1})\right)}. \label{subhere}
\end{align}
Due to its reliance of the random quantity $g(t_{n+1})$, this bound is difficult to directly use. Instead, let us upper bound it with a deterministic quantity. Lemma~\ref{waitingtimes} established that $t_{n+1}$ follows Gamma($n+1$, 1). Therefore,
\begin{align*}
\sum_{n=1}^{\infty} \mathbb{P}\left(t_{n} \leq 1 \right) &= \sum_{n=1}^{\infty} \int_{0}^1 \frac{1}{\Gamma(n)} t^{n-1} e^{-t} \text{dt}\\
&\leq  \sum_{n=1}^{\infty} \frac{1}{\Gamma(n)} \\
&=e.
\end{align*}
The Borel-Cantelli lemma thus implies that $\mathbb{P}\left(t_{n \in \mathbb{N}} \leq 1 \text{ infinitely often } \right) = 0$. Therefore, there almost-surely exists an $n^* < \infty$ such that $g(t_{n}) \geq 1$ for all $n \geq n^*$. Moreover, note that $g(t_{n}) \in (0, 1]$ for $n < n^*$. In turn, we can incorporate the upper bound (\ref{subhere}) to determine
\begin{align}
&\sum_{n=1}^{\infty} \mathbb{P}\left(\left|\frac{\lambda_i(Z^T\mathcal{C}_n \mathcal{C}_n Z)}{(n-1)g(t_{n+1})^2} - \frac{1}{12}\right| > 2d^2 \epsilon \right)\\ 
\leq & \sum_{n=1}^{\infty}  \left(2d^2 \exp{\left(- 2\epsilon^2 (n-1)^2 \right)} + 4d^2 \exp{\left(  - 2(n-1) \epsilon g(t_{n+1})\right)}\right)\\
\leq & \sum_{n=1}^{\infty}  2d^2 \exp{\left(- 2\epsilon^2 (n-1)^2 \right)} + 4d^2  \left(\sum_{n=1}^{n^{*}} 1 + \sum_{n=n^{*}}^{\infty} \exp{\left(  - 2(n-1) \epsilon g(t_{n+1})\right)}\right) \\
\leq & \sum_{n=1}^{\infty} 2d^2 \exp{\left(- 2(n-1)^2 \epsilon^2 \right)} + 4d^2 \left(n^* + \sum_{n=n^*}^{\infty} \exp{\left(  - 2(n-1) \epsilon \right)}\right)\\
< & \infty \text{ almost surely}.
\end{align}

We can thus apply the Borel-Cantelli Lemma to establish that 
\begin{align}
\frac{\lambda_i(Z^T\mathcal{C}_n \mathcal{C}_n Z)}{(n-1)g(t_{n+1})^2} & \convas \frac{1}{12}.
\end{align}
This result implies (\ref{thetanot}) by the following reasoning. 

Lemma~\ref{waitingtimes} established that $t_{n+1} (n+1)^{-1} \convas 1$. Therefore,
\begin{align}
g\left(\frac{t_{n+1}}{n+1} \right) = \frac{t_{n+1}^{p/d}}{g(n)} \convas g(1) = 1
\end{align}
by the continuous mapping theorem. Recognizing that $n/(n-1) \rightarrow 1$ and $g(n)/g(n-1) \rightarrow 1$ as $n$ goes to infinity, we thus have that
\begin{align}
\frac{\lambda_i(Z^T\mathcal{C}_n \mathcal{C}_n Z)}{ng(n)^2} & \convas \frac{1}{12},
\end{align}
giving us the result in (\ref{thetanot}).

Next, we will provide the result in (\ref{thetanot2}). We recognize that 
\begin{align}
&\mathbb{P}\left(|\lambda_1 - \lambda_d| > 4d^2 \delta \right)\\ & \leq  \mathbb{P}\left(\left|\lambda_1 - \frac{(n-1)g(t_{n+1})^2}{12}\right| + \left|\lambda_ d - \frac{(n-1)g(t_{n+1})^2}{12}\right|  > 4d^2 \delta \right) \\
                                                                                                & \leq  \mathbb{P}\left(\left|\lambda_1 - \frac{(n-1)g(t_{n+1})^2}{12}\right|  > 2d^2 \delta \right) + \mathbb{P}\left(\left|\lambda_d - \frac{(n-1)g(t_{n+1})^2}{12}\right|  > 2d^2 \delta \right)\\
                                                                                                &\leq 4d^2 \exp{\left(  \frac{- 2\delta^2}{ g(t_{n+1})^4} \right)} + 8d^2 \exp{\left(  \frac{- 2 \delta}{ g(t_{n+1})} \right)}                                                                                     
\end{align}
by (\ref{choosedelta}). Choosing $\delta = g(t_{n+1})^{2}n^{\epsilon}$ for $\epsilon > 0$ implies that
\begin{align}
\mathbb{P}\left(\frac{|\lambda_1 - \lambda_d|}{g(t_{n+1})^{2 + \epsilon}} > 4d^2 \right) &\leq 4d^2 \exp{\left(- 2n^{2 \epsilon} \right)} + 8d^2 \exp{\left(  - 2g(t_{n+1}) n^{\epsilon} \right)}  
\end{align}
Recall that there almost-surely exists an $n^* < \infty$ such that $g(t_{n}) \geq 1$ for all $n \geq n_*$. In turn, we can once again apply Borell-Cantelli in the same manner
\begin{align}
& \sum_{n=1}^{\infty} \mathbb{P}\left(\frac{|\lambda_1 - \lambda_d|}{g(t_{n+1})^{2} n^{ \epsilon}} > 4d^2 \right) \\
\leq & 4d^2 \sum_{n=1}^{\infty}  \left(\exp{\left(- 2n^{2 \epsilon} \right)} + 2 \exp{\left(  - 2g(t_{n+1}) n^{\epsilon} \right)}   \right)\\
\leq  & 4d^2 \left(\sum_{n=1}^{\infty} \exp{\left(- 2n^{2 \epsilon} \right)} + 2n^* + \sum_{n=n^*}^{\infty} \exp{\left(  - 2n^{\epsilon} \right)}   \right)\\
< & \infty \text{ almost surely}.
\end{align}
The Borell-Cantelli lemma thus implies the result in (\ref{thetanot2}).
\end{proof}

We can now use these results to determine that rectangular latent position network models possess $a(n)-b(n)$ distinctly bunched eigenvalues. We state this result as Corollary~\ref{rectangular_eigenstable}.

\begin{cor}\label{rectangular_eigenstable}
A $d$-dimensional rectangular latent position network model generated with $g(n) = n^{p/d}$ almost surely possesses $a(n)-b(n)$ distinctly bunched eigenvalues for $a(n) = O\left(n^{2\frac{p}{d} - 1 + \epsilon}\right)$ and $b(n)^{-1} = O\left(n^{-1 - 2\frac{p}{d}}\right)$ for any $\epsilon > 0$.
\end{cor}
\begin{proof}
Lemma~\ref{eigenlemma} established that
\begin{align}
\frac{\lambda_i}{n^{1 + 2 \frac{p}{d}}} \convas \frac{1}{12}
\end{align}
 for all $i = 1, \ldots, d$. The continuous mapping theorem thus implies that
 \begin{align}
\frac{1}{\lambda_d^2} &= O\left(n^{-2 - 4 \frac{p}{d}} \right),\\
\frac{\lambda_1}{\lambda_d^2} &= O(n^{-1 - 2\frac{p}{d}})
\end{align}
almost surely.
 Moreover, (\ref{thetanot2}) from Lemma~\ref{eigenlemma} implies that, almost surely,
\begin{align}
(\lambda_1 - \lambda_d)^2 & = O( n^{4\frac{p}{d} + \epsilon}),
\end{align}
for any $\epsilon > 0$. Therefore,
\begin{align}
\frac{(\lambda_1 - \lambda_d)^2}{\lambda_d} &= O\left(n^{2\frac{p}{d} - 1 + \epsilon} \right).
\end{align}
The LPM is thus almost surely meets the criteria for possessing $a(n)-b(n)$ distinctly bunched eigenvalues with $k=1$, $i_1 = 1$, $a(n) = O(n^{2\frac{p}{d} - 1 + \epsilon})$ and $b(n)^{-1} = O(n^{-1 - 2\frac{p}{d}})$ for $\epsilon > 0$.
\end{proof}

\subsubsection{Proof of Corollary~\ref{hofflearnability}}\label{hofflearnabilityproof}
\begin{proof}
Note that this LPM is regular with $G(n) = \sqrt{2 \sigma_1^2 (1+c) \log(n)}$ for any $c > 0$ by Lemma~\ref{exchangeableareregular}, and $e(n) = n^2$. Furthermore, Corollary \ref{hoffeigenstable} establishes that this LPM is almost surely possesses $a(n)-b(n)$ distinctly bunched eigenvalues for $a(n) = O(1)$ and $b(n)^{-1} = O(n^{-1})$. Consulting Table~\ref{differentlinks}, we see that for both link functions $\alpha^K_n = \Theta(1)$ and $\beta^K_n = \Theta(e^{G(n)^2}) = \Theta(n^{2\sigma^2 (1+c)})$. Thus, applying Theorem~\ref{consistencies} indicates that we have learnable latent positions and distances provided that $2\sigma_1^2(1+c) < 1/2$. 
\end{proof}

The above proof relied Corollary~\ref{hoffeigenstable} to establish that the LPM with Gaussian latent positions possesses $a(n)-b(n)$ distinctly bunched eigenvalues. Before presenting this result, we first present Lemma~\ref{gaussianeigen}, as it facilitates its proof.

\begin{lemma}\label{gaussianeigen}
Let the rows of $Z \in \mathbb{R}^{n \times d}$ independently follow a multivariate Gaussian distribution with mean zero and diagonal variance matrix $\Sigma$. Let $\sigma^2_1, \ldots, \sigma^2_d$ denote the entries along the diagonal of $\Sigma$, with $\sigma_1 \geq \sigma_2 \geq \cdots \geq \sigma_d > 0$. Let $\lambda_i$ denote the $i$th largest eigenvalue of $\mathcal{C}_nZZ^T\mathcal{C}_n$, where $\mathcal{C}_n$ is defined as in (\ref{doublecentering}). Then,
\begin{align}
\frac{\lambda_i}{n}  \convas \sigma_i^2. 
\end{align}
Moreover, $|\lambda_i - \lambda_j| = O(1)$ almost surely whenever $\sigma_i = \sigma_j$.
\end{lemma}
\begin{proof}
The proof proceeds very similarly as for Lemma~\ref{eigenlemma}. Recall that both $Z^T\mathcal{C}_n\mathcal{C}_n Z$ and $\mathcal{C}_nZZ^T\mathcal{C}_n$ have the same $d$ non-zero eigenvalues $\lambda_1, \ldots, \lambda_d$. Furthermore,
\begin{align}
\|Z^T\mathcal{C}_n \mathcal{C}_n Z - \mathbb{E}(Z^T\mathcal{C}_n\mathcal{C}_nZ)\|_F \leq  \|Z^TZ - \mathbb{E}(Z^TZ)\|_F + \frac{\|Z^T1_{n}Z - \mathbb{E}(Z^T 1_{n}Z)\|_F}{n}  \label{aaanew}
\end{align}
where $1_{n } \in \mathbb{R}^{n \times n}$ denotes a matrix filled with ones. Note that
\begin{align}
\|Z^TZ - \mathbb{E}(Z^TZ)\|_F \leq \sum_{i=1}^d \sum_{j=1}^d |Z^TZ_{ij} - \mathbb{E}(Z^TZ)_{ij}|
\end{align}
Applying \citet[Lemma 1]{Ravikumar-et-al-covariance-estimation} and the union bound yields
\begin{align}
\mathbb{P}\left(\|Z^TZ - \mathbb{E}(Z^TZ)\|_F > d^2\delta \right) \leq 4d^2 \exp{\left(\frac{-n \delta^2}{3200 \sigma_1^2}  \right)}. \label{bbbnew}
\end{align}

Furthermore, note that $(Z^T1_n Z)_{ij}/n = (\sum_{k=1}^n Z_{ki}/\sqrt{n}) (\sum_{k=1}^n Z_{kj}/\sqrt{n})$ can be viewed as the product of two Gaussian distributed random variables with variances $\sigma_i^2$ and $\sigma_j^2$ respectively. This means that $(Z^T1_n Z)/n$ can be viewed as $u u^T$ where $u$ is a $d$-dimensional Gaussian vector with mean 0 and variance matrix $\Sigma$. Again applying \citet[Lemma 1]{Ravikumar-et-al-covariance-estimation} and the union bound yields
\begin{align}
\mathbb{P}(\|(Z^T1_n Z)_{ij} - \mathbb{E}((Z^T1_n Z)_{ij})\|_F > d^2\delta) &\leq 4d^2 \exp{\left(  \frac{- n\delta^2}{3200 \sigma_1^2} \right)}. \label{cccnew}
\end{align}

Combing Equations (\ref{aaanew}), (\ref{bbbnew}), and (\ref{cccnew}) yields
\begin{align}
\mathbb{P}\left(\|Z^T\mathcal{C}_n \mathcal{C}_n Z - \mathbb{E}(Z^T\mathcal{C}_n\mathcal{C}_nZ)\|_F >  2d^2 \delta \right) \leq 8d^2 \exp{\left(\frac{-n \delta^2}{3200 \sigma_1^2}  \right)}.\end{align}
To translate this into a result for the eigenvalues, we can apply Weyl's inequality \citep{Horn-Johnson-matrix-analysis}. This results in
\begin{align}
\mathbb{P}(|\lambda_i(Z^T\mathcal{C}_n \mathcal{C}_n Z) - \lambda_i(\mathbb{E}(Z^T\mathcal{C}_n \mathcal{C}_n Z))| > 2d^2 \delta) \leq 8d^2 \exp{\left(\frac{-n \delta^2}{3200 \sigma_1^2}  \right)} \label{choosedeltanew}
\end{align}
for $1\leq i \leq d$. 
We can analytically determine the values of $\lambda_i$ by noting that 
\begin{align}
\mathbb{E}(Z^T\mathcal{C}_n \mathcal{C}_n Z)_{ii} &= (n-1)\sigma_i^2\\
\mathbb{E}(Z^T\mathcal{C}_n \mathcal{C}_n Z)_{i \neq j} &= 0,
\end{align}
which indicates that 
\begin{align}
\lambda_i(\mathbb{E}(Z^T\mathcal{C}_n \mathcal{C}_n Z) &= (n-1)\sigma^2_i
\end{align}
for $i \leq d$, 0 otherwise. Substituting this into (\ref{choosedeltanew}) and setting $\delta = (2d^2)^{-1} (n-1) \epsilon$ for $\epsilon > 0$, we see that
\begin{align}
\mathbb{P}\left(\left|\frac{\lambda_i(Z^T\mathcal{C}_n \mathcal{C}_n Z)}{n-1} - \sigma_i^2 \right| > \epsilon\right) \leq 8d^2 \exp{\left(\frac{-n (n-1)^2 \epsilon^2}{12800 d^4 \sigma_1^2}  \right)}
\end{align}
It follows that
\begin{align}
\sum_{n=1}^{\infty} \mathbb{P}\left(\left|\frac{\lambda_i(Z^T\mathcal{C}_n \mathcal{C}_n Z)}{n-1} - \sigma_i^2 \right| > \epsilon\right) &\leq \sum_{n=1}^{\infty} 8d^2 \exp{\left(\frac{-n (n-1)^2 \epsilon^2}{12800 d^4 \sigma_1^2}  \right)}\\
& < \infty.
\end{align}
The Borel Cantelli lemma thus implies that
\begin{align}
\frac{\lambda_i(Z^T\mathcal{C}_n \mathcal{C}_n Z)}{n}  &\convas \sigma_i^2
\end{align}
because $n^{-1}(n-1) \rightarrow 1$ as $n$ goes to infinity.

Moreover, suppose that $\sigma_{i} = \sigma_{j}$ for some pair $i,j \in \left\{1, \ldots, d\right\}$. Then, using the triangle inequality, a union bound, and (\ref{choosedeltanew}), we see that
\begin{align}
& \mathbb{P}(|\lambda_i(Z^T\mathcal{C}_n \mathcal{C}_n Z) - \lambda_j(Z^T\mathcal{C}_n \mathcal{C}_n Z) | > 2d^2 \delta)\\
& \leq \mathbb{P}(|\lambda_i(Z^T\mathcal{C}_n \mathcal{C}_n Z) - (n-1)\sigma_i^2 | > 2d^2 \delta) + \mathbb{P}(|\lambda_j(Z^T\mathcal{C}_n \mathcal{C}_n Z) - (n-1)\sigma_i^2 | > 2d^2 \delta)\\
&\leq 16d^2 \exp{\left(\frac{-n \delta^2}{3200 \sigma_1^2}  \right)}. \label{aqaq}
\end{align}
Setting $\delta = (2d^2)^{-1} \epsilon$ for $\epsilon > 0$, we have that
\begin{align}
\sum_{n=1}^{\infty} \mathbb{P}(|\lambda_i(Z^T\mathcal{C}_n \mathcal{C}_n Z) - \lambda_j(Z^T\mathcal{C}_n \mathcal{C}_n Z) | > \epsilon \delta) & \leq \sum_{n=1}^{\infty} 16d^2 \exp\left(\frac{-n \epsilon^2}{12800d^4 \sigma_1^2 } \right)\\
& < \infty.
\end{align}
Again applying the Borel Cantelli lemma, we thus have that 
\begin{align}
|\lambda_i(Z^T\mathcal{C}_n \mathcal{C}_n Z) - \lambda_j(Z^T\mathcal{C}_n \mathcal{C}_n Z) |  &= O(1)
\end{align}
almost surely.
\end{proof}

We can now proceed with the proof of Corollary~\ref{hoffeigenstable}.

\begin{cor}\label{hoffeigenstable}
Consider a LPM on $S = \mathbb{R}^d$ with each latent position independently and identically distributed according to the multivariate Gaussian distribution with mean zero and diagonal variance matrix $\Sigma$. Let $\sigma^2_1, \ldots, \sigma^2_d$ denote the entries along the diagonal of $\Sigma$, with $\sigma_1 \geq \sigma_2 \geq \cdots \geq \sigma_d > 0$.This LPM almost surely possesses $a(n)-b(n)$ distinctly bunched eigenvalues for $a(n) = O\left(1\right)$ and $b(n)^{-1} = O(n^{-1})$.
\end{cor}

\begin{proof}
Lemma~\ref{gaussianeigen} established that
\begin{align}
\frac{\lambda_i}{n} \convas \sigma_i^2
\end{align}
 for all $i = 1, \ldots, d$. The continuous mapping theorem thus implies that
 \begin{align}
\frac{n \lambda_i}{\left(\lambda_{i} - \lambda_j\right)^2} & \convas \frac{\sigma_i^2}{(\sigma_i^2 - \sigma_j^2)}
\end{align}
for any $i, j \in \left\{1, \ldots, d\right\}$ for which $\sigma_i \neq \sigma_j$. Moreover, Lemma~\ref{gaussianeigen} also implies that
\begin{align}
(\lambda_i - \lambda_j)^2 & = O(1),
\end{align}
almost surely whenever $\sigma_i = \sigma_j$. Therefore,
\begin{align}
\frac{(\lambda_i - \lambda_j)^2}{\lambda_j} &= O\left(\frac{1}{n}\right)
\end{align}
almost surely when $\sigma_i = \sigma_j$. Let $k \leq d$ be given by number of distinct values in that the sequence $\sigma_1, \ldots, \sigma_d$. Let $\sigma_{(1)} > \sigma_{(2)} > \cdots > \sigma_{(k)}$ denote these distinct values. For $j = 1, \ldots, k$ define $i_j$ to be the smallest element of $\left\{1, \ldots, d\right\}$ such that $\sigma_{i_j} = \sigma_{(j)}$. Using these values of $k$ and $i_{1}, \ldots, i_{k}$, the LPM is thus almost surely meets the criteria for possessing $a(n)-b(n)$ distinctly bunched eigenvalues with $a(n) = O(1)$ and $b(n)^{-1} = O(n^{-1})$.
\end{proof}

\subsubsection{Proof of Corollary~\ref{learnablesparsegraphon}}\label{learnablesparsegraphonproof}

\begin{proof}
The proof set-up is almost identical to that of Corollary~\ref{hofflearnabilityproof}, with the sole departure being that $\beta^K_n$ is also scaled by the inverse of sparsity term $s(n)^{-1} = n^p$ resulting in $\beta^K_n = \Theta(n^{2\sigma^2(1+c) + p})$ requiring that $2p < 1- 4\sigma^2(1+c)$. By choosing a small value of $\sigma^2$, we can have $p$ get arbitrarily close to $1/2$.
\end{proof}

\subsection{More Details regarding Sparse graphon-based LPMs}

\subsubsection{Proof of Theorem~\ref{triangles}}\label{prooftriangles}

\begin{proof}
For any $\delta > 0$, the definition of $K_n$ implies that
\begin{align}
s_n s_1^{-1} K_1(\delta) \leq K_n(\delta) \leq s_n K(\delta).
\end{align}
Therefore, 
\begin{align}
\PP(Y^n_{ij} =  1, Y^n_{ik}=1, Y^n_{jk} = 1) &= \mathbb{E}\left(K_n(\delta(Z_i, Z_j)) K_n(\delta(Z_i, Z_k)) K_n(\delta(Z_j, Z_k))\right) \\
                                                                                &\leq s_n^3 \mathbb{E}\left(K(\delta(Z_i, Z_j)) K(\delta(Z_i, Z_k)) K(\delta(Z_j, Z_k))\right) \\
                                                                                &= C_1 s_n^3
\end{align}
for some $C_1 > 0$ that is independent of $n$ where $Z_i, Z_j, Z_k \sim f$ independently. Moreover,
\begin{align}
\PP(Y^n_{ij} =  1, Y^n_{ik}=1) &= \mathbb{E}\left(K_n(\delta(Z_i, Z_j)) K_n(\delta(Z_i, Z_k)))\right) \\
                                                                                &\geq s_n^2 s_1^{-1} \mathbb{E}\left(K_1(\delta(Z_i, Z_j)) K_1(\delta(Z_i, Z_k)))\right) \\
                                                                                &= C_2 s_n^2
\end{align}
for some $C_2 > 0$ independent of $n$ because $K_1(\delta) \leq 1$ and $s_1 \in \mathbb{R}+$.  Turning to the probability of triadic closure, we thus have that
\begin{align}
\PP(Y^n_{ij} =  1, Y^n_{ik}=1, Y^n_{jk} = 1)  &= \frac{\PP(Y^n_{ij} =  1,  Y^n_{ik}=1, Y^n_{jk} = 1)}{\PP(Y^n_{ij} =  1, Y^n_{ik}=1)}\\
&\leq \frac{C_1 s_n^3}{C_2 s_n^2}\\
&\propto s_n \rightarrow 0 \text{ \;\;\; as $n \rightarrow \infty$.}
\end{align}

\end{proof}

\subsubsection{Proof of Theorem~\ref{triangles2}}\label{prooftriangles2}

\begin{proof}
Note that because we have assumed the link function is regular, it must be continuous, non-negative, non-increasing, and not identically zero. Therefore, there must exist a $c \in (0, 1]$ and $\epsilon > 0$ such that $r \leq 2 \epsilon$ implies $K(r)  \geq c$. Let $\Omega(n) = [-g(t_{n+1}), g(t_{n+1})]^d$ and $\Omega_{+}(n) = [0, g(t_{n+1})]^d$. Following Lemma~\ref{uniformZ}, we have that
\begin{align*}
& \PP(Y^n_{\pi(i)\pi(j)} =  1, Y^n_{\pi(i) \pi(k)}=1, Y^n_{\pi(j)\pi(k)} = 1)\\
=& \int_{x,y,z \in \Omega(n)} \frac{1}{2^{3d} g(t_{n+1})^{3d}} K(\|x - y\|) K(\|x - z\|) K(\|y-z\|) \text{d}z \text{d}y \text{d} x\\
\geq  & \int_{x,y,z \in \Omega(n)} \frac{1}{2^{3d} g(t_{n+1})^{3d}} c^3 I(\|x - y\| \leq \epsilon) I(\|x - z\| \leq \epsilon)  \text{d}z \text{d}y \text{d} x\\
\geq  & \int_{x,y,z \in \Omega_{+}(n)} \frac{1}{2^{3d} g(t_{n+1})^{3d}} c^3 I(\|x - y\| \leq \epsilon) I(\|x - z\| \leq \epsilon)  \text{d}z \text{d}y \text{d} x\\
\geq  & \int_{x,y,z \in \Omega_{+}(n)} \frac{1}{2^{3d} g(t_{n+1})^{3d}} c^3 I(\|y\| \leq\epsilon) I(\|z\| \leq\epsilon)  \text{d}z \text{d}y \text{d} x\\
\propto  &  \frac{1}{g(t_{n+1})^{2d}} \int_{y,z \in \Omega_{+}(n)}  I(\|y\| \leq\epsilon) I(\|z\| \leq \epsilon)  \text{d}z \text{d}y\\
=& \frac{\pi^d \epsilon^{2d}}{g(t_{n+1})^{2d} 2^{2d} \Gamma(d/2 + 1)^2}\\
\propto & g(t_{n+1})^{-2d}.
\end{align*}
Moreover, note that the function
\begin{align}
\int_{y,z \in \Omega(n)} \frac{1}{2^{3d} g(t_n+1)^{3d}} K(\|x - y\|) K(\|x - z\|) \text{d}z \text{d}y
\end{align}
is maximized at $x=0$. Therefore,

\begin{align}
& \PP(Y^n_{\pi(i)\pi(j)} =  1, Y^n_{\pi(i) \pi(k)}=1)\\
=& \int_{x,y,z \in \Omega(n)} \frac{1}{2^{3d} g(t_{n+1})^{3d}} K(\|x - y\|) K(\|x - z\|) \text{d}z \text{d}y \text{d} x\\
\leq & \int_{x,y,z \in \Omega(n)} \frac{1}{2^{3d} g(t_{n+1})^{3d}} K(\|y\|) K(\|z\|) \text{d}z \text{d}y \text{d} x\\
=& \frac{1}{2^{2d} g(t_{n+1})^{2d}} \int_{y,z \in \Omega(n)} K(\|y\|) K(\|z\|) \text{d}z \text{d}y\\
\propto& g(t_{n+1})^{-2d}
\end{align}
by Lemma~\ref{distanceswitch}. Turning to the probability of triadic closure, we thus have that
\begin{align}
& \PP(Y^n_{\pi(i)\pi(j)} =  1 \mid Y^n_{\pi(i) \pi(k)}=1, Y^n_{\pi(j)\pi(k)} = 1)\\
=& \frac{\PP(Y^n_{\pi(i)\pi(j)} =  1, Y^n_{\pi(i) \pi(k)}=1, Y^n_{\pi(j)\pi(k)} = 1) }{\PP(Y^n_{\pi(i)\pi(j)} =  1, Y^n_{\pi(i) \pi(k)}=1)}\\
\geq & \frac{\int_{x,y,z \in \Omega_{+}(n)} \frac{1}{2^{3d} g(t_{n+1})^{3d}} c^3 I(\|y\| < \epsilon) I(\|z\| < \epsilon)  \text{d}z \text{d}y \text{d} x }{\frac{1}{2^{2d} g(t_{n+1})^{2d}} \int_{y,z \in \Omega(n)} K(\|y\|) K(\|z\|) \text{d}z \text{d}y }\\
\propto & \frac{g(t_{n+1})}{g(t_{n+1})} = 1,
\end{align}
thus establishing the result.
\end{proof}

\subsection{Towards a Negative Learnability Result}\label{negatives}

In this section, we establish conditions under which regular LPMs are not learnable.

\begin{theorem}\label{negative}
Consider a regular LPM. Let $n$ denote the number of nodes. Suppose
\begin{align}
\lim_{n \rightarrow \infty } n^2 K\left(\frac{G(n)}{1+c}\right) \rightarrow 0. \label{disconnected}
\end{align}
for some $c > 0$, then this class of LPMs do not have learnable latent positions.
\end{theorem}

\begin{proof}
Recall LeCam's theorem \citep{Tsybakov-intro}, in the form it is used to determine minimax estimation rates: 
\begin{lemma}\label{leCam}
Let $\mathcal{P}$ be a set of distributions parameterized by $\theta \in \Theta$. Let $\hat{\Theta}$ denote the class of possible estimators for $\theta \in \Theta$. For any pair $P_{\theta_1}, P_{\theta_2} \in \mathcal{P}$, 
\begin{align}
\inf_{\hat{\theta}\in \hat{\Theta}} \sup_{\theta \in \Theta} \mathbb{E}_{P_{\theta}}(d(\hat{\theta}, \theta)) \geq \frac{\Delta}{8}\exp{(-KL(P_{\theta_1}, P_{\theta_2}))},
\end{align}
where $\Delta = d(\theta_1, \theta_2)$ for some distance $d(\cdot, \cdot)$, and KL denotes the Kullback-Leibler divergence \citep{Cover-and-Thomas-2nd}.
\end{lemma}
Let $\Theta$ be the set of possible latent positions and $\mathcal{P}$ be the distributions over graphs implied by a regular LPM with link probability function $K$. Without loss of generality, consider the latent space to be $S = \mathbb{R}^1$. We require that the latent positions $Z_1, \ldots, Z_n \in \mathbb{R}$ be such $|Z_i| \leq G(n)$ for some differentiable and non-decreasing function $G(n)$. Suppose $G(n) = (1+c) g(n)$ for some non-decreasing differentiable function $g$, and let $c > 0$ be a small constant.

 To get a decent lower bound for this setting via LeCam, we choose two candidate sets of positions $\theta_1, \theta_2 \in \Theta$ corresponding to probability models that do not differ much in KL-divergence, but have embeddings differing by a non-shrinking amount in $n$. To accomplish this, we exploit the fact that the larger the distance between two nodes, the smaller the change in the connection probability (and thus the KL divergence) due to a small perturbation in the distance. 

Consider $\theta^n_1 \in \mathbb{R}^n$ according to $\theta^n_1 = (0, g(n), 0, g(n), \ldots)$. That is, every odd-indexed latent position is 0, and every even-indexed latent position is $g(n)$. Similarly, define $\theta^n_1 \in \mathbb{R}^n$ according to $\theta^n_1 = (0, G(n), 0, G(n), \ldots)$. Let the distance metric on $\Theta^2$ follow from the definition of learnable latent positions. That is,
\begin{align}
d(\theta^n_1, \theta^n_2) &= \inf_{O \in \mathcal{O}_1, Q \in \mathcal{Q}_{n1}}\frac{\|\theta^n_1T- Q -\theta^n_2\|_F^2}{n}\\
                                         &= \frac{c}{2}.
\end{align}
Here, $\mathcal{O}_1$ and $\mathcal{Q}_{n1}$ capture all possible isometric transformations. Notice that this distance is constant in $n$. However,
\begin{align}
KL(P_{\theta_2^n}, \theta_1^n) &\leq \frac{n^2}{2} K(g(n)) \log\left(\frac{K(g(n))}{K((1 + c)g(n))} \right) \label{ineq}\\
& \rightarrow 0
\end{align}
as $n$ goes to infinity due to the assumption in (\ref{disconnected}). Thus, by Lemma~\ref{leCam}, this class of LPMs is not learnable.

\end{proof}

\begin{table}[h]
\begin{tabular}{|l|l|l|l|l|l|}
\hline
$K(x)=$               & $K'(x)=$                           & $\frac{|K'(x)|}{|x|K(x) \epsilon}=$     & $\frac{x^2|K(x)|}{K'(x)^2}=$      & $\alpha_n^K \sim$ & $\beta_n^K \sim$          \\ \hline
$\frac{1}{1+e^x}$     & $\frac{-e^x}{(1+e^x)^2}$           & $\frac{e^x}{\epsilon x (1+e^x)}$        & $\frac{(1+e^x)^3x^2}{e^{2x}}$     & $\infty$          & $\Theta(e^{G(n)} G(n)^2)$ \\ \hline
$\frac{1}{1+e^{x^2}}$ & $\frac{-2xe^{x^2}}{(1+e^{x^2})^2}$ & $\frac{2e^{x^2}}{\epsilon(1+e^{x^2})} $ & $\frac{(1+e^{x^2})^3}{4e^{2x^2}}$ & $\Theta(1)$       & $\Theta(e^{G(n)^2})$      \\ \hline
$\tau e^{-x^2}$       & $-2x\tau e^{-x^2}$                 & 2                                       & $\frac{e^{x^2}}{4\tau}$           & $\Theta(1)$       & $\Theta(e^{G(n)^2})$      \\ \hline
$\frac{1}{(c+x^2)^q}$ & $\frac{-2qx}{(c+x^2)^{q+1}}$       & $\frac{2q}{c + x^2}$                    & $\frac{(c + x^2)^{q+2}}{4q^2}$    & $\Theta(1)$       & $\Theta(G(n)^{2q+4})$     \\ \hline
\end{tabular}
\caption{Values of $\alpha_n^K$ and $\beta_n^K$ for different choices of link function $K(x)$} \label{differentlinks}
\end{table}

\section*{Acknowledgements}
We are grateful to the members of the CMU Networkshop for feedback on our
results and their presentation, and to conversations with Creagh Briercliffe, David Choi, Emily Fox,
Alden Green, Peter Hoff, Jeannette Janssen, Dmitri Krioukov, and Cristopher Moore. 


\bibliography{locusts}

\begin{thebibliography}{55}
\providecommand{\natexlab}[1]{#1}
\providecommand{\url}[1]{\texttt{#1}}
\expandafter\ifx\csname urlstyle\endcsname\relax
  \providecommand{\doi}[1]{doi: #1}\else
  \providecommand{\doi}{doi: \begingroup \urlstyle{rm}\Url}\fi

\bibitem[Arias-Castro et~al.(2021)Arias-Castro, Channarond, Pelletier, and
  Verzelen]{Arias-Castro-et-al-estimating-latent-distances}
Ery Arias-Castro, Antoine Channarond, Bruno Pelletier, and Nicolas Verzelen.
\newblock On the estimation of latent distances using graph distances.
\newblock \emph{Electronic Journal of Statistics}, 15\penalty0 (1), 2021.

\bibitem[Bollob{\'a}s et~al.(2007)Bollob{\'a}s, Janson, and
  Riordan]{Bollobas-Janson-Riordan-Phase}
B{\'e}la Bollob{\'a}s, Svante Janson, and Oliver Riordan.
\newblock The phase transition in inhomogeneous random graphs.
\newblock \emph{Random Structures and Algorithms}, 31:\penalty0 3--122, 2007.
\newblock \doi{10.1002/rsa.20168}.
\newblock URL \url{http://arxiv.org/abs/math/0504589}.

\bibitem[Borg and Groenen(2005)]{Borg-Groenen-MDS}
Ingwer Borg and Patrick~J.F. Groenen.
\newblock \emph{{Modern Multidimensional Scaling: Theory and Applications}}.
\newblock Springer, 2005.

\bibitem[Borgs et~al.(2014)Borgs, Chayes, Cohn, and
  Zhao]{Borgs-Chayes-Zhao-sparse-graph-convergence}
Christian Borgs, Jennifer~T. Chayes, Henry Cohn, and Yufei Zhao.
\newblock An $l^p$ theory of sparse graph convergence {I}: Limits, sparse
  random graph models, and power law distributions.
\newblock Electronic pre-print, arxiv:1401.2906, 2014.
\newblock URL \url{http://arxiv.org/abs/1401.2906}.

\bibitem[Borgs et~al.(2017)Borgs, Chayes, Cohn, and
  Holden]{Borgs-Chayes-Holden-sparse-exchangeable}
Christian Borgs, Jennifer~T Chayes, Henry Cohn, and Nina Holden.
\newblock Sparse exchangeable graphs and their limits via graphon processes.
\newblock \emph{The Journal of Machine Learning Research}, 18\penalty0
  (1):\penalty0 7740--7810, 2017.

\bibitem[Caron(2012)]{Caron-Bayesian-bipartite}
Fran{\c{c}}ois Caron.
\newblock Bayesian nonparametric models for bipartite graphs.
\newblock In P.~Bartlett, F.~C.~N. Pereira, C.~J.~C. Burges, L{\'e}on Bottou,
  and Kilian~Q. Weinberger, editors, \emph{Advances in Neural Information
  Processing Systems 25 [NIPS 2012]}, pages 2051--2059. Curran Associates,
  2012.
\newblock URL
  \url{https://papers.nips.cc/paper/4837-bayesian-nonparametric-models-for-bipartite-graphs}.

\bibitem[Caron and Fox(2017)]{Caron-Fox-sparse-graphs}
Fran{\c{c}}ois Caron and Emily~B Fox.
\newblock Sparse graphs using exchangeable random measures.
\newblock \emph{Journal of the Royal Statistical Society: Series B (Statistical
  Methodology)}, 79\penalty0 (5):\penalty0 1295--1366, 2017.

\bibitem[Chatterjee(2015)]{Chatterjee-matrix-estimation}
Sourav Chatterjee.
\newblock Matrix estimation by universal singular value thresholding.
\newblock \emph{The Annals of Statistics}, 43\penalty0 (1):\penalty0 177--214,
  2015.

\bibitem[Choi and Wolfe(2011)]{Choi-Wolfe-LPM-learnability}
David~S Choi and Patrick~J Wolfe.
\newblock Learnability of latent position network models.
\newblock In \emph{Statistical Signal Processing Workshop (SSP), 2011 IEEE},
  pages 521--524. IEEE, 2011.

\bibitem[Cover and Thomas(2006)]{Cover-and-Thomas-2nd}
Thomas~M. Cover and Joy~A. Thomas.
\newblock \emph{Elements of Information Theory}.
\newblock John Wiley, New York, second edition, 2006.

\bibitem[Crane and Dempsey(2016)]{Crane-Dempsey-Framework-Network-Modelling}
Harry Crane and Walter Dempsey.
\newblock A framework for statistical network modeling.
\newblock Electronic pre-print, arxiv:1509.08185, 2016.
\newblock URL \url{http://arxiv.org/abs/1509.08185}.

\bibitem[D'Amour and Airoldi(2016)]{DAmour-Airoldi-misspecification}
Alexander D'Amour and Edoardo Airoldi.
\newblock Misspecification, sparsity, and superpopulation inference for sparse
  social networks.
\newblock 2016.
\newblock URL \url{http://www.alexdamour.com/content/damour_jobmkt.pdf}.

\bibitem[Davenport et~al.(2014)Davenport, Plan, Van Den~Berg, and
  Wootters]{Davenport-et-al-1-bit}
Mark~A Davenport, Yaniv Plan, Ewout Van Den~Berg, and Mary Wootters.
\newblock 1-bit matrix completion.
\newblock \emph{Information and Inference: A Journal of the IMA}, 3\penalty0
  (3):\penalty0 189--223, 2014.

\bibitem[Diaconis and Janson(2008)]{Diaconis-Janson-graph-limits}
Persi Diaconis and Svante Janson.
\newblock Graph limits and exchangeable random graphs.
\newblock \emph{Rendiconti di Matematica e delle sue Applicazioni},
  28:\penalty0 33--61, 2008.
\newblock URL \url{http://arxiv.org/abs/0712.2749}.

\bibitem[Diaz et~al.(2020)Diaz, McDiarmid, and
  Mitsche]{Diaz-McDiarmid-Mitsche-geometric-graphs}
Josep Diaz, Colin McDiarmid, and Dieter Mitsche.
\newblock Learning random points from geometric graphs or orderings.
\newblock \emph{Random Structures \& Algorithms}, 57\penalty0 (2):\penalty0
  339--370, 2020.

\bibitem[Gibbs and Su(2002)]{Gibbs-Su-probability-metrics}
Alison~L Gibbs and Francis~Edward Su.
\newblock On choosing and bounding probability metrics.
\newblock \emph{International statistical review}, 70\penalty0 (3):\penalty0
  419--435, 2002.

\bibitem[Gilbert(1961)]{Gilbert-Random-Plane}
Edward~N Gilbert.
\newblock Random plane networks.
\newblock \emph{Journal of the Society for Industrial and Applied Mathematics},
  9\penalty0 (4):\penalty0 533--543, 1961.

\bibitem[Handcock et~al.(2007)Handcock, Raftery, and
  Tantrum]{Hancock-Raftery-Tantrum-latent-cluster-network}
Mark~S Handcock, Adrian~E Raftery, and Jeremy~M Tantrum.
\newblock Model-based clustering for social networks.
\newblock \emph{Journal of the Royal Statistical Society: Series A (Statistics
  in Society)}, 170\penalty0 (2):\penalty0 301--354, 2007.

\bibitem[Herlau et~al.(2016)Herlau, Schmidt, and
  M{\o}rup]{Herlau-Schmidt-Morup}
Tue Herlau, Mikkel~N. Schmidt, and Morten M{\o}rup.
\newblock Completely random measures for modelling block-structured sparse
  networks.
\newblock In D.~D. Lee, M.~Sugiyama, U.~V. Luxburg, I.~Guyon, and R.~Garnett,
  editors, \emph{Advances in Neural Information Processing Systems 29 [NIPS
  2016]}, pages 4260--4268. Curran Associates, 2016.
\newblock URL \url{http://arxiv.org/abs/1507.02925}.

\bibitem[Hoeffding(1963)]{Hoeffding-on-Hoeffding}
Wassily Hoeffding.
\newblock Probability inequalities for sums of bounded random variables.
\newblock \emph{Journal of the American Statistical Association}, 58:\penalty0
  13--30, 1963.
\newblock \doi{10.1080/01621459.1963.10500830}.
\newblock URL \url{http://www.lib.ncsu.edu/resolver/1840.4/2170}.

\bibitem[Hoff(2005)]{Hoff-billinear-dyadic}
Peter~D Hoff.
\newblock Bilinear mixed-effects models for dyadic data.
\newblock \emph{Journal of the american Statistical association}, 100\penalty0
  (469):\penalty0 286--295, 2005.

\bibitem[Hoff et~al.(2002)Hoff, Raftery, and Handcock]{Hoff-Raftery-Handcock}
Peter~D. Hoff, Adrian~E. Raftery, and Mark~S. Handcock.
\newblock Latent space approaches to social network analysis.
\newblock \emph{Journal of the American Statistical Association}, 97:\penalty0
  1090--1098, 2002.
\newblock \doi{10.1198/016214502388618906}.
\newblock URL
  \url{http://www.stat.washington.edu/research/reports/2001/tr399.pdf}.

\bibitem[Horn and Johnson(1990)]{Horn-Johnson-matrix-analysis}
Roger~A Horn and Charles~R Johnson.
\newblock \emph{Matrix analysis}.
\newblock Cambridge university press, 1990.

\bibitem[Kallenberg(2002)]{Kallenberg-mod-prob}
Olav Kallenberg.
\newblock \emph{Foundations of Modern Probability}.
\newblock Springer-Verlag, New York, second edition, 2002.

\bibitem[Kartun-Giles et~al.(2018)Kartun-Giles, Krioukov, Gleeson, Moreno, and
  Bianconi]{Kartun-Giles-et-al-Sparse-Power-Law}
Alexander~P Kartun-Giles, Dmitri Krioukov, James~P Gleeson, Yamir Moreno, and
  Ginestra Bianconi.
\newblock Sparse power-law network model for reliable statistical predictions
  based on sampled data.
\newblock \emph{Entropy}, 20\penalty0 (4):\penalty0 257, 2018.

\bibitem[Kingman(1993)]{Kingman-Poisson}
John Frank~Charles Kingman.
\newblock \emph{Poisson Processes}.
\newblock Oxford University Press, Oxford, 1993.

\bibitem[Krioukov and Ostilli(2013)]{Krioukov-Ostilli-Duality-Networks}
Dmitri Krioukov and Massimo Ostilli.
\newblock Duality between equilibrium and growing networks.
\newblock \emph{Physical Review E}, 88\penalty0 (2):\penalty0 022808, 2013.

\bibitem[Laurent and Massart(2000)]{Laurent-Massart-adaptive}
Beatrice Laurent and Pascal Massart.
\newblock Adaptive estimation of a quadratic functional by model selection.
\newblock \emph{Annals of Statistics}, pages 1302--1338, 2000.

\bibitem[Ledoux and Talagrand(1991)]{Ledoux-Talagrand}
Michel Ledoux and Michel Talagrand.
\newblock \emph{Probability in {Banach} Spaces: Isoperimetry and Processes}.
\newblock Springer-Verlag, Berlin, 1991.
\newblock URL \url{http://www.math.jussieu.fr/~talagran/book.ps.gz}.

\bibitem[Ma and Ma(2017)]{Ma-Ma-Exploration}
Zhuang Ma and Zongming Ma.
\newblock Exploration of large networks via fast and universal latent space
  model fitting.
\newblock \emph{arXiv preprint arXiv:1705.02372}, 2017.

\bibitem[McFarland and Brown(1973)]{McFarland-Brown-social-distance}
David~D. McFarland and Daniel~J. Brown.
\newblock Social distance as a metric: A systematic introduction to smallest
  space analysis.
\newblock In E.~O. Laumann, editor, \emph{Bonds of Pluralism: The Form and
  Substance of Urban Social Networks}, pages 213--253. John Wiley, New York,
  1973.

\bibitem[Meester and Roy(1996)]{Meester-Roy-continuum-percolation}
Ronald Meester and Rahul Roy.
\newblock \emph{Continuum Percolation}.
\newblock Cambridge University Press, Cambridge, 1996.

\bibitem[Newman(2003)]{MEJN-on-network-structure-and-function}
Mark E.~J. Newman.
\newblock The structure and function of complex networks.
\newblock \emph{{SIAM} Review}, 45:\penalty0 167--256, 2003.
\newblock URL \url{http://arxiv.org/abs/cond-mat/0303516}.

\bibitem[Newman(2010)]{MEJN-on-networks}
Mark E.~J. Newman.
\newblock \emph{Networks: An Introduction}.
\newblock Oxford University Press, Oxford, England, 2010.

\bibitem[Orbanz and Roy(2015)]{Orbanz-Roy-Bayesian}
Peter Orbanz and Daniel~M. Roy.
\newblock Bayesian models of graphs, arrays and other exchangeable random
  structures.
\newblock \emph{{IEEE} Transactions on Pattern Analysis and Machine
  Intelligence}, 37:\penalty0 437--461, 2015.
\newblock \doi{10.1109/TPAMI.2014.2334607}.
\newblock URL \url{http://arxiv.org/abs/1312.7857}.

\bibitem[Palla et~al.(2016)Palla, Caron, and Teh]{Palla-Caron-Teh}
Konstantina Palla, Fran{\c{c}}ois Caron, and Yee~Whye Teh.
\newblock Bayesian nonparametrics for sparse dynamic networks.
\newblock Electronic pre-print, arXiv:1607.01624, 2016.
\newblock URL \url{http://arxiv.org/abs/1607.01624}.

\bibitem[Penrose(2003)]{Penrose-random-geometric-graphs}
Mathew Penrose.
\newblock \emph{Random Geometric Graphs}.
\newblock Oxford University Press, Oxford, 2003.

\bibitem[Penrose(1991)]{Penrose-on-continuum-percolation}
Mathew~D. Penrose.
\newblock On a continuum percolation model.
\newblock \emph{Advances in Applied Probability}, 23:\penalty0 536--556, 1991.
\newblock \doi{10.1017/S0001867800023727}.

\bibitem[Rastelli(2018)]{Rastelli-sparse-weighted-networks}
Riccardo Rastelli.
\newblock The sparse latent position model for nonnegative weighted networks.
\newblock \emph{arXiv preprint arXiv:1808.09262}, 2018.

\bibitem[Rastelli et~al.(2016)Rastelli, Friel, and
  Raftery]{Rastelli-latent-variable-network}
Riccardo Rastelli, Nial Friel, and Adrian~E Raftery.
\newblock Properties of latent variable network models.
\newblock \emph{Network Science}, 4\penalty0 (4):\penalty0 407--432, 2016.

\bibitem[Ravikumar et~al.(2011)Ravikumar, Wainwright, Raskutti, Yu,
  et~al.]{Ravikumar-et-al-covariance-estimation}
Pradeep Ravikumar, Martin~J Wainwright, Garvesh Raskutti, Bin Yu, et~al.
\newblock High-dimensional covariance estimation by minimizing
  $\ell_1$-penalized log-determinant divergence.
\newblock \emph{Electronic Journal of Statistics}, 5:\penalty0 935--980, 2011.

\bibitem[Rocha et~al.(2018)Rocha, Janssen, and
  Kalyaniwalla]{Rocha-Janssen-Kalyaniwalla-linear-graphs}
Israel Rocha, Jeannette Janssen, and Nauzer Kalyaniwalla.
\newblock Recovering the structure of random linear graphs.
\newblock \emph{Linear Algebra and its Applications}, 557:\penalty0 234--264,
  2018.

\bibitem[Schweinberger et~al.(2017)Schweinberger, Krivitsky, and
  Butts]{Schweinberger-et-al-Finite-Population-Graphs}
Michael Schweinberger, Pavel~N Krivitsky, and Carter~T Butts.
\newblock Foundations of finite-, super-, and infinite-population random graph
  inference.
\newblock Electronic pre-print, arxiv:1707.04800, 2017.
\newblock URL \url{http://arxiv.org/abs/1707.04800}.

\bibitem[Seginer(2000)]{Seginer-expected-norm-matrices}
Yoav Seginer.
\newblock The expected norm of random matrices.
\newblock \emph{Combinatorics, Probability and Computing}, 9:\penalty0
  149--166, 2000.

\bibitem[Shalizi and Asta(2017)]{CRS-Asta-consistency-of-embedding}
Cosma~Rohilla Shalizi and Dena Asta.
\newblock Consistency of maximum likelihood embedding for continuous
  latent-space network models.
\newblock Submitted, 2017.

\bibitem[Shalizi and Rinaldo(2013)]{your-favorite-ergm-sucks}
Cosma~Rohilla Shalizi and Alessandro Rinaldo.
\newblock Consistency under sampling of exponential random graph models.
\newblock \emph{Annals of Statistics}, 41:\penalty0 508--535, 2013.
\newblock \doi{10.1214/12-AOS1044}.
\newblock URL \url{http://arxiv.org/abs/1111.3054}.

\bibitem[Snijders(2010)]{Snijders-marginalization-for-ERGMs}
Tom A.~B. Snijders.
\newblock Conditional marginalization for exponential random graph models.
\newblock \emph{Journal of Mathematical Sociology}, 34:\penalty0 239--252,
  2010.
\newblock \doi{10.1080/0022250X.2010.485707}.
\newblock URL \url{http://irs.ub.rug.nl/dbi/50291029d2}.

\bibitem[Sorokin(1927)]{Sorokin-social-mobility}
Pitirim~Aleksandrovich Sorokin.
\newblock \emph{Social Mobility}.
\newblock Harper and Brothers, New York, 1927.

\bibitem[Sussman et~al.(2014)Sussman, Tang, and
  Priebe]{Sussman-Tang-Priebe-consistent}
Daniel~L Sussman, Minh Tang, and Carey~E Priebe.
\newblock Consistent latent position estimation and vertex classification for
  random dot product graphs.
\newblock \emph{IEEE Transactions on Pattern Analysis and Machine
  Intelligence}, 36:\penalty0 48--57, 2014.

\bibitem[Todeschini et~al.(2020)Todeschini, Miscouridou, and
  Caron]{Todeschini-Miscouridou-Caron}
Adrien Todeschini, Xenia Miscouridou, and Fran{\c{c}}ois Caron.
\newblock Exchangeable random measures for sparse and modular graphs with
  overlapping communities.
\newblock \emph{Journal of the Royal Statistical Society Series B: Statistical
  Methodology}, 82\penalty0 (2), 2020.

\bibitem[Tsybakov(2008)]{Tsybakov-intro}
Alexandre~B. Tsybakov.
\newblock \emph{Introduction to Nonparametric Estimation}.
\newblock Springer Verlag, New York, 2008.
\newblock Translated by Vladimir Zaiats.

\bibitem[Veitch and Roy(2015)]{Veitch-Roy-graphex-class}
Victor Veitch and Daniel~M. Roy.
\newblock The class of random graphs arising from exchangeable random measures.
\newblock Electronic pre-print, arXiv:1512.03099, 2015.
\newblock URL \url{http://arxiv.org/abs/1512.03099}.

\bibitem[Wasserman and Faust(1994)]{Wasserman-Faust}
Stanley Wasserman and Katherine Faust.
\newblock \emph{Social Network Analysis: Methods and Applications}.
\newblock Cambridge University Press, Cambridge, England, 1994.

\bibitem[Xu(2017)]{Xu-rates-convergence-graphon}
Jiaming Xu.
\newblock Rates of convergence of spectral methods for graphon estimation.
\newblock \emph{arXiv preprint arXiv:1709.03183}, 2017.

\bibitem[Yu et~al.(2014)Yu, Wang, and Samworth]{Yu-Wang-Samworth-davis-kahan}
Yi~Yu, Tengyao Wang, and Richard~J Samworth.
\newblock A useful variant of the {D}avis--{K}ahan theorem for statisticians.
\newblock \emph{Biometrika}, 102\penalty0 (2):\penalty0 315--323, 2014.

\end{thebibliography}
\bibliographystyle{plainnat}

\end{document}